\documentclass[11pt]{amsart}
\usepackage[margin=1.2in]{geometry}
\usepackage{amsmath}
\usepackage{color}
\usepackage{amsthm}
 \usepackage{lmodern}
\usepackage{amsfonts}
\usepackage{amssymb}
\usepackage{amscd}%%% 
\numberwithin{equation}{section}
\newtheorem{theorem}{Theorem}[section]
\newtheorem{cor}[theorem]{Corollary}
\newtheorem{proposition}[theorem]{Proposition}
\newtheorem{lemma}[theorem]{Lemma}
\newtheorem{prop}[theorem]{Proposition}
\theoremstyle{definition}
\newtheorem{definition}[theorem]{Definition}
\newtheorem{remark}[theorem]{Remark}
\newtheorem{example}[theorem]{Example}

\newcommand\q{\mathfrak q}

\newcommand\I{\Cal I}

\newcommand\half{\tfrac{1}{2}}

\newcommand\be{\beta}

\newcommand\g{\mathfrak g}

\newcommand\ga{\widehat{\mathfrak g}}
\newcommand\h{\mathfrak h}
\newcommand\ha{\widehat{\mathfrak h}}

\newcommand\n{\mathfrak n}

\newcommand\D{\Delta}
\renewcommand\l{\lambda}
\newcommand\Dp{\Delta^+}

\newcommand\Da{\widehat\Delta}
\newcommand\Pia{\widehat\Pi}
\newcommand\Dap{\widehat\Delta^+}
\newcommand\Wa{\widehat{W}}
\renewcommand\d{\delta}

\renewcommand\t{\mathfrak t}
\renewcommand\a{\alpha}
\renewcommand\aa{\mathfrak a}

\newcommand{\Z}{\mathbb Z}
\renewcommand\th{\theta}

\newcommand\nat{\mathbb N}
\newcommand\ganz{\mathbb Z}

\newcommand\s{\sigma}
\renewcommand\L{\Lambda}
\renewcommand\aa{\mathfrak a}

\newcommand\e{\epsilon}
\newcommand\C{\mathbb C}
\newcommand\R{\mathbb R}

\renewcommand\I{\sqrt{-1}}

\renewcommand\ha{\widehat{\mathfrak h}}

\newcommand{\fg}{\mathfrak{g}}

\newcommand\p{\mathfrak p}

\newcommand{\rank}{{\rm rank}}

\newcommand{\ZZ}{\mathbb{Z}}

\newcommand{\sdim}{\text{\rm sdim}}

\newcommand{\vac}{{\bf 1}}
\newcommand{\bea}{\begin{eqnarray}}
\newcommand{\eea}{\end{eqnarray}}
\begin{document}
\title{Unitarity of minimal $W$--algebras and their representations I}
\author[Victor~G. Kac, Pierluigi M\"oseneder Frajria,  Paolo  Papi]{Victor~G. Kac\\Pierluigi M\"oseneder Frajria\\Paolo  Papi}

\begin{abstract} We begin a systematic study of unitary  representations of minimal $W$--algebras. In particular, we classify unitary minimal $W$--algebras and make substantial progress in classification of their unitary irreducible highest weight modules.  We also compute the characters of these modules.\end{abstract}
%\date{\today}
\maketitle

\tableofcontents

\section{Introduction}
In the present paper we study unitarity of minimal $W$--algebras and of their representations. 
Minimal $W$--algebras  are the simplest conformal vertex algebras  among the simple vertex algebras
$W_k(\g,x,f),$ constructed in \cite{KRW}, \cite{KW1}, associated to a datum $(\g,x,f)$ and $k\in\R$. Here $\g=\g_{\bar 0}\oplus \g_{\bar 1}$ is a basic  Lie superalgebra, i.e. $\g$ is simple, its even part $\g_{\bar 0}$ is a reductive Lie algebra and $\g$ carries an even invariant non-degenerate supersymmetric bilinear form $(.|.)$, $x$ is an $ad$--diagonalizable element of $\g_{\bar 0}$ with eigenvalues in $\tfrac{1}{2}\Z$, $f\in\g_{\bar 0}$ is such that $[x,f]=-f$ and the eigenvalues of $ad\,x$ on the centralizer $\g^f$ of $f$ in $\g$ are non-positive, and  $k\ne -h^\vee$, where $h^\vee$ is the dual Coxeter number of $\g$. The most important examples are provided by $x$ and $f$ to be part of an $sl_2$ triple $\{e,x,f\}$, where $[x,e]=e, [x,f]=-f, [e,f]=x$. In this case $(\g,x,f)$ is called a {\it Dynkin datum}.
Recall that $W_k(\g,x,f)$ is the unique simple quotient  of the universal $W$-algebra, denoted by $W^k(\g,x,f),$ which is  freely strongly generated by elements labeled by a basis of the centralizer of  $f$ in $\g$ \cite{KW1}.\par
We proved in \cite[Lemma 7.3]{KMP}  that if $\phi$ is a conjugate linear involution of $\g$ such that
\begin{equation}\label{primas} \phi(x)=x,\quad \phi(f)=f\text{ and }\overline{(\phi(a)|\phi(b))}=(a|b),\,a,b\in\g,\end{equation} then $\phi$ induces a conjugate linear involution of the vertex algebra $W^k(\g,x,f)$, which descends to $W_k(\g,x,f)$.\par

We also proved in \cite[Proposition 7.4]{KMP} that if $\phi$ is a conjugate linear involution of $W_k(\g,x,f)$, this vertex algebra   carries a non-zero $\phi$--invariant Hermitian form $H(\cdot,\cdot)$ for all $k\ne - h^\vee$ if and only if  $(\g,x,f)$ is  a Dynkin datum; moreover, such $H$  is unique, up to a real constant factor, and we normalize it by the condition $H(\vac,\vac)=1$. A module $M$ for a vertex algebra $V$ is called {\it unitary} if  there is a conjugate linear involution $\phi$ of $V$ such that there is a positive definite $\phi$--invariant Hermitian form on $M$. The vertex algebra $V$ is called unitary if the adjoint module is.
% corresponding $\phi$--invariant Hermitian form $H$ is positive definite.
\par
For some levels $k$ the vertex algebra $W_k(\g,x,f)$ is trivial, i.e. isomorphic to $\C$; then it is trivially unitary. Another easy case is when  $W_k(\g,x,f)$ ``collapses'' to the affine part.
In both cases we will say that $k$ is {\it collapsing level}.\par
In the case of a Dynkin datum let $\g^\natural$ be the centralizer of the $sl_2$ subalgebra $\mathfrak s=span\, \{e,x,f\}$ in $\g_{\bar 0}$; it is a reductive subalgebra. 
 If $\phi$ satisfies the first two conditions in \eqref{primas}, it fixes 
$e,x,f$, hence $\phi(\g^\natural)=\g^\natural$. It is easy to see that unitarity of $W_k(\g,x,f)$ implies, when $k$ is not collapsing,  that $\phi_{|[\g^\natural,\g^\natural]}$ is a compact involution.\par
In the present paper we consider only {\it minimal} data $(\g,x,f)$, defined by the property that for the $ad\,x$--gradation $\g=\bigoplus\limits_{j\in\tfrac{1}{2}\mathbb Z}\g_j$ one has 
\begin{equation}\label{seconda}\g_j=0\text{ \ if $|j|>1$, and } \g_{-1}=\C f.\end{equation}
In this case $(\g,x,f)$ is automatically a Dynkin datum. The corresponding $W$--algebra is called {\it minimal}.
The element $f\in\g$ is a root vector attached to a root $-\theta$ of $\g$, and we shall normalize the invariant bilinear form on $\g$ by the usual condition 
$(\theta|\theta)=2$, which is equivalent to  $(x|x)=\tfrac{1}{2}$.
Recall that the dual Coxeter number $h^\vee$ of $\g$  is half of the eigenvalue of its Casimir element of $\g$, attached to the bilinear form $(.|.)$.
We shall denote by   $W_k^{\min}(\g)$ the minimal $W$--algebra, corresponding to $\g$ and $k\ne -h^\vee$, and by $W^k_{\min}(\g)$  the corresponding universal $W$-algebra.
\par
We proved in \cite[Proposition 7.9]{KMP} that, if  $W^{\min}_k(\g)$ is  unitary and $k$ is not a collapsing level, then the parity of $\g$ is compatible with the $ad\,x$--gradation, i.e. the parity of the whole subspace $\g_j$ is $2j\mod 2$.

%We proved in \cite{KMP} that if the $W$-algebra $W_k(\g,x,f)$ is non-trivial  unitary, and $k$ is not a collapsing level, then, in addition, the  parity of $\g$  is compatible with the $\frac{1}{2}\Z$-gradation, i.e. the parity of the whole $ad\,x$-eigenspace $\g_j$ is $2j\mod 2$. Thus, in order to classify unitary $W$-algebras, we need to consider only those of them, which
%are associated to a Dynkin datum, and such that the corresponding  $ad\,x$-gradation is compatible with the parity of $\g$.
%\par
%In this paper we consider only the {\it minimal}  such Dynkin datum, namely, we require that 
%$$\g_j=0\text{ \ if $|j|>1$, and } \g_{-1}=\C f.$$ 
\par
It follows from \cite{KRW}, \cite{KW1} that for each basic simple Lie superalgebra $\g$ there is at most one minimal Dynkin datum, compatible with parity, and  the  complete list of the  $\g$ which admit such a datum is as follows:
\begin{align}\label{ssuper}
\begin{split}
sl(2|m)\text{ for $m\ge3$},\quad psl(2|2),\quad  spo(2|m)\text{ for $m\geq 0$,}\\
osp(4|m)\text{ for $m>2$ even},\quad D(2,1;a) \text{ for }a\in \C,\quad  F(4),\quad G(3).
\end{split}
\end{align}
The even part $\g_{\bar 0}$ of $\g$ in this case is isomorphic to the direct sum of the  reductive Lie algebra $\g^\natural$ and $\mathfrak s\cong sl_2$.

%{\color{red}Recall that a module M for a conformal vertex algebra $V$ is said to be unitary if there is a conjugate linear involution $\psi$ of $V$ and a $\psi$--invariant positive definite Hermitian form on $M$. The vertex algebra is said to be unitary if  It turns out that, by Proposition \ref{converse}, a minimal $W$-algebra $W_k^{\min}(\g)$ with non-collapsing $k$ can be unitary only if 
% the conjugate linear involution is unduced by a conjugate linear involution  $\phi$ on $\g$  which is {\it almost compact}, 
%according to the following definition.}
One of our conjectures (see Conjecture 4 in Section 8) states  that any unitary $W^k_{\min}(\g)$-module descends to $W_k^{\min}(\g)$. In fact, it is tempting to conjecture that for  any conformal vertex algebra $V$ any  unitary $V$-module  descends to the simple quotient of $V$.\par
It turns out (cf. Proposition \ref{converse}) that  a conjugate linear  involution of the universal  minimal $W$--algebra $W^k_{\min}(\g)$ at non-collapsing  level $k$ is necessarily  induced by a  conjugate linear  involution $\phi$ of $\g$. Moreover,  by Proposition
\ref{casi},
if $W^k_{\min}(\g)$ admits a unitary highest weight module and $k$ is not collapsing, then $\g^\natural$ has to  be semisimple. As explained above, the involution $\phi$ of $\g$ must be {\it almost compact}, 
according to the following definition.
\begin{definition}\label{ac} A conjugate linear involution $\phi$ on $\g$ is called almost compact if
\begin{enumerate}
\item[(i)] $\phi$ fixes $e,x,f$;
\item[(ii)] $\phi$ is a compact conjugate linear involution of $\g^\natural$.
\end{enumerate}
\end{definition}
\noindent Indeed (i) is equivalent to the first two requirements in \eqref{primas}, and the third requirement in \eqref{primas} follows from Lemma \ref{31} in Section 3.

So, in order to study unitarity of highest weight modules, it is not restrictive to assume that the conjugate linear involution  of $W^k_{\min}(\g)$ is induced by an almost compact conjugate linear involution of $\g$.

We prove in Sections \ref{2} and \ref{4} that an almost compact conjugate linear involution $\phi$ exists for all $\g$ from the list \eqref{ssuper}, except that $a$ must lie in $\R$ in case of  $D(2,1;a)$, and is 
essentially unique. 
\par
It was shown in \cite{KW1} that the central charge of $W_k^{\min}(\g)$ equals
\begin{equation}\label{cc} c(k)=\frac{k\,d}{k+h^\vee}-6k+h^\vee-4,\text{ where $d=\sdim\g$.}\end{equation}
Here is another useful way to write this formula:
\begin{equation}\label{ccc} c(k)=7h^\vee+d-4-12 \sqrt{\phantom{x}} - 6\frac{(k+h^\vee-\sqrt{\phantom{x}})^2}{k+h^\vee},\text{ where $\sqrt{\phantom{x}}=\sqrt{\frac{d\,h^\vee}{6}}$.}\end{equation}

Recall that the most important superconformal algebras in conformal field theory are the simple minimal $W$--algebras or are obtained from them by a simple modification:
\begin{enumerate}
\item[(a)] $W^{\min}_k(spo(2|N))$ is the Virasoro vertex algebra for $N=0$, the Neveu-Schwarz vertex algebra for $N=1$, the $N=2$ vertex algebra for $N=2$, and becomes the $N=3$ vertex algebra 
after tensoring with one fermion; it is the Bershadsky-Knizhnik algebra for $N>3$;
\item[(b)]  $W^{\min}_k(psl(2|2))$ is the $N=4$ vertex algebra;
\item[(c)]  $W^{\min}_k(D(2,1;a))$ tensored with four fermions and one boson is the big $N=4$ vertex algebra.
%\item[(d)]  $W^{\min}_k(F(4))$ and $W^{\min}_k(G(3))$ are the Shatashvili-Vafa  vertex algebras.
\end{enumerate}\par
The unitary Virasoro ($N=0$), Neveu-Schwarz ($N=1$) and  $N=2$ simple vertex algebras, along with their irreducible unitary modules,  were classified in the mid 80s.  Up to isomorphism,  these vertex algebras depend only on the central charge $c(k)$, given by  \eqref{cc}. Putting $k=\frac{1}{p}-1$ in \eqref{ccc} in all three cases, we obtain 
\begin{align}
\label{1}c(k)&=1-\frac{6}{p(p+1)} \quad \text{for   Virasoro vertex algebra,}\\
\label{22}c(k)&=\frac{3}{2}\left(1-\frac{8}{p(p+2)}\right)\quad \text{for   Neveu-Schwarz vertex algebra,}\\
\label{3}c(k)&=3\left(1-\frac{2}{p}\right)\quad \text{for   $N=2$ vertex algebra.}
\end{align}
The following theorem is a result of several papers, published in the 80s in physics and mathematics literature, see e.g. \cite{ET3} for references.
\begin{theorem}The complete list of unitary $N=0,1,$ and $2$ vertex algebras is as follows: either $c(k)$ is given by \eqref{1}, \eqref{22}, or \eqref{3}, respectively, for $p\in \mathbb Z_{\ge 2},$  or  $c(k)\ge 1, \frac{3}{2}$ or $3$, respectively. 
\end{theorem}
The above three cases cover all minimal $W$--algebras, associated with $\g$, such that the  eigenspace $\g_0$ of $ad\,x$ is abelian. Thus, we may assume that $\g_0$ is not abelian. 

In order to study unitarity of the simple minimal $W$--algebra $W_k^{\min}(\g)$, one needs to consider the more general framework of representation theory of universal minimal $W$--algebras $W^k_{\min}(\g)$. Of course, unitarity of $W^k_{\min}(\g)$ is equivalent to that of $W_k^{\min}(\g)$. 
It is therefore natural to study unitarity of irreducible $W^k_{\min}(\g)$--modules. For that purpose, we  take, 
in Section \ref{freeb}, a long detour  to  develop a general theory of invariant Hermitian forms on modules over the vertex algebra of free bosons, which will be eventually applied to our main object of interest. As a byproduct  we obtain a field theoretic version of  the {\it Fairlie construction}, which yields
explicit models of unitary representations of the Virasoro algebra for certain values of the highest weight (cf. \cite[3.4]{KR},  Example \ref{FC}).
% Let $W^k_{\min}(\g)$ be the universal minimal $W$--algebra of level $k\ne- h^\vee$ for $\g$, so that $W_k^{\min}(\g)$ is its unique simple quotient. Obviously, an irreducible $W^k_{\min}(\g)$-module is a module over $W_k^{\min}(\g)$, and unitarity of  $W^k_{\min}(\g)$  is equivalent to that of   $W_k^{\min}(\g)$.

We  consider in Section \ref{8}   the free field realization $\Psi: W^k_{\min}(\g)\to \mathcal V^k=V^{k+h^\vee}(\C x)\otimes V^{\alpha_k}(\g^\natural)\otimes F(\g_{1/2})$  introduced  in \cite{KW1} (here $V^\gamma(\aa)$ denotes the universal affine vertex algebra associated to the Lie algebra $\aa$ and to a 2-cocycle $\gamma$,   $\a_k$ is the 2--cocycle defined in \eqref{eq:5.16}, and $F(\g_{1/2})$ is the fermionic vertex algebra 
``attached'' to $\g_{1/2}$). Let $M(\mu)$ be the Verma module of highest weight $\mu\in\C$ for the bosonic vertex algebra $V^{k+h^\vee}(\C x)$ and  consider the $\mathcal V^k$--module
 $N(\mu)=M(\mu)\otimes V^{\a_k}(\g^\natural)\otimes F(\g_{1/2})$. Applying to $N(\mu)$  results from Section \ref{freeb}, we obtain in Proposition \ref{FCW} a generalization of the Fairlie construction to universal minimal $W$--algebras.

The  conformal vertex algebras $(W^k_{\min}(\g), L)$ and $(\mathcal V^k, \widehat L(0))$  (see \eqref{nuovovir}) both admit 
Hermitian invariant forms $H(\cdot, \cdot)_W$ and  $H(\cdot, \cdot)_{free}$, respectively. Unfortunately, the embedding 
$\Psi$ is not conformal, i.e., $\Psi(L)\ne \widehat L(0)$, in particular $\Psi$ is not an isometry (which was erroneously claimed in \cite{wrong}). So, though the vertex algebra
$\mathcal V^k$ is unitary, this  does not imply the
unitarity of $W^k_{\min}(\g)$. A few explicit computations suggest the following conjecture, which we were unable to prove.
\vskip5pt
\noindent{\sl  {\bf Conjecture 1.} For each $w\in W^k_{\min}(\g)$, $H(w,w)_W\geq H(\Psi(w),\Psi(w))_{free}.$
In  particular if $\mathcal V^k$ is unitary, then  $ W^k_{\min}(\g)$ is unitary.}
\vskip5pt
 We start the study of unitary modules over minimal $W$-algebras  in Section 8  by 
introducing    the irreducible  highest weight $W^k_{\min}(\g)$--modules $L^W(\nu,\ell_0)$  with highest weight $(\nu,\ell_0)$,  where $\nu$ is a real  weight of $\g^\natural$ and $\ell_0\in\mathbb R$ is the minimal eigenvalue of $L_0$. We prove that $L^W(\nu,\ell_0)$ admits a  $\phi$--invariant nondegenerate Hermitian form (unique up to normalization), see Lemma \ref{ef}.
%,  and that there exists a unitary highest weight module over  $W^k_{\min}(\g)$ if and only if  the levels  $M_i(k)$ of the affine algebra $\widehat \g_i^\natural$ in $W^k_{\min}(\g)$, given in Table 2 (in Section 7), where $\g_i^\natural$ are simple components of $\g^\natural$,
%(defined in \eqref{Mi}) 
% are non-negative integers: see  Proposition \ref{necessary}. 
In Section \ref{Necessary} we also determine necessary conditions for the unitarity of $L^W(\nu,\ell_0)$.   Part of the necessary conditions is  displayed in Proposition \ref{l0nec}. They say that unitarity of $L^W(\nu, \ell_0)$ implies that the levels $M_i(k)$ of the affine Lie algebras $\widehat \g_i^\natural$ in $W^k_{\min}(\g)$ (given in Table 2, Section \ref{7}), 
where $\g_i^\natural$ are the simple components of $\g^\natural$, are non-negative integers, $\nu$ is dominant integral of levels $M_i(k)$, and the inequality \eqref{eh} below holds.
Proposition \ref{boundary} provides a further necessary condition, which says  that \eqref{eh} must be an equality when $\nu$ is an ``extremal'' weight. See Theorem \ref{nec} (1) below for a precise statement.\par
% Proposition \ref{boundary} provides further conditions which hold when $\nu$ has an ``extremality property''.  In Proposition \ref{sufficient} we prove a partial converse  result by giving sufficient conditions for unitarity.
 %These conditions, together with  some useful data, 
%are displayed  in each case in Subsections \ref{1111}-\ref{fff}.  \par 
%Summing up, we obtain the following result 
In Section 10, using the generalization of the Fairlie construction, developed in Section 9, we prove a partial converse result: if $M_i(k)+\chi_i\in\mathbb Z_+$, where $\chi_i$ are negative integers, displayed in Table 2, and $\nu$ is  dominant integral weight for $\g^\natural$ which is not extremal, then the $W^k_{\min}(\g)$--module $L^W(\nu,\ell_0)$ is unitary for $l_0$ sufficiently large, see Proposition \ref{sufficient}.\par
In Section 11 we prove our central Theorem \ref{u??}, which claims that actually Proposition \ref{sufficient} holds for $l_0$ satisfying the inequality \eqref{eh}, provided that $\nu$ is not extremal.
This is established by  the following construction. Let $\ga$ be the affinization of $\g$. We introduce  in \eqref{modulo} a highest weight module $\overline M(\widehat \nu_h)$ over $\ga$, whose  
highest weight $\widehat\nu_h$ depends on $h\in \C$, with the following two properties
\begin{enumerate}
\item $\overline M(\widehat \nu_h)$ is irreducible, except possibly for an explicit set $J$ of values of $h$.
\item For  the quantum Hamiltonian reduction functor $H_0$, the $W^k_{\min}(\g)$-module $H_0(\overline M(\widehat\nu))$ admits a Hermitian form, depending polynomially on $h$.
\end{enumerate}
Using the irreducibility theorem   by Arakawa \cite{Araduke}, we deduce that  $H_0(\overline M(\widehat \nu_h))=L^W(\nu, \ell(h))$ for $h\notin J$, where $\ell(h)$ is defined by \eqref{p}.
%We  show that the determinant $P(l_0)$ of the Hermitian form on  $L^W(\nu,\ell_0)$ is positive if the inequality \eqref{eh} is strict.
It turns out that, miraculously, 
if $h\in J$, then $\ell(h)$ does not satisfy \eqref{eh}.  Moreover   $L^W(\nu,\ell_0)$ is  unitary for $l_0\gg 0$.
%Moreover, due to another miracle,  for each $h'\leq B(k,\nu)$ there exists $l_0$, satisfying \eqref{eh}, such that $h'=q(l_0)$ so that $L^W(\nu,h')$ is  unitary. 
By continuity, the determinant  of the Hermitian form on  $L^W(\nu,\ell_0)$ is positive if the inequality \eqref{eh} holds. See Theorem \ref{nec} (2) below for a precise statement.

Let us state our main results. First of all, if $\g=sl(2|m)$ with $m\ge 3$ or $osp(4|m)$ with $m\ge 2$ even, then none of the $W^k_{\min}(\g)$--modules $L^W(\nu, \ell_0)$ are unitary for a non-collapsing level $k$. For the remaining $\g$ from the list \eqref{ssuper} the Lie algebra $\g^\natural$ is semisimple (actually simple, except for $\g=D(2,1;a)$, when $\g^\natural =sl_2\oplus sl_2$). Let $
 \theta_i^\vee$ be the coroots of the highest roots $\theta_i$ of the simple components $\g^\natural_i$ of $\g^\natural$. Let $2\rho^\natural$ be the sum of positive roots of $\g^\natural $, and let $\xi$ be a  highest weight
of the $\g^\natural$--module $\g_{-1/2}$ (this module is irreducible, except for  $\g=psl(2|2)$ when it is $\C^2\oplus \C^2$). 
%Let $\chi_i$ be the negative integers, displayed in  Table 2. 
Let $\nu$ be a dominant integral weight for $\g^\natural$ and $l_0\in\mathbb R$.
We prove the following theorem.
\begin{theorem}\label{nec} Let  $L^W(\nu,\ell_0)$ be  an irreducible highest weight $W^k_{\min}(\g)$--module for  $\g= psl(2|2),$ $ spo(2|m)$ with $m\ge 3, \, D(2,1;a),\, F(4)$ or $G(3)$.
\begin{enumerate}
\item This module can be unitary  only if the following conditions hold:
\begin{enumerate}
\item $M_i(k)$ are non-negative integers,
\item $\nu(\theta_i^\vee) \leq M_i(k)$ for all $i$,
\item
\begin{equation}\label{eh}
l_0\ge \frac{(\nu|\nu+2\rho^\natural)}{2(k+h^\vee)}+\frac{(\xi|\nu)}{k+h^\vee}((\xi|\nu)-k-1),
\end{equation}
 and equality holds in \eqref{eh} if $\nu(\theta^\vee_i)>M_i(k)+\chi_i$ for $i=1$ or $2$.
 \end{enumerate}
\item This module is unitary if the  following conditions hold:
\begin{enumerate}
\item $M_i(k)+\chi_i\in\mathbb Z_+$ for all $i$,
\item  $\nu(\theta^\vee_i) \leq M_i(k)+\chi_i$ for all $i$ (i.e. $\nu$ is not extremal),
\item  inequality \eqref{eh} holds.
\end{enumerate}\end{enumerate}
\end{theorem}

\vskip10pt
\noindent{\sl  {\bf Conjecture 2.} The modules $L^W(\nu,\ell_0)$ are unitary if $\nu$ is extremal and $l_0= R.H.S.\text{ of }\eqref{eh}$. In other words, the necessary conditions of unitarity in Theorem  \ref{nec} (1) are sufficient.}
\vskip10pt
We were able to prove this conjecture only for $\g=psl(2|2)$ and $spo(2|3)$, obtaining thereby a complete classification of unitary simple
highest weight $W^k_{\min}(\g)$--modules in these two cases. Note that 
papers  \cite{ET1}, \cite{ET2}  and \cite{M}  respectively claim  (without proof) these results.\par
Since $\nu=0$ is extremal iff $k$ is collapsing, we obtain   the following   complete classification of minimal simple unitary $W$--algebras:\par\noindent
\begin{theorem} The  simple minimal $W$--algebra $W_{-k}^{\min}(\g)$ with $k\ne h^\vee$ and $\g_0$ non-abelian is non-trivial unitary if and only if 
\begin{enumerate}
\item $\g=sl(2|m),\ m\ge 3$, $k=1$ (in this case the $W$--algebra is a free boson);
\item  $\g=psl(2|2)$, $k\in\mathbb N +1$;
\item  $\g=spo(2|3)$, $k\in\frac{1}{4}(\mathbb N+2)$;
\item $\g=spo(2|m)$, $m>4$, $k\in\frac{1}{2}(\mathbb N+1)$;
%\item   $\g=D(2,1;--\frac{m}{m+n})$, $k=\frac{mn}{m+n}\mathbb N$, \text{ where $m,n\in\mathbb N$ are coprime,  $k\neq \frac{1}{2}$};
\item   $\g=D(2,1;\frac{m}{n})$, $k\in\frac{mn}{m+n}\mathbb N$, \text{ where $m,n\in\mathbb N$ are coprime,  $k\neq \frac{1}{2}$};
\item   $\g=F(4)$, $k\in\frac{2}{3}(\mathbb N+1)$;
\item  $\g=G(3)$, $k\in\frac{3}{4}(\mathbb N+1)$.
\end{enumerate}
\end{theorem}
This result, along with all known results on unitarity of vertex algebras,  leads to the following general conjecture.
\vskip5pt
\noindent{\sl  {\bf Conjecture 3.}  A CFT type vertex operator algebra admitting a invariant Hermitian  form and  having a unitary module is unitary.}
\vskip5pt
%This conjecture  holds for affine vertex algebras, Virasoro and  $N=1,2,3,4$ superconformal algebras.
%%The case (1) is just the free boson vertex algebra. 
%The results (2) and  (3) of Theorem \ref{12} are consistent with the results 
%of \cite{ET1} and \cite{M} respectively.\par
%By Theorem \ref{main} (b) the vertex algebra $W_k^{\min}(\g)$, with $k\neq --h^\vee$ and 
%non--collapsing, has a
%non--trivial unitary module only when $W_k^{\min}(\g)$ is a unitary vertex 
%algebra. 

%We are planning to classify its unitary modules and compute their characters in a subsequent publication.\par

In the final Section \ref{14} we provide character formulas for all unitary $W^k_{\min}\!(\g)$-modules $L^W\!(\nu,\ell_0)$, which are obtained  by applying the quantum Hamiltonian reduction to the corresponding irreducible highest weight modules over the affinization $\ga$ of $\g$. There are two cases to consider. In the first case, called massive (or typical), when  inequality \eqref{eh} is strict, this character formula is easy to prove (see the proof of Proposition \ref{irr}), which leads to the character formula \eqref{ff1}.  In the second case, called massless (or atypical), when the  inequality \eqref{eh} is equality, there is a general KW-formula for maximally atypical tame  integrable $\ga$-modules, conjectured in \cite{KWNT} and proved  in \cite{GK2} for all $\g$ in question, except for $\g=D(2,1;\frac{n}{m})$,  $\nu\ne0$, which leads to the character formula  \eqref{ff2}.  Character formulas were also given in \cite{ET2}  (resp. \cite{M}) for the $N=4$   superconformal algebra  (resp. for $W^k_{\min}(spo(2|3))$, hence for the $N=3$  superconformal algebra).  The proofs  given in  these papers are incomplete since they assume that their list of singular vectors is complete and that in the  usual argument of inclusion-exclusion of Verma modules subsingular vectors cancel out. 
Their formulas for both massive and massless representations coincide with  \eqref{ff1} and \eqref{ff2}, respectively.   

In our next paper of this series we will study unitarity
of twisted representations of minimal W-algebras.

\vskip10pt
Throughout the paper the base field is $\C$, and $\mathbb Z_+$ and $\mathbb N$ stand  for the set of non negative and positive integers, respectively.\vskip5pt

\section{Setup}
\subsection{Basic  Lie superalgebras} Let  $\g=\g_{\bar 0}\oplus \g_{\bar 1}$ be a basic  finite-dimensional Lie superalgebra over $\C$ as in 
\eqref{ssuper}.
%Recall that this means that the even part $\g_{\bar 0}$  of  $\g$ is a proper reductive subalgebra and that $\g$ carries a non-degenerate bilinear  supersymmetric invariant form $(\cdot, \cdot)$. Recall that a complete list of such superalgebras consists of $sl(m|n)$ for $1\le m<n$, $osp(m|n)\cong spo(n|m)$ for $m,n\ge 1$, $n$ even, $psl(n|n)$  for $n\ge 2$, $F(4), G(3)$ and $D(2,1;a)$ for $a\ne 0,-1$ \cite{Kacsuper}; we shall assume that $a\in\R$.\par
Choose a Cartan subalgebra  $\h$ of $\g_{\bar 0}$. It is a maximal $ad$--diagonalizable subalgebra of $\g$, for which the root space decomposition is of the form
\begin{equation}\label{rs}
\g=\h\oplus\bigoplus_{\a\in\D}\g_\a,
\end{equation}
where $\D\subset\h^*\setminus\{0\}$ is the set of roots. In all cases, except for $\g\cong psl(2|2)$,
the root spaces have dimension $1$. In the case $\g=psl(2|2)$ one can achieve this property by embedding in $pgl(2|2)$ and replacing 
\eqref{rs} by the root space decomposition with respect to a Cartan subalgebra of $pgl(2|2)$, which we will do.\par
Let $\Dp$ be a subset of positive roots and $\Pi=\{\a_1,\ldots,\a_r\}$ be the corresponding set of simple roots. We will denote by $\Pi_{\bar 0},\,\Pi_{\bar 1},$ the sets of even and odd simple roots, respectively.
%Let $\Pi_1$ be the set of simple odd roots.
For each $\a\in\Dp$ choose $X_\a\in \g_\a$ and $X_{-\a}\in\g_{-\a}$ such that $(X_\a|X_{-\a})=1$,
 and let $h_\a=[X_\a,X_{-\a}]$. Let  $e_i=X_{\a_i}, f_i=X_{-\a_i},\,i=1,\ldots,r$.  The set $\{e_i, f_i, h_{\a_i}\mid i=1,\ldots,r\}$ 
generates $\g$, and satisfies the following  relations 
\begin{equation} \label{r1}
[e_i,f_j]=\d_{ij}h_{\a_i},\quad
[h_{\a_i},e_j]=(\a_i|\a_j)e_j,\quad
[h_{\a_i},f_j]=-(\a_i|\a_j)f_j.
\end{equation}
The Lie superalgebra $\tilde \g$ on generators  $\{e_i, f_i, h_{\a_i}\mid i=1,\ldots,r\}$ subject to relations \eqref{r1} is a (infinite-dimensional) $\Z$--graded Lie algebra, where the grading is defined by $\deg h_{\a_i}=0, \deg e_i=-\deg f_i=1$,   with a unique $\Z$--graded  maximal ideal, and $\g$ is the quotient of $\tilde \g$ by this ideal.
We assume that 
$(\a_i|\a_j)\in\R$ for all $\a_i,\a_j\in\Pi$.\par
\subsection{Conjugate linear involutions and real forms}\label{222} In the above setting, given a collection of  complex numbers $\L=\{\l_1,\ldots,\l_r\}$ such that $\l_i\in \sqrt{-1}\R$ if $\a_i$ is an odd root and $\l_i\in \R$ if $\a_i$ is an even root,
we can define an antilinear involution $\omega_\L:\g\to\g$ setting
\begin{equation}\label{omega}
\omega_\L(e_i)=\l_if_i,\quad\omega_\L(f_i)={\bar\l_i}^{-1}e_i,\quad
\omega_\L(h_{\a_i})=-h_{\a_i},\ 1\leq i\leq r.
\end{equation}
Since $\omega_\L$ preserves  relations \eqref{r1}, it induces  an antilinear involution of $\tilde\g$, and, since $\omega_\L$ preserves the $\Z$--grading of $\tilde\g$, it preserves its unique maximal ideal, hence it induces an antilinear involution of $\g$.

%Clearly $\omega_\L(X_\a)$ is a multiple of  $X_{-\a}$. 
Set $\sigma_\a=-1$ if $\a$ is an odd negative root and  $\sigma_\a=1$ otherwise, so that 
$(X_\a|X_{-\a})=\sigma_\a$. Let 
$$\xi_\a=\begin{cases}sgn(\a|\a) \quad&\text{if $\a$ is an even root,} \\ 1\quad &\text{if $\a$ is an odd root.}\end{cases}$$
%It was shown in \cite[Lemma 3.8]{YK}  that there is a choice of root vectors $\{X_\a\}$ such that the  assignment 
%$X_\a\mapsto \sqrt{-1}^{p(\a)}\s_\a X_{-\a}, h\mapsto h, h\in\h$, defines an anti-involution of $\g$. We fix such a choice of the $\{X_\alpha\}$, which we shall call 
%the Iohara-Koga root vectors.
%, one can choose the $X_\a$ in such a way that 
% \begin{equation}\label{norm}
%\omega_{\L_0}(X_\a)=\sqrt{-1}^{p(\a)}\s_\a X_{-\a}, \quad \a\in\D.\end{equation}
%Hereafter $p(\alpha) =0\in\mathbb Z$ (resp. $1\in\mathbb Z$) if $\alpha$ is even (resp. odd).
%We fix such a choice of the $X_\a$.
%We define
%$\l_\a\in\C$ for each $\a\in\D$ by 
%$$\l_\a=\begin{cases}-\I\quad&\text{if $\a$ is odd}\\-sgn(\a|\a)\quad&\text{if $\a$ is even}\end{cases}.$$ 
Then in \cite[(4.13), (4.15)]{GKMP} it is proven (using results from \cite{YK}), that one can choose  root vectors $X_\a$ in such a way that 
\begin{equation}\label{ox}
\omega_{\L}(X_\a)=-\s_\a\xi_\a\l_\a X_{-\a},
\end{equation}
where 
\begin{equation} \label{lambdaalfa}
\l_\a=\prod_i(-\xi_{\a_i}\l_i)^{n_i}\text{ \ for $\a=\sum_{i=1}^{r}n_i\a_i$.}
\end{equation} 
We shall call this a {\it good choice} of root vectors.

\subsection{Invariant Hermitian forms on vertex algebras} Let $V$ be a conformal vertex algebra with conformal vector $L=\sum_{n\in\mathbb Z} L_nz^{-n-2}$ (see \cite{KMP} for the definition and undefined notation).   Let $\phi$ be a conjugate linear involution of $V$. A Hermitian form $H(\, . \,\, , \, .\, )$ on $V$ is called $\phi$--{\it invariant} if, for all $a\in V$, one has \cite{KMP}
\begin{equation}
H(v,Y(a,z)u)=H(Y(A(z)a,z^{-1})v,u),\quad u,v\in V.
\end{equation}
Here the linear map $A(z):V\to V((z))$ is defined by 
\begin{equation}\label{12}
A(z)=e^{zL_1}z^{-2L_0}g,
\end{equation}
where
\begin{equation}\label{113}
g(a)=e^{-\pi\sqrt{-1}(\tfrac{1}{2}p(a)+\D_a)}\phi(a),\quad a\in V,
\end{equation}
$\D_a$ stands for the   $L_0$--eigenvalue of $a$, and 
$$p(a)=\begin{cases}0\in\Z&\text{ if $a\in \g_{\bar{0}}$},\\ 1 \in\Z&\text{ if $a\in \g_{\bar{1}}$. }
\end{cases}
$$

%\begin{definition}\label{unidef} We say that a conformal vertex algebra $V$ is {\sl unitary} if there exists a conjugate linear involution $\phi$ of $V$ and a 
%$\phi$--invariant positive definite Hermitian form on $V$.
%\end{definition}
\section{The almost compact conjugate linear involution of $\g$}\label{2}
 From now on we let $\g$ be a basic simple finite-dimensional Lie superalgebra such that 
 \begin{equation}\label{goss}
\g_{\bar 0}=\mathfrak s\oplus \g^\natural.
\end{equation}
where  $\mathfrak s\cong sl_2$ and $\g^\natural$ is the centralizer of $\mathfrak s $ in $\g$.

This corresponds to consider $\g$  as in  Table 2 of \cite{KW1}. We will also assume that $\g^\natural$ is not abelian; this condition rules out   $\g=spo(2|m),\ m=0,1,2$. The explicit list is given in the leftmost column of Table 1. Note that $sl(2|1)$ and $osp(4|2)$ are missing there since $sl(2|1)\cong spo(2|2)$ and
$osp(4|2)\cong D(2,1;a)$ with $a=1, -2$ or $-\tfrac{1}{2}$.
\par
First, we prove the simple lemma mentioned in the Introduction, which states that the first two conditions of \eqref{primas} imply the third one.
\begin{lemma}\label{31} Let $\g$ be a simple Lie superalgebra with an invariant supersymmetric bilinear form $(.|.)$, let $x\in \g$, and let $\phi$ be a conjugate linear involution of $\g$, such that 
\begin{equation}\label{aux} 
(x|x)\text{ is a non-zero real number, and } \phi(x)=x.
\end{equation} Then 
\begin{equation}\label{forma}\overline{(\phi(a)|\phi(b))}=(a|b),\text{ for all  $a,b\in\g$.}\end{equation}
\end{lemma}
\begin{proof} Note that $\overline{(\phi(a)|\phi(b))}$ is an  invariant supersymmetric bilinear form as well, hence it is proportional to $(a|b)$ since $\g$ is simple. Due to \eqref{aux}
these two bilinear forms coincide.
\end{proof}

We now discuss the existence of an almost compact involution of $\g$ (see Definition \ref{ac}).
\begin{prop} \label{constructphi} For any $sl_2$--triple $\mathfrak s=\{e,x,f\}$, such that $[e,f]=x, [x,e]=e, [x,f]=-f,$ and \eqref{goss} holds, an almost compact involution exists.
\end{prop}
\begin{proof}  Choose  a Cartan subalgebra $\mathfrak t$ of $\g_{\bar 0}$. We observe that if we prove the existence of an almost compact involution $\phi$ for
a special choice of $\{e,x,f\}$, then an almost compact involution exists for any choice of the $sl_2$--triple. Indeed, if $\{e',x',f'\}$ is another 
$sl_2$--triple, then there is an inner automorphism $\psi$ of $\mathfrak s$ mapping $\{e,x,f\}$ to $\{e',x',f'\}$, which extends to an inner automorphism of $\g$. Therefore 
$\phi'=\psi\phi\psi^{-1}$ is an almost compact involution for $\{e',x',f'\}$.
%According to Definition \ref{ac}, we  have to verify that there exists a conjugate linear
%automorphism $\phi$  of $\g$  and a $sl_2$--triple  $\{e,f,x\}$ in $\mathfrak s$ such that $(\g^\natural)^\phi$, the set of fixed points of $\phi$ in $\g^\natural$,  is a compact form of $\g^\natural$ and  $\phi(f)=f, \phi(x)=x$,  $\phi(e)=e$. 
The construction of  $\{e,x,f\}$ and $\phi$ and the verification of properties (i)--(iii) in Definition \ref{ac} will be done in four steps:
\begin{enumerate}
\item make a suitable choice of positive roots for $\g$ with respect to  $\mathfrak t$; 
\item define $\phi$ by specializing \eqref{omega};
\item construct $\{e,f,x\}$ and verify that  $\phi(f)=f, \phi(x)=x$, $\phi(e)=e$;
\item check  that  $\phi$ is a compact involution for  $\g^\natural$;
\end{enumerate}
{\sl Step 1.} We need some preparation. Let 
$\D^\natural$ be the set of roots  of  $\g^\natural$ with respect to the Cartan subalgebra $\mathfrak t\cap \g^\natural$.
Let $\{\pm\theta\}$ be the $\mathfrak t\cap \mathfrak s$--roots of $\mathfrak s$. Then $R_{\bar 0}=\{\pm\theta\}\cup\D^\natural$ is the set of roots of $\g_{\bar 0}$ with respect to $\t$. 
%Fix a set of positive roots $\Dp_0$ of $\D_0$ such that $\theta\in \Dp_0$.

Let $R$ be the set of roots of $\g$ with respect to $\t$, let  $R^+$ be the subset of positive roots whose corresponding set of simple roots $S=\{\a_1,\ldots,\a_r\}$ is displayed in Table 1. 
%When $\g=psl(2|2)$ we display  a positive system in $gl(2|2)$ (see the discussion in  \cite[4.2]{GKMP}).
\begin{table}\label{Teble1}
{\scriptsize
\begin{tabular}{c | c| c |c }
$\g$&$S$&\text{ $(. |.)$} & $\theta$\\
\hline
$psl(2|2)$&
$\{\e_1-\d_1,\d_1-\d_2,\d_2-\e_2\}$&$(\e_i|\e_j)=\d_{i,j}=-(\d_i|\d_j)$ & $\e_1-\e_2$\\
& &$(\e_i|\d_j)=0$
\\\hline
$sl(2|m), m>2$&
$\{\e_1-\d_1,\d_1-\d_2,\ldots,\d_m-\e_2\}$&$(\e_i|\e_j)=\d_{i,j}=-(\d_i|\d_j)$ & $\e_1-\e_2$\\
& &$(\e_i|\d_j)=0$
\\\hline
$osp(4|m), m>2$&$\{\e_1-\e_2,\e_2-\d_1,\d_1-\d_2,\ldots,\d_{m-1}-\d_m,2\d_m\}$&$(\e_i|\e_j)=\d_{i,j}=-(\d_i|\d_j)$ & $\e_1+\e_2$\\
& &$(\e_i|\d_j)=0$\\\hline
$spo(2|2m+1), m\ge1$&
$\{\d_1-\e_1,\e_1-\e_2,\ldots,\e_{m-1}-\e_m,\e_m\}$&$(\e_i|\e_j)=-\tfrac{1}{2}\d_{i,j}, (\d_1|\d_1)=\tfrac{1}{2},\,(\e_i|\d_1)=0$& $2\d_1$
\\\hline
$spo(2|2m), m\ge 3$&
$\{\d_1-\e_1,\e_1-\e_2,\ldots,\e_{m-1}-\e_m,\e_{m-1}+\e_m\}$&$(\e_i|\e_j)=-\tfrac{1}{2}\d_{i,j}, (\d_1|\d_1)=\tfrac{1}{2},\,(\e_i|\d_1)=0$& $2\d_1$
%\\\hline
%$D(2,1;a)$&$\{\e_1-\e_2, \e_2-\delta, 2\delta\}$& $(\e_i|\e_j)=\delta_{i,j}, (\delta|\delta)=-(1+a)/2$ & $\e_1+\e_2$\\
% & & $(\e_1|\delta)=(\e_2|\delta)=(1-a)/4$\\\hline
 \\\hline
$D(2,1;a)$&$\{\e_1-\e_2-\e_3, 2\e_2, 2\e_3\}$& $(\e_1|\e_1)=\frac{1}{2}, (\e_2|\e_2)=\frac{-1}{2(1+a)},  (\e_3|\e_3)=\frac{-a}{2(1+a)}$ & $2\e_1$\\
 && $(\e_1|\e_2)=(\e_1|\e_3)=(\e_2|\e_3)=0$ \\\hline
$F(4)$&$\{\tfrac{1}{2}(\d_1-\e_1-\e_2-\e_3), \e_3,\e_2-\e_3,\e_1-\e_2\}$&$(\e_i|\e_j)=-\tfrac{2}{3}\d_{i,j}, (\d_1|\d_1)=2$& $\d_1$\\
& &$(\e_i|\d_1)=0$
\\\hline
$G(3)$&$\{\d_1+\e_3, \e_1,\e_2-\e_1\}$&$(\e_i|\e_j)=\frac{1-3\d_{i,j}}{4}, (\d_1|\d_1)=\frac{1}{2}$& $2\d_1$\\
& &$(\e_i|\d_1)=0,\ \e_1+\e_2+\e_3=0$
\end{tabular}
}

\caption{Simple roots, invariant form, and highest root of $\g$}

\end{table}
Note that   $\theta$ is the  highest root of 	$R$.
\vskip5pt
\noindent{\sl Step 2.} Define 
\begin{equation}\L_0=\{\l_1,\ldots,\l_r\},\quad\l_i=\begin{cases}- sgn (\a_i|\a_i)\quad&\text{if $\a_i$ is  even,}\\
-\sqrt{-1}\quad&\text{if $\a_i$ is  odd.}\end{cases}\label{lambda0}\end{equation}
 Set 
  $\phi=\omega_{\L_0}$ (see \eqref{omega}).
% $S_1$ the set of odd roots in $S$:
%  \begin{equation}\label{omega0}
 % \l_i=-sign(\a_i|\a_i)\text{ for $\a_i\notin S_1$,}\quad  \l_i=\sqrt{-1}\text{ for $\a_i\in S_1$}.
 % \end{equation} 
 % By \eqref{ox}, \eqref{lambdaalfa}
%  \begin{equation}\label{ancora}
% \phi(X_\a)=-\s_\a\xi_\a\prod_i(-\xi_{\a_i}\l_i)^{n_i} X_{-\a}, \quad\text{where  $\a=\sum_{i=1}^{r}n_i\a_i$.}
% \end{equation}

 \vskip5pt
\noindent{\sl Step 3.} Consider a good choice of root vectors $X_\a$ for $\L_0$.
Set 
\begin{equation}\label{efx} x=\tfrac{\sqrt{-1}}{2}(X_\theta-X_{-\theta}), \quad 
e=\tfrac{1}{2}(X_\theta+X_{-\theta}+\sqrt{-1}h_\theta), \quad f=\tfrac{1}{2}(X_\theta+X_{-\theta}-\sqrt{-1}h_\theta).\end{equation}
% $$X_\theta=-\tfrac{\sqrt-1}{2}x+\tfrac{1}{2}(e+f),\quad X_{-\theta}=\tfrac{\sqrt-1}{2}x+\tfrac{1}{2}(e+f)$$
If  $\theta=\sum_{i=1}^{r}m_i\a_i$, then, by our special choice of $\Dp$, we have  either $m_i=2$ for exactly one odd simple root $\a_i$, or $m_i=m_j=1$ for exactly two odd distinct simple roots $\a_i,\a_j$ (this corresponds to the fact that $R^+$ is distinguished, in the terminology of \cite{GKMP}). By \eqref{ox} we have
\begin{equation}\label{phixtheta}\phi(X_\theta)=- (\sqrt-1)^2 X_{-\theta}=X_{-\theta}. 
\end{equation}
Since $h_\theta=\sum_{i=1}^r m_ih_{\a_i}$ and $\phi(h_{\a_i})=-h_{\a_i}$, it is clear from \eqref{efx}
 that $\phi$ fixes $e, f, x$. One checks directly that $\{e,f,x\}$ is an $sl_2$--triple.

\vskip5pt
\noindent{\sl Step 4.} Endow $\g$ with the $\ganz$--grading
\begin{equation}\label{grading}\g=\bigoplus_{i\in\ganz} \mathfrak q_i\end{equation}
which assigns   degree $0$ to $h\in\t$ and to $e_i$ and $f_i$ if $\a_i$ is even,
and degree $1$ to $e_i$ and degree $-1$ to $f_i$, if $\a_i$ is odd.

A direct check on Table 1 shows that $\mathfrak q_0=\g^\natural$.
Recall from \cite[Proposition 4.5]{GKMP} that the fixed points of $\phi$ in $\mathfrak q_0$ are a compact form of $\q_0$ if and only  if 
$\l_i(\a_i|\a_i)<0$ for all $\a_i\in S\setminus S_1$. Step 4 now follows from \eqref{lambda0}.
\end{proof}
\section{Explicit expressions for almost compact real forms}\label{4}
%Recall that we are considering Lie superalgebras for which \eqref{goss} holds. In this setting, we 
% say that a conjugate linear involution of $\g$  is {\it almost compact} if
 %\begin{enumerate}
%\item  it is compact on $\g^\natural$;
%\item   it is the identity on $\mathfrak s$.
%\end{enumerate}\par
 In this section we exhibit explicitly  an almost compact involution 
$\phi$ in each case and discuss its uniqueness. If $\phi$ is an almost compact involution of $\g$,  we denote by $\g^{ac}$ the corresponding real form  (the fixed point set of $\phi$).
We can define $\g^{ac}$ by specifying a real form $\g^{ac}_{\bar 0}$ of $\g_{\bar 0}$ and a real form $\g^{ac}_{\bar 1}$ of  $\g_{\bar 1}$.
% (cf. \cite[Proposition 5.3.2]{Kacsuper}). 
%We chose  $\phi$ to be the conjugation with respect to this real form, which we denote by $\g^{ac}$.

(1)  $\g=spo(2|m)$. Then $\g_{\bar 0}=sl_2\oplus so_m$ and $\g_{\bar 1}=\C^2\otimes \C^m$ as $\g_{\bar 0}$--module. We set 
$$\g^{ac}_{\bar 0}=sl_2(\mathbb R)\oplus so_m(\mathbb R),\quad \g^{ac}_{\bar 1}=\R^2\otimes \R^m.
$$
Explicitly, let $B$ be a non-degenerate $\R$--valued bilinear form of the superspace $\R^{2|m}$ with matrix
$\left(\begin{array}{c c|c} 0 &1 &0\\-1 & 0& 0\\\hline  0& 0& I_m\end{array}\right)$. Then for $\g=spo(2|m)$ we have:
$$\g^{ac}=\{A \in sl(m|n;\R)\mid B(Au,v)+(-1)^{p(A)p(u)}B(u,Av)=0\}.
$$

(2) $\g=psl(2|2)$. 
Let $H$ be a $\C$--valued non-degenerate sesquilinear form on the superspace $\C^{2|2}$ whose matrix is $diag(\I,-\I,1,1)$. Set
$$\tilde\g^{ac}=\{A \in sl(2|2;\C)\mid H(Au,v)+(-1)^{p(A)p(u)}H(u,Av)=0\}.
$$
Then 
$$\g^{ac}=\tilde\g^{ac}/\R \sqrt{-1} I.$$
Explicitly, we have 
 $\g_{\bar 0}=sl_2\oplus sl_2$ and $\g_{\bar 1}=\left\{\left(\begin{array}{c|c} 0 & B \\\hline C & 0\end{array}\right)\mid B,C\in M_{2,2}(\C)\right\}$ as a $\g_{\bar 0}$--module. Then 
\begin{align*}\tilde \g^{ac}_{\bar 0}&=\left\{\left(\begin{array}{c|c} A & 0 \\\hline 0& D\end{array}\right)\mid  A\in su(1,1), D\in su_2\right\},\\
\tilde \g^{ac}_{\bar 1}&=\left\{\left(\begin{array}{c c |c}0 & 0 & u \\ 0 & 0  & v\\\hline
\I \bar u^t &-\I \bar v^t & 0\end{array}\right)\mid  u, v\in \C^2\right\}.
\end{align*}

(3)  $\g=D(2,1;a)$. Then $\g_{\bar 0}=sl_2\oplus sl_2\oplus sl_2= so(4,\C)\oplus sl_2$ and $\g_{\bar 1}=\C^2\otimes \C^2\otimes \C^2=
\C^4\otimes \C^2$ as $\g_{\bar 0}$--module. We set 
$$\g^{ac}_{\bar 0}=so(4,\R)\oplus span_\R\{e,f,x\},\quad \g^{ac}_{\bar 1}=\R^4\otimes \R^2.
$$
To get an explicit realization, consider
the contact Lie superalgebra (see \cite{Kacsuper} for more details)  $$K(1, 4)=\C[t,\xi_1,\xi_2,\xi_3,\xi_4]$$ where $t$ is an even variable and $\xi_i, 1\leq i\leq 4$, are  odd variables. Introduce on the associative superalgebra $K(1, 4)$ a $\Z$--grading by letting 
$$\deg' t=2,\quad \deg'\xi_i=1,$$
and the bracket
$$\{F,G\}=(2-\sum_{i=1}^4\xi_i\partial_i)F\partial_t G-\partial_t F(2-\sum_{i=1}^4\xi_i\partial_i)G+\sum_{i=1}^4
(-1)^{p(F)}\partial _iF\partial_iG,
$$
where $\partial_i=\partial_{\xi_i}$. This is a $\Z$--graded Lie superalgebra with compatible grading 
$\deg  F=\deg'F-2$. We have
$$K(1,4)=\bigoplus_{j\geq -2} K(1,4)_j,$$
where
\begin{align*}&K(1,4)_{-2}=\C 1,  &&K(1,4)_{-1}=span_\C(\xi_i\mid 1\leq i \leq 4), \\&
K(1,4)_0=span_\C(\xi_i\xi_j, t\mid 1\leq i,j \leq 4),
&&K(1,4)_1=\g_1'\oplus \g_1'',\text{ where}\\&\g_1'=span_\C(t\xi_i\mid 1\leq i\leq 4), &&\g_1''=span_\C(\xi_i\xi_j\xi_k\mid 1\leq i,j,k \leq 4).
\end{align*}
Note that $span_\C(\xi_i\xi_j\mid 1\leq i,j \leq 4)=\Lambda^2\C^4\cong so(4,\C)$,  that $\g'_1$ is isomorphic to the standard representation $\C^4$ of $so(4,\C)$ and that $\g''_1$ is isomorphic to
$\Lambda^3\C^4$, so that $K(1,4)_1=\C^4\oplus \C^4$ as $so(4,\C)$--module. Also notice that $\{\g_1',\g_1'\}=\C t^2, \{\g_1'',\g_1''\}=0$.
Fix now a copy $\tilde \g_b$ of  an $so(4,\C)$--module $\C^4$  in $\C^4\oplus \C^4$, depending on a constant $b\in \R$,  as follows. Set, for $1\leq i\leq 4$,
$$a_i=t\xi_i+ b \hat \xi_i,\text{ where }  \hat \xi_i=(-1)^{i+1}\prod_{j\ne i} \xi_j,$$
and define
$$\tilde \g_b=\sum_{i=1}^4\C a_i.$$
 Let $b\in \R$. Note that, setting $\xi=\xi_1\xi_2\xi_3\xi_4$, we have 
\begin{align*}
\{t\xi_i+ b \hat \xi_i,t\xi_j+ b \hat \xi_j\}&=\d_{ij}(-t^2+2b\xi).
\end{align*}
Hence, if we set
$$e=-t^2+2b\xi,\quad f=-1,\quad x=t/2,$$
then $\{e,x,f\}$ is an $sl_2$--triple. Set 
$$\g^{ac}=\R. 1\oplus \left(\sum_{i=1}^4\R \xi_i\right)\oplus\left( \sum_{i,j=1}^4\R \xi_i\xi_j\oplus \R\tfrac{t}{2}\right)\oplus\left( \sum_{i=1}^4\R a_i\right)\oplus \R(-t^2+2b\xi). $$
Then $\g^{ac}$ is an almost compact form  of  $D(2,1;\frac{1+b}{1-b})$. To prove this, it suffices to calculate the  Cartan matrix for a choice of Chevalley generators of the complexification of $\g^{ac}$.
Fix a Cartan subalgebra in $\g^\natural = so(4,\C)$ as the span of  $v_2=-\I\xi_1\xi_2, v_3=-\I\xi_3\xi_4$. Set  $v_1=t$; then $\{v_1,v_2,v_3\}$ is a basis of a  Cartan subalgebra of $\g$. Let $\{\e_1,\e_2,\e_3\}$ be the   dual basis  to $\{v_1,v_2,v_3\}$. One can choose $\{\a_1=\e_2-\e_1,\a_2=\e_1-\e_3,\a_3=\e_1+\e_3\}$ as a set of simple roots. The associated Chevalley generators are 
{\small
\begin{align*}
 &e_1=-\I a_1+a_2 &&\!\!e_2=\xi_1\xi_3+\xi_2\xi_4+\I(\xi_1\xi_4-\xi_2\xi_3)&&e_3=\xi_1\xi_3-\xi_2\xi_4-\I(\xi_1\xi_4+\xi_2\xi_3)\\
 &f_1=\I \xi_1+\xi_2  &&\!\!f_2=\xi_1\xi_3+\xi_2\xi_4-\I(\xi_1\xi_4-\xi_2\xi_3)&&f_3=\xi_1\xi_3-\xi_2\xi_4+\I(\xi_1\xi_4+\xi_2\xi_3)\\
 &h_1=-2v_1+2v_2+2b\, v_3 &&\!\!h_2=4v_1-4v_3 &&h_3=4v_1+4v_3
\end{align*}
}
and the corresponding Cartan matrix, normalized as in \cite{Kacsuper},
is 
$\begin{pmatrix} 0 &1 &\frac{1+b}{1-b}\\ -1&2&0\\-1 & 0 &2\end{pmatrix}$. Hence $a=\frac{1+b}{1-b}$ and therefore all $a\neq -1$ occur in this 
construction. Since this subalgebra is $17$--dimensional, it is isomorphic to $D(2,1;a)$.
\begin{remark} 
Note that $a=0$ for $b=-1$. In this case, $D(2,1;0)$ contains a $11$--dimensional   solvable  ideal generated by $f_1$, which is spanned by
$h_1$ and the root vectors relative to roots having $\a_1$ in their support.
If  we replace $a_i$ by $a_i/b$ and $h_1$ by $h_1/b$, and let $b$ tend to  $+\infty$, we  recover also the Lie superalgebra of derivations of $psl(2|2)$, and its almost compact real form.
%; the real form is
%$$\g^{ac}=\R. 1\oplus \left(\sum_{i=1}^4\R \xi_i\right)\oplus\left( \sum_{i,j=1}^4\R \xi_i\xi_j\oplus \R\tfrac{t}{2}\right)\oplus\left( \sum_{i=1}^4\R \hat\xi_i\right)\oplus \R\xi. $$
\end{remark}
(4)  $\g=G(3)$. Then $\g_{\bar 0}=sl_2\oplus G_2$ and $\g_{\bar 1}=\C^2\otimes L_{\min}$, where  $L_{\min}$ is the complex 7--dimensional irreducible representation  of $G_2$, and we let 
$$\g^{ac}_{\bar 0}=sl_2(\mathbb R)\oplus G_{2,0},\quad \g^{ac}_{\bar 1}=\R^2\otimes  L_{\min,0}.
$$
where $G_{2,0}$ is the real compact form of $G_2$ and $L_{\min,0}$ is the real 7--dimensional  irreducible  representation  of $G_{2,0}$ whose complexification is $L_{\min}$.

(5)  $\g=F(4)$. Then $\g_{\bar 0}=sl_2\oplus so_7$ and $\g_{\bar 1}=\C^2\otimes spin_7$, where $spin_7$ is the complex spinor representation of $so_7,$ and we let
$$\g^{ac}_{\bar 0}=sl_2(\mathbb R)\oplus so_7(\R),\quad \g^{ac}_{\bar 1}=\R^2\otimes  spin (\R^7), 
$$
where $spin(\R^7)$ is the spinor representation of the compact group $so_7(\R)$.

It is proved in \cite[Proposition 5.3.2]{Kacsuper} that in both cases (4) and (5) $\g^{ac}=\g^{ac}_{\bar 0}\oplus \g^{ac}_{\bar 1}$ is an almost compact form of $\g$.
\subsection{Uniqueness of the almost compact involution} 
 \begin{proposition}  
%Let $\mathfrak n=\g_0+\g_{-1/2}+\g_{-1}$ be a Lie superalgebra with compatible
%$\tfrac{1}{2} \mathbb Z$-grading, such that $\g_0$ is reductive, its representation on $\g_{-1/2}$
%is irreducible, $\g_{-1}=\C f$ is 1-dimensional, and $[\g_{-1/2},\g_{-1/2}]=\g_{-1}$.
%Let $\phi$ be a conjugate linear involution of $\mathfrak n$, which preserves the 
%grading, fixes $f$, and induces a compact involution on $\g_0$.
%Then the Hermitian form $H$ defined on $\g_{-1/2}$ by
%$$H(u,v)=\langle\phi(u),v\rangle$$
%is invariant with respect to  the compact form of $\g_0$, defined by $\phi$,
%and therefore is positive definite. Moreover, 
An almost compact involution  is uniquely determined up to a sign by its action on $\g_0$, provided that the $\g_0$--module $\g_{1/2}$ is irreducible.\end{proposition}
\begin{proof} 
%Hermitian form with respect to a compact group is unique up to a 
%real factor, and always exists a definite one. 
%For the second claim,  
If there are two different extensions of the compact involution, then their ratio $\psi$, say,  is 
identical on $\g_0$, hence, by Schur's lemma, $\psi$ acts as a scalar on 
$\g_{-1/2}$. Since $\phi(f)=f$, we conclude that this scalar is $\pm 1$.
\end{proof}
It remains to discuss the cases $\g=sl(2|m),\,m\geq 3$, and  $psl(2|2)$, since in all other cases of Table 1 the $\g_0$--module $\g_{1/2}$ is irreducible. In this cases $\g$ is of type I, that is $\g_{\bar 1}=\g_{\bar 1}^+\oplus \g_{\bar 1}^-$ where  $\g_{\bar 1}^{\pm}$ are contragredient 
irreducible $\g_{\bar 0}$--modules and $[\g_{\bar 1}^{\pm},\g_{\bar 1}^{\pm}]=0$.
Let $\d_\l$ be the linear  map on $\g$ defined by setting 
\begin{equation}\label{deltal}
{\d_\l}_{|\g_{\bar 0}}=Id,\quad {\d_\l}_{|\g_{\bar 1}^+}=\l Id,\quad{\d_\l}_{|\g_{\bar 1}^-}=\l^{-1}Id.\end{equation} 
Then $\d_\l$ is an automorphism of $\g$ for any $\l\in\C$.   Suppose that $\phi'$ is another conjugate almost compact linear involution such that $\phi'_{|\g_{\bar 0}}=\phi$. Then
$\phi'= \phi\circ \gamma$ with $\gamma$ an automorphism of $\phi$ such that $\gamma_{|\g_{\bar 0}}=Id.$ If $\g=sl(2|m)$, by \cite[Lemmas 1 and 2]{SERG}, we have  $\gamma=\d_\l$.
Since $\phi(\g_{\bar 1}^+)=\g_{\bar 1}^-$ and $(\phi')^2=Id$ we have that  $\l\in\R$.
If $\g=psl(2|2)$, then $\gamma$ belongs to a three-parameter family of  automorphisms explicitly described in  \cite[\S 4.6]{GKMP}, and contained in $SL(2,\C)$. This $SL(2,\C)$ is the group of automorphisms of $\g$ corresponding to the Lie algebra $sl_2$ of outer derivations of $\g$.
\par

\begin{remark} Note that if $\phi$ is an almost compact involution, then 
$$\tilde \phi(a)=(-1)^{2j}\phi(a),\ a\in\g_{j}$$
is again an almost compact involution.
\end{remark}
\section{The bilinear form $\langle\cdot,\cdot\rangle$ on $\g_{-1/2}$}
Let $\mathfrak s=\{e,x,f\}$ be an $sl_2$--triple as in Proposition \ref{constructphi}. Consider  the following symmetric bilinear  forms on $\g_{-1/2}$ and  $\g_{1/2}$ respectively: 
\begin{align}\label{sesq}\langle u,v\rangle&=(e|[u,v]),\ u,v\in\g_{-1/2},\\\langle u,v\rangle_{ne}&=(f|[u,v]),\ u,v\in\g_{1/2}.\label{ne}
\end{align}
Note that, since $[f,\g_{-1/2}]=0$, we have
\begin{equation}\label{negative}
\langle [e,u],[e,v]\rangle_{ne}=-\tfrac{1}{2}\langle u,v\rangle,\ u,v\in\g_{-1/2}.
\end{equation}
We want to prove the following 
\begin{proposition}\label{fh} We can choose an almost  compact involution such that the bilinear form $\langle . ,.\rangle$ is positive definite on $\g^{ac}\cap\g_{-1/2}$.
In particular, the Hermitian form $\langle \phi(u),v\rangle$ (resp. $\langle \phi(u),v\rangle_{ne}$) is  positive definite (resp, negative definite) on $\g^{ac}\cap\g_{-1/2}$ (resp. $\g^{ac}\cap\g_{1/2}$).
\end{proposition}
 The proof requires a detailed analysis of the action of an almost compact involution on $\g_{-1/2}$. Define structure constants  $N_{\a,\beta}$  for a good choice of  root vectors (see Subsection \ref{222}) by the relation $$[X_\a,X_\beta]=N_{\a,\beta}X_{\a+\beta}.$$
Observe that $\{X_{-\theta}, X_{\theta}, \tfrac{1}{2}h_\theta\}$ is a $sl_2$--triple in $\mathfrak s$. Let 
$$\g=\C X_{\theta} \oplus \tilde \g_{1/2} \oplus \tilde \g_0 \oplus \tilde \g_{-1/2} \oplus \C X_{-\theta}$$
be the decomposition into $ad\,\tfrac{1}{2}h_\theta$ eigenspaces. By the $sl_2$ representation theory, $ad\,X_{\pm \theta}: \tilde\g_{\mp1/2}\to \tilde\g_{\pm 1/2}$ is an isomorphism 
of $\g^\natural$--modules. Moreover,  by our choice of $R^+$ in Section 3, the roots of $\tilde \g_{-1/2}$ (resp.  $\tilde \g_{1/2}$)  are precisely the negative (resp. positive) odd roots.  In particular, the map $\a\mapsto -\theta+\a$ defines a bijection between the positive and negative odd roots.
We shall need the following properties.
\begin{lemma} For a positive odd root $\a$ we have
\begin{align}
\label{rss} &N_{-\theta,\a}N_{\th,\a-\th}=1,\\
&\label{Npm1} N_{-\theta,\a}^2=1.
\end{align}
In particular $N_{\theta,\a}$ is real. \end{lemma}
\begin{proof}
Relation \eqref{rss} is proven in \cite[Lemma 4.3]{GKMP}. Equation \eqref{Npm1} follows from \cite[(4.8)]{GKMP}, noting that the $-\theta$--string through $\a$ has length $1$.
 \end{proof} 

Arguing as in  Proposition \ref{constructphi}, we can assume in the proof of Proposition \ref{fh} that   $\{e,x,f\}$ is the $sl_2$--triple defined in \eqref{efx}; $ad\, x$ defines on $\g$ a  minimal grading
\begin{equation}\label{mg}
\g=\C f\oplus\g_{-1/2}\oplus \g_0\oplus \g_{1/2}\oplus\C e. 
\end{equation}
Set, for an odd root $\a\in R^+$ 
\begin{equation}\label{u}u_\a=X_\a+\sqrt{-1} N_{-\theta,\a}X_{\a-\theta}.\end{equation}
Note that 
$$[x,u_\a]=\tfrac{\sqrt{-1}}{2}[X_\theta-X_{-\theta},X_\a+\sqrt{-1} N_{-\theta,\a}X_{\a-\theta}]
=\!\!\tfrac{1}{2}N_{-\theta,\a}N_{\theta,\a-\theta}X_{\a}-\tfrac{\sqrt{-1}}{2}N_{-\theta,\a}X_{\a-\theta}=-\tfrac{1}{2}u_\a,
$$
hence 
$\{u_\a\mid \a\in R^+, \a \text{ odd}\}$ is a basis of $\g_{-1/2}$. 
\begin{lemma}
 If $\a$ is a positive odd root then 
\begin{equation}\label{secondaa}\phi(u_\a)= -N_{-\theta,\a} u_{\theta-\a}.\end{equation}
%Moreover, there is a choice of the  root vectors $X_{\pm \theta}$ such that $N_{-\theta,\a}=1$ for any  positive odd root $\a$.
\end{lemma}
\begin{proof} 
By \eqref{ox}, $\phi(X_\a)=-\I X_{-\a}$ if $\a$ is an odd positive root, hence, by \eqref{Npm1}, since $N_{-\theta,\a}$ is real,
\begin{align}\label{primaa}\begin{split}\phi(u_\a)&=\phi(X_\a+\sqrt{-1}N_{-\theta,\a}X_{\a-\theta})= 
-(\I X_{-\a}+N_{-\theta,\a}X_{\theta-\a})\\&=
-N_{-\theta,\a}(X_{\theta-\a}+\I N_{-\theta,\a}^{-1}X_{-\a}).
\end{split}\end{align}
Note that, since $\phi(x)=x$,  $\phi(u_\a)$ has to belong to $\g_{-1/2}$. This forces 
\begin{equation} \label{ns}
N_{-\theta,\a}N_{-\theta,\theta-\a}=1,
\end{equation} 
and    \eqref{primaa} becomes \eqref{secondaa}.\end{proof}
\begin{proof}[Proof of Proposition \ref{fh}]
Set $v_\a=\tfrac{1}{2}(u_\a+\phi(u_\a))+ 	\tfrac{\I}{2}(u_\alpha-\phi(u_\alpha))$, where $\a$ runs over the positive odd roots. It is clear that $v_\a\in\mathfrak r$. We want to prove that the vectors $v_\a$ form an orthogonal basis of $\mathfrak r$. We need two auxiliary computations: 
\begin{align}\label{a1}
[e,u_\a]&=\I X_\a+N_{-\theta,\a}X_{\a-\th},\\
\label{a2} \langle u_\a,u_\be\rangle&=-(N_{-\theta,\a}+N_{-\theta,\be})\d_{\th-\a,\be}.
\end{align}
To prove \eqref{a1} use \eqref{rss}:
\begin{align*}
[e,u_\a]&=\tfrac{1}{2}[X_\th+X_{-\th}+\sqrt{-1}h_\theta,X_\a+\I N_{-\theta,\a}X_{\a-\th}]\\
&=\tfrac{1}{2}(\I N_{-\theta,\a}N_{\th,\a-\th}X_\a+N_{-\theta,\a}X_{\a-\th}+\I X_\a+N_{-\theta,\a}X_{\a-\th})\\
&=\I X_\a+N_{-\theta,\a}X_{\a-\th}.
\end{align*}
To prove \eqref{a2} use  \eqref{a1}:
\begin{align*}\langle u_\a,u_\be\rangle&=(e|[u_\a,u_\be])=([e,u_\a]|u_\be)\\
&=(\I X_\a+N_{-\theta,\a}X_{\a-\th}|X_\be+\sqrt{-1}N_{-\theta,\be}X_{\be-\theta})=\\
&=\s_{\a-\theta}N_{-\theta,\a}\d_{\th-\a,\be}-\s_\a N_{-\theta,\be}\d_{\th-\a,\be}\\
&=-(N_{-\theta,\a}+N_{-\theta,\be})\d_{\th-\a,\be}.
\end{align*}
Set $$M_{\a,\be}=-(N_{-\theta,\a}+N_{-\theta,\be}).$$
Then, using \eqref{a2}
\begin{align*}
&\langle v_\a,v_\be\rangle\\&=\langle \tfrac{1+\I}{2} u_\a-\tfrac{1-\I}{2}N_{-\theta,\a}u_{\theta-\a}, \tfrac{1+\I}{2} u_\be-\tfrac{1-\I}{2}N_{-\theta,\be}u_{\theta-\be}\rangle\\
&=\tfrac{\I}{2}\langle u_\a, u_\be\rangle
-\tfrac{1}{2}N_{-\theta,\a}\langle u_{\theta-\a}, u_\be\rangle-\tfrac{1}{2}N_{-\theta,\be}\langle u_\a, u_{\th-\be}\rangle-\tfrac{\I}{2}N_{-\theta,\a}N_{-\theta,\be}
\langle u_{\theta-\a}, u_{\theta-\be}\rangle\\
&=
\tfrac{\I}{2}M_{\a,\be}\d_{\th-\a,\be}
-\tfrac{1}{2}N_{-\theta,\a}M_{\th-\a,\be}\d_{\a,\be}-\tfrac{1}{2}N_{-\theta,\be}M_{\a,\th-\be}\d_{\th-\a,\th-\be}\\&-\tfrac{\I}{2}N_{-\theta,\a}N_{-\theta,\be}
M_{\th-\a,\th-\be}\d_{\a,\th-\be}.
\end{align*}
Therefore by \eqref{rss} and \eqref{ns}
$$\langle v_\a,v_\be\rangle=2\d_{\a,\beta}.$$
In particular, the restriction of $\langle\cdot,\cdot\rangle$ to $\g^{ac}\cap \g_{-1/2}$ is positive definite.
The final claim follows immediately from \eqref{negative}, using that $[e,\g_{-1/2}]=\g_{1/2}$.
\end{proof}
\section{A general theory of invariant Hermitian forms on modules over the vertex algebra of free boson and the Fairlie construction}\label{freeb}

Consider the infinite dimensional Heisenberg Lie algebra $\mathcal H=(\C[\tau,\tau^{-1}]\otimes \C a)\oplus \C K$ with $K$ central and bracket
$$
[\tau^n\otimes a, \tau^m\otimes a]=\d_{n,-m}nK.
$$

Let $\mathcal H_0=\C a+\C K$, and, given 
 $\mu\in\C$, define  $\mu^*\in \mathcal H_0^*$ by $\mu^*(a)=\mu, \mu^*(K)=1$. Let $M(\mu)$ be the Verma module for the Lie algebra $\mathcal H$ with highest weight $\mu^*$. Let $v_\mu$ be a highest weight vector,  i.e. $(\tau^n\otimes a)v_\mu=0$ for $n>0$, $h v_\mu=\mu^*(h)v_\mu$ for $h\in \mathcal H_0$. 
 It is well known that $M(0)$ carries a simple vertex algebra structure, called the vertex algebra of free boson, which we denote by $V^1(\C a)$, and that $M(\mu)$  is a simple module over the vertex algebra  $V^{1}(\C a)$. Moreover,  $V^{1}(\C a)$ is the universal enveloping vertex algebra of the nonlinear Lie conformal algebra $R=\C[T]\otimes\C a$  with $\l$--bracket 
$$
[a_\l a]=\l.
$$
We introduce conformal weight $\D$ on $V^1(\C a)$ by letting $\D_a=1$, and for $v\in V^1(\C a)$ we write the corresponding quantum field as 
$Y(v,z)=\sum_{j\in\mathbb Z} v_jz^{-j-\D_v}$.

Fix $t\in\C$ and set 
\begin{equation}\label{L(t)}L(t)=\frac{1}{2}:aa: + tTa\in V^1(\C a).\end{equation} It is an energy-momentum element for all $t$.  Set $H(t)=L(t)_0=\frac{1}{2}:aa:_0-ta_0$. Since $a_0=0$ as operator on $ V^1(\C a)$,  $H(t)=\frac{1}{2}:aa:_0$.  (Note that the conformal weights on $V^1(\C a)$ are the eigenvalues of $H(t)=H(0)$).

If $b\in V^{1}(\C a)$, write $b^\mu_n$ for $b^{M(\mu)}_n$. By the $-1$--st product identity 
$$
:aa:^{\mu}_0=2\sum_{j\in \nat}a_{-j}^\mu a_j^\mu+(a_0^\mu)^2.
$$
In particular 
\begin{equation}\label{actionxx}
:aa:^{\mu}_0 v_\mu=\mu^2 v_\mu.
\end{equation}
On the other hand, by  the commutator formula, 
\begin{equation}\label{xxbracketx}
\frac{1}{2}[:aa:^\mu_0,a^\mu_{j}]=\frac{1}{2}\sum_r\binom{1}{r}(:aa:_{(r)}a)_{j}=(Ta)^\mu_j+a^\mu_j=-j a^\mu_j.
\end{equation}
Recall that a basis of $M(\mu)$ is 
\begin{equation}\label{basismu}
\{(a^\mu_{-j_1})^{i_1} \cdots (a_{-j_r}^\mu)^{i_r}. v_\mu\mid j_1>\cdots >j_r>0\}.
\end{equation}

Let $M(\mu)_n$ be the eigenspace for the energy operator $H(t)$  corresponding to the eigenvalue $n+\half\mu^2-t\mu$. Since
$$
\frac{1}{2}:aa:^\mu_0 + t(Ta)^\mu_0=\frac{1}{2}:aa:^\mu_0 - ta^\mu_0
$$
and $[a^\mu_0,a_{-j}^\mu]=0$ for all $j$, it follows from \eqref{actionxx} and \eqref{xxbracketx} that
$$
M(\mu)_n=span \{(a^\mu_{-j_1})^{i_1} \cdots (a_{-j_r}^\mu)^{i_r}. v_\mu\mid \sum_s i_sj_s=n\}.
$$
Thus
$$
M(\mu)=\oplus_{n\in \ZZ_+} M(\mu)_n.
$$
 This shows that $M(\mu)$ is a positive energy $V^{1}(\C a)$--module, i.e. real parts of the eigenvalues of  $H(t)$ are bounded below. Moreover its minimal energy subspace is 
 $$
 M(\mu)_0=\C v_\mu.
 $$

 \begin{lemma}\label{Zassbosons}
\begin{equation}\label{Zasseq}
e^{zL(t)_1}=e^{zL(0)_1}\prod_{n=1}^{\infty}e^{-\frac{2t}{n}z^na_n}.
\end{equation}
\end{lemma}
\begin{proof}
Identify $V^1(\C a)$ with the polynomial algebra in infinitely many variables using \eqref{basismu} with $\mu=0$:$$\mathcal P=\C[a_{-1},a_{-2},\ldots,a_{-n},\ldots].$$ Since $L(t)_1\vac=0$ and $a_n\vac=0$ if $n>0$, both $L(t)_1$ and $a_n$ for $n>0$ act as derivations of the algebra $\mathcal P$ under our identification. It follows that both sides of \eqref{Zasseq} are automorphisms of $\mathcal P$. 
It is therefore enough to check the equality only on the generators $a_{-n}$.

We need the following formulae:
\begin{align}
[a_\l L(t)]&=\l a +t\l^2\vac,\\
[a_n,L(t)_1]&=na_{n+1}+\d_{n,-1}2tI,\label{Lta}\\
[a_n,a_{-m}]&=\d_{n,m}nI.
\end{align}
Applying these formulae we find 
\begin{equation}\label{exptrans}
e^{-\frac{2t}{n}z^na_n}(a_{-m})=e^{-2tz^n\tfrac{\partial}{\partial a_{-n}}}(a_{-m})=a_{-m}-\d_{n,m}2tz^m I.
\end{equation}
It follows that
\begin{equation}\label{prodtr}
e^{zL(0)_1}\prod_{n=1}^{\infty}e^{-\frac{2t}{n}z^na_n}(a_{-m})=e^{zL(0)_1}(a_{-m}-2tz^m\vac)=e^{zL(0)_1}a_{-m}-2tz^m I.
\end{equation}
To conclude we only need to check that, if $n\ge1$, then
$$
L(t)_1^n(a_{-m})=L(0)^n_1a_{-m}-2n!\d_{n,m}t I.
$$
We prove this by induction on $n$. If $n=1$ the formula reads
$$
L(t)_1(a_{-m})=L(0)_1a_{-m}-2\d_{1,m}t I.
$$
Using \eqref{Lta} with $t=0$ we see that the latter formula is equivalent to 
$$
L(t)_1(a_{-m})=ma_{-m+1}-2\d_{1,m}t I,
$$
which is just \eqref{Lta}.

If $n>1$ and $m=1$, then 
\begin{align*}
L(t)_1^n(a_{-1})=&L(t)_1^{n-1}L(t)_1(a_{-1})=L(t)_1^{n-1}(-2t)=0=L(0)_1^n(a_{-1}).
\end{align*}
If $n>1$ and $m>1$, then 
\begin{align*}
L(t)_1^n(a_{-m})=&L(t)_1^{n-1}L(t)_1(a_{-m})=L(t)_1^{n-1}(ma_{-m+1})\\
=&L(0)_1^{n-1}(ma_{-m+1})-2(n-1)!m\d_{n-1,m-1}t I)\\
=&L(0)_1^{n}a_{-m}-2n!\d_{n,m}t I.
\end{align*}
\end{proof}
 
 Let $\phi$ be the conjugate linear involution of the vector space $\C a$ defined by $\phi(a)=-a$. Assume from now on that $t\in\sqrt{-1}\R$. This assumption is necessary since, in order to apply the results of \cite{KMP}, we need to assume $\phi(L(t))=L(t)$. Set (cf. \eqref{12})
\begin{equation}\label{112}
A(z,t)=e^{zL(t)_1}z^{-2H(0)}g,
\end{equation}
where $g$ is defined in \eqref{113}.
 Let $\pi_Z:V^{1}(\C a)\to Zhu_{H(0)}(V^{1}(\C a))$ be the canonical projection to the Zhu algebra (see e.g. \cite[Section 2]{KMP}). Let $\omega$ be the conjugate linear anti-homomorphism of $Zhu_{H(0)}(V^{k+h^\vee}(\C x))$ defined by
  $$\omega(\pi_Z(v))=\pi_Z(A(1,t)v)$$
 It is proven in \cite[Proposition 6.1]{KMP} that $\omega$ is indeed  well-defined. 
 \begin{lemma} \label{6.2}
\begin{equation}
 \omega(\pi_Z(a))v_\mu =(\mu-2t)v_\mu
\end{equation}
\end{lemma}
\begin{proof}
By Lemma \ref{Zassbosons}, since $g(a)=a$ and $L(0)_1a=0$,
\begin{align*}
 \omega(\pi_Z(a))v_\mu=&(A(1,t)a)^\mu_0v_\mu=(e^{L(t)_1}a)^\mu_0v_\mu
 =(e^{L(0)_1}(a)-2t\vac)^\mu_0v_\mu=a^\mu_0v_\mu-2tv_\mu
\\
 =&(\mu-2t)v_\mu.
\end{align*}
\end{proof}

 Recall from \cite[Definition 6.4]{KMP} that if $V$ is a conformal vertex algebra and  $\phi$ is a conjugate linear involution of $V$,
 a Hermitian form $H(\, . \,\, , \, .\, )$ on a  $V$--module $M$ is called  $\phi$--invariant if, for all $v\in V$, $m_1,m_2\in M$
%\begin{equation}\label{i}
%( Y(e^{zL_1}gz^{-2L_0}a, z^{-1})u, v)=( u,
%Y(\phi(a), z)v).
%\end{equation}
\begin{equation*}\label{iM}
( m_1, Y_M(a, z)m_2)=( 
Y_M(A(z)a, z^{-1})m_1,m_2).
\end{equation*}
By abuse of terminology, we shall call $H(\cdot, \cdot)$ an $L$--invariant Hermitian form, where $L$ is the conformal vector of $V$.
If $\mu\in\C$ we denote by $\Re(\mu)$ and $\Im(\mu)$ the real and imaginary part of $\mu$, respectively.

\begin{proposition}\label{Eform} There is a non-zero $L(t)$--invariant Hermitian  form on $M(\mu)$ if and only if $t=\sqrt{-1}\Im(\mu)$.
\end{proposition}
\begin{proof} Let $(\cdot,\cdot)$ be the unique Hermitian form on $\C v_\mu$ such that $(v_\mu,v_\mu)=1$.
By Proposition 6.7 of \cite{KMP}, there is a non-zero $L(t)$--invariant Hermitian form on $M(\mu)$ if and only if $(\cdot,\cdot)$  is an $\omega$--invariant Hermitian form on $\C v_\mu$. By Lemma \ref{6.2}, that is equivalent to 
$$
\mu=(v_\mu,a_0 v_\mu)=(v_\mu,\pi_Z(a) v_\mu)=(\omega(\pi_Z(a))v_\mu,v_\mu)=\overline{\mu-2t}.
$$
Thus
$$
-2\overline{t}=2\sqrt{-1}\Im(\mu),
$$
hence the statement.
\end{proof}
We denote by $H_\mu$ the unique  $L(\sqrt{-1}\Im(\mu))$--invariant Hermitian form on $M(\mu)$ such that $H_\mu(v_\mu,v_\mu)=1.$
%Next we want to describe explicitly the $L(\sqrt{-1}\Im(\mu))$--invariant form on $M(\mu)$.
\begin{lemma}\label{64} If $m,m'\in M(\mu),$ then
$$
H_\mu(m,a^\mu_{n}m')=H_\mu(a^\mu_{-n}m,m')+\d_{n,0}2\sqrt{-1}\Im(\mu)H_\mu(m,m').
$$
\end{lemma}
\begin{proof}
By invariance of the Hermitian form,
\begin{align*}
H_\mu(m,a_nm')=&Res_{z}z^{n}H_\mu(m,Y^\mu(a,z)m')\\
=&Res_{z}z^{n}H_\mu(Y^\mu(A(z)a,z^{-1})m,m')\\
=&Res_{z}z^{n}H_\mu(Y^\mu(e^{zL(t)_1}z^{-2L(t)_0}g(a),z^{-1})m,m')\\
=&Res_{z}z^{n-2}H_\mu(Y^\mu(e^{zL(0)_1}a-2\sqrt{-1}\Im(\mu)z\vac,z^{-1})m,m')\\
=&Res_{z}z^{n-2}H_\mu(Y^\mu(a-2\sqrt{-1}\Im(\mu)z\vac,z^{-1})m,m').
\end{align*}
The last two steps follow by \eqref{prodtr} and the fact that $L(0)_1a=0$.
As
$$
Y^\mu(a,z^{-1})=\sum_ra^\mu_rz^{r+1},\ Y^\mu(\vac,z^{-1})=\sum_r\d_{r,0}Iz^{r},
$$
we get the result.
\end{proof}

It is now easy to compute the invariant form in the basis  \eqref{basismu}: 
\begin{align}
H_\mu\left((a^\mu_{-j_1})^{i_1} \cdots (a_{-j_r}^\mu)^{i_r}. v_\mu,(a^\mu_{-j'_1})^{i'_1} \cdots (a_{-j'_{r'}}^\mu)^{i'_{r'}}. v_\mu\right)\notag\\
=H_\mu\left((a^\mu_{j'_{r'}})^{i'_{r'}} \cdots (a_{j'_{1}}^\mu)^{i'_{1}}(a^\mu_{-j_1})^{i_1} \cdots (a_{-j_r}^\mu)^{i_r}. v_\mu,v_\mu\right)\label{formmu}.
\end{align}
It follows that the basis is orthogonal and 
$$
\left\| (a^\mu_{-j_1})^{i_1} \cdots (a_{-j_r}^\mu)^{i_r}. v_\mu\right\|_\mu=\prod_s i_s!j_s^{i_s}
$$
In particular the form is positive definite and its values on the chosen basis do not depend on $\mu$.

Let $\mu\in\C$ and $t\in \sqrt{-1}\R$.
Let $M(\mu,t)^\vee$ be the conjugate dual of $M(\mu)$  with action given by, for $b\in V^1(\C a)$, $m\in M(\mu)$, $f\in M(\mu,t)^\vee$, 
$$
(Y^{M(\mu,t)^\vee}(b,z)f)(m)=f(Y^{\mu}(A(t,z)b,z^{-1})m),
$$
 where $A(z,t)$ is defined by \eqref{12}, \eqref{113}.\par
Using the $L(\sqrt{-1}\Im(\mu))$--invariant form on $M(\mu)$ (see \eqref{formmu}), we can identify $M(\mu)$ and $M(\mu,t)^\vee$ (as vector spaces) by indentifying $m$ with $f_m:m'\mapsto H_\mu(m',m)$.

 We now want to describe explicitly the action of $V^{1}(\C a)$ under this identification. We need the following result:
 \begin{lemma}\label{zg}If $t\in\sqrt{-1}\R$, then 
\begin{align}
&z^{2H(0)}e^{z^na_n} z^{-2H(0)}=e^{z^{-n}a_n}\\
& e^{tz^na_n} g  = g e^{t(-z)^na_n}
\end{align}
\end{lemma}
 \begin{proof}If $b\in V^1(\C a)$ then,
\begin{align*}e^{2z H(0)}e^{z^na_n} e^{-2zH(0)}b&=z^{-2\D_b}\sum_r\frac{1}{r!}z^{2\D_b-2rn}z^{nr}a^r_nb\\
&=\sum_r\frac{1}{r!}z^{-rn}a^r_nb=e^{z^{-n}a_n}b.
\end{align*}
For the second formula note that
$$
g(a_n^rb)=(-1)^{\D_b-nr}\phi(a^r_nb)=(-1)^{\D_b-nr}(-1)^ra^r_n\phi(b)=(-1)^{-nr}(-1)^ra^r_ng(b)
$$
so, since $t$ is purely imaginary,
\begin{align*}e^{tz^na_n} g(b)&=\sum_r\frac{1}{r!}t^rz^{nr}a^r_ng(b)=\sum_r\frac{1}{r!}(-1)^{-nr}z^{nr}g(t^ra^r_nb)=e^{t(-z)^{n}a_n}b.
\end{align*}
 \end{proof}
 
\begin{proposition}\label{16}If $m\in M(\mu)$ and $f_m\in M(\mu,t)^\vee$ is defined by $f_m(m')=H_\mu(m',m)$, then
$$
Y^{M(\mu,t)^\vee}(b,z)f_m=f_{Y^\mu(\prod_{n=1}^{\infty}e^{\frac{2(-t+\sqrt{-1}\Im(\mu))}{n}(-z)^{-n}a_n}b,z)m}
$$
In particular the fields
\begin{equation}\label{Ymut}
Y^{\mu,t}(b,z):=Y^\mu(\prod_{n=1}^{\infty}e^{\frac{2(-t+\sqrt{-1}\Im(\mu))}{n}(-z)^{-n}a_n}b,z)
\end{equation}
define a $V^1(\C a)$--module structure on $M(\mu)$. 
\end{proposition}
\begin{proof}By definition,
\begin{align*}
(Y^{M(\mu,t)^\vee}(b,z)f_m)(m')=&H_\mu(Y^\mu(A(t,z)b,z^{-1})m',m)\\
=&(Y^\mu(e^{zL(t)_1}z^{-2L(t)_0}g(b),z^{-1})m',m)
\end{align*}
Using \eqref{Zasseq}
we can write
$$
e^{z(L(t))_1}=e^{zL(0)_1}\prod_{n=1}^{\infty}e^{-\frac{2t}{n}z^na_n}=e^{zL(\sqrt{-1}\Im(\mu))_1}\prod_{n=1}^{\infty}e^{-\frac{2(t-\sqrt{-1}\Im(\mu))}{n}z^na_n}
$$
so, by Lemma \ref{zg},
\begin{align*}
(Y&^{M(\mu,t)^\vee}(b,z)f_m)(m')\\
=&H_\mu(Y^\mu(e^{zL(\sqrt{-1}\Im(\mu))_1}\prod_{n=1}^{\infty}e^{-\frac{2(t-\sqrt{-1}\Im(\mu))}{n}z^na_n}z^{-2L(t)_0}g(b),z^{-1})m',m)\\
=&H_\mu(Y^\mu(e^{zL(\sqrt{-1}\Im(\mu))_1}z^{-2L(t)_0}g\prod_{n=1}^{\infty}e^{\frac{2(-t+\sqrt{-1}\Im(\mu))}{n}(-z)^{-n}a_n}b,z^{-1})m',m)
\end{align*}
Since the form $H_\mu$ is $L(\sqrt{-1}\Im(\mu))$--invariant, we find that
\begin{align*}
(Y&^{M(\mu,t)^\vee}(b,z)f_m)(m')
=H_\mu(m',Y^\mu(\prod_{n=1}^{\infty}e^{\frac{2(-t+\sqrt{-1}\Im(\mu))}{n}(-z)^{-n}a_n}b,z)m)\\
=&f_{Y^\mu(\prod_{n=1}^{\infty}e^{\frac{2(-t+\sqrt{-1}\Im(\mu))}{n}(-z)^{-n}a_n}b,z)m}(m').
\end{align*}
\end{proof}

To simplify notation write $\mathbf a=(a_{-1},a_{-2},\ldots)$. If $I$ is an infinite sequence $(i_1,i_2,\ldots)$, with $i_j\in\ZZ_+$ almost all zero, then set $\mathbf a^I=\prod_{r=1}^\infty a_{-r}^{j_r}$.
We can regard $b\in V^1(\C a)$ as a polynomial $b(\mathbf a)$. More precisely, we write
$$
b(\mathbf a)=\sum_{I}  c_I\mathbf a^I\vac,\ c_I\in\C.
$$
We also set
$$
\rho(z)=(z\vac_0,z^2\vac_0,z^3\vac_0,\ldots)=(zI,z^2I,z^3I,\ldots).
$$

\begin{lemma}\label{atmun}Write $Y^{\mu,t}(b,z)=\sum_{r\in \ZZ}b^{\mu,t}_rz^{-r-\D_b}$. Then
$$b^{\mu,t}_r=  \left(b(\mathbf a^{shift})\right)^\mu_{r},   
\text{ 
where $\mathbf{a}^{shift}=\mathbf a+2(-t+\sqrt{-1}\Im(\mu))\rho(-1)$. }$$            
\end{lemma}
\begin{proof}Since $b^{\mu,t}_r=Res_z z^{r+\D_b-1}(Y^{\mu,t}(b, z))$, we need to check that
$$
 Res_z z^{r+\D_b-1}(Y^{\mu,t}(b, z))=b(\mathbf a+(-t+\sqrt{-1}\Im(\mu))\rho(-1))^\mu_{r}.
 $$
 It is enough to check this for $b=\mathbf a^I\vac$. Using \eqref{exptrans}, we can write
 $$
\prod_{n=1}^{\infty}e^{\frac{2(-t+\sqrt{-1}\Im(\mu))}{n}(-z)^{-n}a_n} \mathbf a^I\vac=(\mathbf a+2(-t+\sqrt{-1}\Im(\mu))\rho(-z^{-1}))^I\vac.$$
It follows that
$$
Y^{\mu,t}(b, z)=Y^\mu(b(\mathbf a+2(-t+\sqrt{-1}\Im(\mu))\rho(-z^{-1}),z)
$$
hence we need to check that
\begin{align*}
 Res_z z^{r+\D_{\mathbf a^I}-1}&(Y((\mathbf a+2(-t+\sqrt{-1}\Im(\mu))\rho(-z^{-1}))^I\vac\otimes v, z))\\&=\left((\mathbf{a}^{shift})^I\vac\otimes v\right)_r.
\end{align*}
Indeed, setting $t_0=2(-t+\sqrt{-1}\Im(\mu))$ and letting $q$ be the number of $j$ such that $i_j\ne0$,
\begin{align*}
&Y^\mu((\mathbf a+t_0\rho(-z^{-1}))^I\vac, z))\\
 &=\sum_s\sum_{j_1\le i_1,\ldots,j_q\le i_q}\left(\prod_{p=1}^q\binom{i_p}{j_p}\left(\frac{t_0}{(-z)^p}\right)^{i_p-j_p}\right)\mathbf a^J\vac_sz^{-s-\sum_{p=1}^qpj_p}\\
 %%%
   &=\sum_s\sum_{j_1\le i_1,\ldots,j_q\le i_q}\left(\prod_{p=1}^q(-1)^{p(i_p-j_p)}t_0^{i_p-j_p}\binom{i_p}{j_p}\right)\mathbf a^J\vac_sz^{-s-\D_{\mathbf a^I\vac}}
\end{align*}
so
\begin{align*}
 Res_z z^{r+\D_{\mathbf a^I\vac}-1}&(Y^\mu(\mathbf a+t_0\rho(-z^{-1}))^I\vac, z))\\
 %%%%
   &=\sum_{j_1\le i_1,\ldots,j_q\le i_q}\left(\prod_{p=1}^q((-1)^{p}t_0)^{i_p-j_p}\binom{i_p}{j_p}\right)\mathbf a^J\vac_r\\
 %%%%
 &=\left((\mathbf a+t_0\rho(-1))^I\vac\right)_r,
\end{align*}
as wished.
\end{proof}

In particular, 
\begin{equation}\label{amut}
a^{\mu,t}_r=(a_{-1}\vac)^{\mu,t}_r=a^\mu_r-2(-t+\sqrt{-1}\Im(\mu))\d_{r,0}I
\end{equation}
so 
$$a_0^{\mu,t}=(\mu-2(-t+\sqrt{-1}\Im(\mu)))I=(\overline{\mu}+2t)I.$$
Hence we have an isomorphism of $V^1(\C a)$--modules
\begin{equation}
M(\mu,t)^\vee\cong M(\overline{\mu}+2t).
\end{equation}

Let $M[\mu,t]$ denote the vector space $M(\mu)$ equipped with the $V^{1}(\C a)$--module structure given by $b\mapsto Y^{\mu,t}(b,z)$ so that 
$$
M[\mu,t]\simeq M(\mu,t)^\vee\simeq M(\overline \mu+2t).
$$

Let $\Upsilon_{\mu,t}:M[\mu,t]\to M(\overline{\mu}+2t)$ denote such  an isomorphism.  By \eqref{amut}, $\Upsilon_{\mu,t}(v_\mu)\in\C v_{\overline{\mu}+2t}$. We can therefore normalize $\Upsilon_{\mu,t}$  so that $v_\mu\mapsto v_{\overline{\mu}+2t}$. It follows from \eqref{amut} that, if $j_1\ge j_2\ge \cdots \ge j_r$,
$$
\Upsilon_{\mu,t}(a^\mu_{-j_1}\cdots a^\mu_{-j_r}v_\mu)=\Upsilon_{\mu,t}(a^{\mu,t}_{-j_1}\cdots a^{\mu,t}_{-j_r}v_\mu)=a^{\overline{\mu}+2t}_{-j_1}\cdots a^{\overline{\mu}+2t}_{-j_r} v_{\overline{\mu}+2t}.
$$
Note that, by \eqref{formmu},
\begin{equation}\label{Upsilonisom}
H_{\overline{\mu}+2t}(\Upsilon_{\mu,t}(m),\Upsilon_{\mu,t}(m'))=H_\mu(m,m').
\end{equation}
Moreover
\begin{align}
\label{prima}Y^{\mu,t+s}(b,z)&=Y^\mu(\prod_{n=1}^{\infty}e^{\frac{2(-t-s+\sqrt{-1}\Im(\mu)}{n}(-z)^{-n}a_n}b,z)\\&=Y^{\mu,t}(\prod_{n=1}^{\infty}e^{\frac{-2s}{n}(-z)^{-n}a_n}b,z)\notag,
\end{align}
and, if $m\in M(\mu)$ and $m'\in M(\overline \mu+2s)$,
\begin{align*}
&H_{\overline\mu+2s}(\Upsilon_{\mu,s}Y^{\mu,t}(b,z)m,m')=H_{\overline\mu+2s}(\Upsilon_{\mu,s}Y^{\mu,s+(t-s)}(b,z)m,m')\\
&=H_{\overline\mu+2s}(\Upsilon_{\mu,s}Y^{\mu,s}(\prod_{n=1}^{\infty}e^{\frac{-2(t-s)}{n}(-z)^{-n}a_n}b,z)m,m')\\
&=H_{\overline\mu+2s}(Y^{\overline\mu+2s}(\prod_{n=1}^{\infty}e^{\frac{-2(t-s)}{n}(-z)^{-n}a_n}b,z)\Upsilon_{\mu,s}(m),m')\\
&=H_{\overline\mu+2s}(Y^{\overline\mu+2s,t+s-\sqrt{-1}\Im(\mu)}(b,z)\Upsilon_{\mu,s}(m),m').
\end{align*}
It follows that
\begin{align}\label{Traslation}
&\Upsilon_{\mu,s}Y^{\mu,t}(b,z)=Y^{\overline\mu+2s,t+s-\sqrt{-1}\Im(\mu)}(b,z)\Upsilon_{\mu,s}
\end{align}
In particular, if $\mu$ is real,
\begin{align}\label{translmureal}
\Upsilon_{\mu,s}Y^{\mu,t}(b,z)=
Y^{\mu+2s,t+s}(b,z)\Upsilon_{\mu,s}
\end{align}

\begin{lemma}\label{invgeneral}
If $m,m'\in M(\mu)$ and $b\in V^1(\C a)$, then
\begin{equation}\label{invmuts}
H_\mu(m,Y^{\mu,t}(b,z)m')=H_\mu(Y^{\mu,s}(A(-\sqrt{-1}\Im(\mu)+t+s,z)b,z^{-1})m,m').
\end{equation}

In particular, if $b$ is quasiprimary for $L(-\sqrt{-1}\Im(\mu)+t+s)$, then
\begin{equation}\label{mutprimary}
H_\mu(m,b^{\mu,t}_nm')=H_\mu(g(b)^{\mu,s}_{-n}m,m').
\end{equation}
\end{lemma}
\begin{proof}We first  prove that
\begin{equation}\label{invmut}
H_\mu(m,Y^{\mu,t}(b,z)m')=H_\mu(Y^{\mu,t}(A(-\sqrt{-1}\Im(\mu)+2t,z)b,z^{-1})m,m').
\end{equation} 
Indeed,
\begin{align*}
H_\mu(m,Y^{\mu,t}(b,z)m')&=H_{\overline{\mu}+2t}(\Upsilon_{\mu,t}(m),\Upsilon_{\mu,t}(Y^{\mu,t}(b,z)m'))\\
&=H_{\overline{\mu}+2t}(\Upsilon_{\mu,t}(m),Y^{\overline{\mu}+2t}(b,z)\Upsilon_{\mu,t}(m'))\\
&=H_{\overline{\mu}+2t}(Y^{\overline{\mu}+2t}(A(-\sqrt{-1}\Im(\mu)+2t,z)b,z^{-1})\Upsilon_{\mu,t}(m),\Upsilon_{\mu,t}(m'))\\
&=H_{\overline{\mu}+2t}(\Upsilon_{\mu,t}(Y^{\mu,t}(A(-\sqrt{-1}\Im(\mu)+2t,z)b,z^{-1})m),\Upsilon_{\mu,t}(m')),
\end{align*}
so \eqref{invmut} follows.

To prove \eqref{invmuts} write
$$
H_\mu(m,Y^{\mu,t}(b,z)m')=H_\mu(m,Y^{\mu,s}(\prod_{n=1}^{\infty}e^{\frac{-2(t-s)}{n}(-z)^{-n}a_n}b,z)m').
$$
By \eqref{invmut}, setting $s_0=-\sqrt{-1}\Im(\mu)+2s$,
\begin{align*}
H_\mu(m,Y^{\mu,t}(b,z)m')&=H_\mu(Y^{\mu,s}(A(s_0,z)\prod_{n=1}^{\infty}e^{\frac{-2(t-s)}{n}(-z)^{-n}a_n}b,z^{-1})m,m')\\
&=H_\mu(Y^{\mu,s}(e^{zL(s_0)_1}z^{-2L(s_0)_0}g\prod_{n=1}^{\infty}e^{\frac{-2(t-s)}{n}(-z)^{-n}a_n}b,z^{-1})m,m')\\
&=H_\mu(Y^{\mu,s}(e^{zL(s_0)_1}\prod_{n=1}^{\infty}e^{\frac{-2(t-s)}{n}z^{n}a_n}z^{-2L(s_0)_0}g(b),z^{-1})m,m').
\end{align*}
Since, if $p\in\sqrt{-1}\R$, 
$$
e^{z(L(p))_1}=e^{zL(0)_1}\prod_{n=1}^{\infty}e^{-\frac{2p}{n}z^na_n},
$$
we find that
\begin{align*}
H_\mu(m,Y^{\mu,s}(b,z)m')
&=H_\mu(Y^{\mu,s}(e^{zL(0)_1}\prod_{n=1}^{\infty}e^{-\frac{2(-\sqrt{-1}\Im(\mu)+s+t)}{n}z^na_n}z^{-2L(0)_0}g(b),z^{-1})m,m')\\
&=H_\mu(Y^{\mu,s}(e^{zL(-\sqrt{-1}\Im(\mu)+s+t)_1}z^{-2L(0)_0}g(b),z^{-1})m,m')\\
&=H_\mu(Y^{\mu,s}(e^{zL(0)_1}\prod_{n=1}^{\infty}e^{-\frac{2(-\sqrt{-1}\Im(\mu)+s+t)}{n}z^na_n}z^{-2L(0)_0}g(b),z^{-1})m,m')\\
&=H_\mu(Y^{\mu,s}(A(-\sqrt{-1}\Im(\mu)+s+t,z)b,z^{-1})m,m').
\end{align*}

\end{proof}
\begin{example}[The Fairlie construction]\label{FC}
Since $L(s)=\half a_{-1}^2\vac+sa_{-2}\vac$, by \eqref{amut} we have
\begin{align*}
L(s)^{\mu,t}_n&=\half(a_{-1}+2t-2\sqrt{-1}\Im(\mu))^2\vac_n+s(a_{-2}-2t+2\sqrt{-1}\Im(\mu))\vac_n\\
&=\half a_{-1}^2\vac_n+2(t-\sqrt{-1}\Im(\mu))a_{-1}\vac_n+2(t-\sqrt{-1}\Im(\mu))^2\vac_n\\
&+sa_{-2}\vac_n-2s(t-\sqrt{-1}\Im(\mu))\vac_n\\
&=\half :aa:_n+2(t-\sqrt{-1}\Im(\mu))a_n+s(Ta)_n
+2(t-\sqrt{-1}\Im(\mu))(t-\sqrt{-1}\Im(\mu)-s)\vac_n.
\end{align*}
In other words
\begin{equation}\label{Lsmut}
L(s)^{\mu,t}_n=\half :aa:^\mu_n+s(Ta)^\mu_n+2(t-\sqrt{-1}\Im(\mu))a^\mu_n+2(t-\sqrt{-1}\Im(\mu))(t-\sqrt{-1}\Im(\mu)-s)\vac_n^\mu
\end{equation}
In particular, if $\mu\in \R$, we have
\begin{equation}\label{L(s)real}
L(s)^{\mu,t}_n=\half :aa:^\mu_n+s(Ta)^\mu_n+2ta^\mu_n+2(t^2-st)\vac^\mu_n,
\end{equation}
and,  setting $s=2t$, \eqref{L(s)real} becomes
\begin{equation}\label{628}
L(s)^{\mu,s/2}_n=\half :aa:^\mu_n+s(Ta)^\mu_n+sa^\mu_n-\tfrac{1}{2}s^2\vac^\mu_n=L(s)^\mu_n+sa^\mu_n-\tfrac{1}{2}s^2\vac^\mu_n.
\end{equation}
By the $-1$--st product identity,

$$:aa:^{\mu}_n=\begin{cases}\sum_{j\in \mathbb Z}a_{-j}^\mu a_{j+n}^\mu \quad&\text{if $n\ne 0$,}\\2\sum_{j\in \nat}a_{-j}^\mu a_j^\mu+(a_0^\mu)^2 \quad&\text{if $n=0$.}\end{cases}
$$
Moreover, $(Ta)^{\mu}_n=-(n+1)a^\mu_n$, hence, substituting in \eqref{628}, we obtain
$$
L(s)^{\mu,s/2}_n=\frac{1}{2}\sum_{j\in\ZZ}a^\mu_{-j}a^\mu_{j+n}-s na^\mu_n\ \text{    if $n\ne0$,}
$$
 while
$$
L(s)^{\mu,s/2}_0=\sum_{j\in\nat}a^\mu_{-j}a^\mu_{j}+\frac{\mu^2-s^2}{2}I.
$$
Since $b\mapsto Y^{\mu,s/2}(b,z)$ gives a $V^1(\C a)$--module structure to $M(\mu)$
and
$$
[L(s)_\l L(s)]=(T+2\l)L(s)+\frac{\l^3}{12}(1-12s^2),
$$
by the  Borcherds commutator formula, 
$$
[L(s)^{\mu,s/2}_n,L(s)^{\mu,s/2}_m]=(n-m)L(s)^{\mu,s/2}_{n+m}+\frac{n^3-n}{12}(1-12s^2)\d_{n,-m}.
$$
Finally, since $L(s)$ is quasiprimary for $L(s)$ and $g(L(s))=L(s)$, by \eqref{mutprimary} we have
$$
H_\mu(m,L(s)^{\mu,s/2}_nm')=H_\mu(L(s)^{\mu,s/2}_{-n}m,m').
$$
%\subsection{Tensor product}\label{TP} 
\vskip10pt
We now extend the previous analysis of invariant Hermitian forms on bosons to the case of  the vertex algebra $V^1(\C a)\otimes V$
where   $V$ is a conformal vertex algebra.\par
Let $\widehat L$ be the  conformal vector of $V$. Set 
\begin{equation}\label{nuovovir}
\widehat L(s)=L(s)+\widehat L.
\end{equation}

If $M$ is a $V$--module, then $M(\mu)\otimes M$ is a $V^1(\C a)\otimes V$--module and, if $M$ is equipped with a $\widehat L$--invariant form $(\,.\,,.\,)$, then $H_\mu(\,.\,,.\,)\otimes (\,.\,,.\,)$ is a $\widehat L(\sqrt{-1}\Im(\mu))$--invariant form on $M(\mu)\otimes M$ that we keep denoting by $H_\mu(\,.\,,.\,)$.
 
 The arguments developed  in  this section for  $V^1(\C a)$ can be carried out in the same way in the more general setting of the vertex algebra 
 \begin{equation}\label{Vref}V^1(\C a)\otimes V,\end{equation} where $V$ is any conformal vertex algebra. In particular, we have 
 \begin{proposition}\label{FCwithV}If $b\in V^1(\C a)\otimes V$ and $M$ is a $V$--module, then
the fields
$$
Y^{\mu,t}(b,z)=Y^\mu(\prod_{n=1}^{\infty}e^{\frac{2(-t+\sqrt{-1}\Im(\mu))}{n}(-z)^{-n}a_n}b,z)
$$
define a $V^1(\C a)\otimes V$--module structure on $M(\mu)\otimes M$. 
\end{proposition}

As before, we can regard $b\in V^1(\C a)\otimes V$ as a polynomial $b(\mathbf a)$ with values in $V$. More precisely, we write
$$
b(\mathbf a)=\sum_{I}  \mathbf a^I\otimes c_I,\ c_I\in V.
$$
The following is the generalization of Lemma \ref{atmun}. The proof is the same.
\begin{lemma}\label{atmuntV}Write $Y^{\mu,t}(b,z)=\sum_{r\in -\D_b+\ZZ}b^{\mu,t}_rz^{-r-\D_b}$. Then
$$b^{\mu,t}_r=  b(\mathbf a+2(-t+\sqrt{-1}\Im(\mu))\rho(-1))^\mu_{r}.     
$$                
\end{lemma}
\end{example}

\section{Minimal $W$--algebras}\label{7}
\subsection{$\l$--brackets and conjugate linear involutions} Let, as before, $\g$ be a basic classical Lie superalgebra, and $x\in\g$ be an element, for 
which $ad\, x$ is diagonalizable with eigenvalues in $\tfrac{1}{2}\mathbb Z$, the $ad\,x$--gradation of $\g$ satisfies  \eqref{seconda} with some $f\in\g_{-1}$ and is compatible with the parity of $\g$. Then for some $e\in\g_1$, $\{e,x,f\}$ is an $sl_2$--triple as in Proposition 
 \ref{constructphi}, i.e. \eqref{goss} holds with $\g^\natural$ the centralizer of $f$ in $\g$.
 Recall that the invariant bilinear form $(.|.)$ on $\g$ is normalized by the condition $(x|x)=\tfrac{1}{2}$, and we have the orthogonal direct sum of ideals
 \begin{equation}\label{ods}
 \g_0=\C x\oplus \g^\natural.\end{equation}
 Choose a Cartan subalgebra $\h^\natural$ of $\g^\natural$, so that, by \eqref{ods}, 
$\h=\C x\oplus \h^\natural$
is a Cartan subalgebra of $\g_0$ (and of $\g$).\par 

%Let $\D$ be the set  of roots of $\g$ with respect to $\h$. Let $\{\pm\theta\}$ be the set of roots of $\mathfrak s$. Fix a set $\Dp$ of positive roots on $\D$ in such a way that  $\theta\in\Dp$ and $e\in \g_\theta$ (hence $f\in\g_{-\theta}$). Note that $x=\tfrac{1}{2} h_\theta$ and  
%Since $-\theta$ is a minimal even root,  $ad\, x$ defines on $\g$ a {\it minimal} grading
%\begin{equation}\label{mg}
%\g=\C f\oplus\g_{-1/2}\oplus \g_0\oplus \g_{1/2}\oplus\C e. 
%end{equation}
%Recall that the dual Coxeter number of a simple Lie algebra is half of the eigenvalue of its Casimir element for the invariant bilinear form normalized by the condition that the squared length of a long root is $2$. For the basic Lie superalgebras it is defined in a similar way, except that one needs, in addition, to have a choice of positive roots with the highest root $\theta$ even, and the   bilinear form  is normalized by the condition $(\theta | \theta)=2$. Its values on simple roots are given in Table 1.
% that $(.|.)_{|\mathfrak s}$ is non-degenerate, hence $(x|x)\ne 0$ and therefore we have the orthogonal direct sum of %reductive subalgebras
% \begin{equation}\label{ods}
% \g_0=\C x\oplus \g^\natural.
% \end{equation}
%Denoting by $\n_{0+}$ and  $\n_{0-}$ the sum of all root spaces corresponding to $\Delta_{0+}$,
%$-\Delta_{0+}$, we obtain the
%triangular decomposition
%$
 % \g = \n_- \oplus \h \oplus \n_+ \, \quad
 % \g_0 = \n_{0-} \oplus \h \oplus \n_{0+} \, .
%$

Let 
\begin{equation}\label{gnatural}\g^\natural=\bigoplus_{i=0}^s\g^\natural_i\end{equation} be  the decomposition of $\g^\natural$ into the direct sum of  ideals, where $\g_0^\natural$  is the center and the $\g^\natural_i$ are simple  for $i>0$. Let $h^\vee$ be the dual Coxeter number of $\g$, and denote by  $\bar h^\vee_i$  half of the eigenvalue of the Casimir element of $\g^\natural _i$ with respect to $(. |.)_{|\g^\natural_i\times\g^\natural_i}$, when acting on $\g^\natural _i$. Note that $\bar h_0^\vee=0$. \par
In  \cite{KRW} the authors introduced (as a special case of a more general construction) the universal {\sl minimal} $W$--algebra $W^k_{\min}(\g)$, whose simple quotient is $W_k^{\min}(\g)$, attached to the grading \eqref{mg}. This is a vertex algebra  strongly and freely generated by elements $L$, $J^{\{v\}}$ where $v$ runs over a basis of $\g^\natural$, $G^{\{u\}}$ where $u$ runs over a basis of $\g_{-1/2}$, with the following  $\lambda$--brackets  (\cite[Theorem 5.1]{KW1}): $L$ is a Virasoro element  (conformal vector) with central charge $c(k)$ given by \eqref{cc},
 $J^{\{u\}}$ are primary of conformal weight $1$, $G^{\{v\}}$ are primary of conformal weight $\frac{3}{2}$, and 
\begin{align}
\label{primarel}[{J^{\{u\}}}_\l G^{\{v\}}]&=G^{\{[u,v]\}} &&\text{ for  $u\in \g^\natural , \ v\in \g_{-1/2}$,}\\
\label{secondarel}[{J^{\{u\}}}_\l J^{\{v\}}]&=J^{\{[u,v]\}}+\l\beta_k(u|v) &&\text{ for $u,v \in \g^\natural$,}
\end{align}
where
\begin{equation}
\label{eq:5.17}\beta_k (u,v) =\d_{i,j}(k+ \tfrac{h^\vee-\bar h_i^\vee}{2})(u|v),\quad u\in\g^\natural_i,\ v\in\g^\natural_j,\,i,j\ge 0.
\end{equation}
Furthermore, the most  explicit formula for the $\l$--bracket between the $G^{\{u\}}$ is given in \cite[(1.1)]{AKMPP} and in \cite[Theorem 5.1 (e)]{KW1}. We will need both formulas:
 \begin{align}\label{GGsimplifiedfurther2}
 [{G^{\{u\}}}_{\lambda}G^{\{v\}}]=&-2(k+h^\vee)\langle u,v\rangle L+\langle u,v\rangle\sum_{\alpha=1}^{\dim \g^\natural} 
:J^{\{u^\alpha\}}J^{\{u_\alpha\}}:+\\\notag
&\sum_{\gamma=1}^{\dim\g_{1/2}}:J^{\{[u,w^{\gamma}]^\natural\}}J^{\{[w_\gamma,v]^\natural\}}:
+2(k+1)\partial J^{\{[[e,u],v]^\natural\}}\\\notag
&+ 4 \l  \sum_{i} \frac{p(k)}{k_i} J^{\{[[e, u],v]_i^\natural\}}+2\l^2\langle u,v\rangle p(k)\vac,
%\end{align}
%%%%%%%%%%%%%%%
%&[{G^u}_\l G^v]=\\&C\left(\tfrac{1}{4p(k)}\langle u,v\rangle\sum_i :J^{x_i}J^{x_i}:-\tfrac{1}{4p(k)}\sum_{i, j}(\langle[u,x_i],[v,x_j]\rangle+\langle[u,x_j],[v,x_i]\rangle):J^{x_i}J^{x_j}:\right)+\notag\\
%&C\left(-\tfrac{k+h^\vee}{2p(k)}\langle u,v\rangle L+\sum_r\tfrac{1}{k+\tfrac{h^\vee-\bar h_r^\vee}{2}}(\tfrac{1}{2}TJ^{[[x_\theta,u],v]_r^\natural}+\l J^{[[x_\theta,v],v]_r^\natural})+\tfrac{\l^2}{2}\langle u,v\rangle\right).\notag
%&\begin{align}\label{GGsimplifiedfurther}
%&[{G^{\{u\}}}_{\lambda}G^{\{v\}}]
\end{align}
\begin{align}[{G^{\{u\}}}_\l G^{\{v\}}]=&-2(k+h^\vee)\langle u,v\rangle L+\langle u,v\rangle\sum_{\alpha=1}^{\dim \g^\natural} \label{GGsimplifiedfurther}
:J^{\{u^\alpha\}}J^{\{u_\alpha\}}:+\\\notag
&2\sum_{\a,\be}\langle[u_\alpha,u],[v,u^\be]\rangle:J^{\{u^\a\}}J^{\{u_\be\}}:
 +2(k+1) (\partial+ 2\lambda) J^{\{[[e,u],v]^{\natural}\}}\\\notag
&+ 2 \lambda \sum_{\a,\be}\langle[u_\alpha,u],[v,u^\be]\rangle
 J^{\{ [u^\a,u_\be]\}}
       + 2\l^2\langle u,v\rangle p(k)\vac,
\end{align}
where $\{u_\a\}$ and $\{u^\a\}$ (resp. $\{w_\gamma\}, \{w^\gamma\}$) are dual bases of $\g^\natural$ (resp. $\g_{1/2}$) with respect to $(.|.)$ (resp. with respect to $\langle\cdot,\cdot\rangle_{\rm ne}$), $a\mapsto a^\natural_i$ (resp. $a\mapsto a^\natural$) for  $a \in\g_0$  is the orthogonal projection to $\g^\natural_i$ (resp $\g^\natural$), 
 $p(k)$ is  the  monic quadratic polynomial proportional to  \eqref{pol},  introduced in \cite[Table 4]{AKMPP}, and thoroughly investigated in \cite{Y}, and $k_i= k + \tfrac{1}{2}(h^\vee-\bar h^\vee_i)$, $i=1,\ldots,s$ (see Table 2 below for the values of $\bar h^\vee_i$).
\vskip5pt

The following proposition is a special case of \cite[Lemma 7.3]{KMP}, in view of Lemma \ref{31}.
\begin{prop}\label{descend}
Let $\phi$ be a conjugate linear involution of $\g$ such that
$
\phi(f)=f,\ \phi(x)=x,\ \phi(e)=e.
$
Then the map
\begin{equation}\label{phiespilcita} \phi(J^{\{u\}})= J^{\{\phi(u)\}},\quad \phi(G^{\{v\}})=G^{\{\phi(v)\}},\quad \phi(L)=L,\ \ u\in \g^\natural,\ v\in \g_{-1/2}\end{equation}
 extends to
 a  conjugate linear involution of the vertex algebra $W^k_{\min}(\g)$.
\end{prop}

The following result is a sort of converse to Proposition \ref{descend}.
\begin{proposition} \label{converse}Assume that $k\in\R$ is non-collapsing.
 Let $\psi$ be a conjugate linear  involution of $W^k_{\min}(\g)$. Then there exists a conjugate linear  involution $\phi$ of $\g$ satisfying \eqref{primas} such that  $\psi$ is the conjugate linear  involution induced by $\phi$.
\end{proposition}
\begin{proof} If $a,b\in\g^\natural$, define $\phi(a)$ by 
$$\psi(J^{\{a\}})=J^{\{\phi(a)\}}.$$
Then 
\begin{align}
&\psi([{J^{\{a\}}}_\l J^{\{b\}}])=\psi(J^{\{[a,b]\}})+\l\overline{\beta_k(a,b)}=J^{\{\phi([a,b])\}}+\l\overline{\beta_k(a,b)}\label{pm}\\
&[{J^{\{\phi(a)\}}}_\l J^{\{\phi(b)\}}]=J^{\{[\phi(a),\phi(b)]\}}+\l\beta_k(\phi(a),\phi(b))\label{sm}
\end{align}
Since $\psi$ is a vertex algebra conjugate linear  automorphism, \eqref{pm} equals \eqref{sm}, so that $\phi$ is a conjugate  linear involution of $\g^\natural$,
and we have 
 \begin{align}\label{uguale}\overline{\beta_k(a,b)}=\beta_k(\phi(a),\phi(b)).
\end{align} 
%The explicit expression of $p(k)$ given in \cite[\S5]{Y} shows that  the linear polynomial appearing in \eqref{eq:5.17} is a factor of $p(k)$.  
Since $k$ is not collapsing, relations \eqref{abbbb}, \eqref{pol} and  \eqref{uguale} imply that 
\begin{align}\label{ugualee}\overline{(a|b)}=(\phi(a)|\phi(b))\text{ for $a,b\in\g^\natural$.}
\end{align}
\par
 We now prove that there is a unique extension 
of $\phi$ to a conjugate linear  automorphism of $\g$ fixing $e$, $x$, and $f$. Note that  $\phi(\g_{-1/2})\subset \g_{-1/2}$ and that $\g_{1/2}=[e,\g_{-1/2}]$. In particular, setting $\phi(x)=x,\, \phi(f)=f,\,  \phi(e)=e,$ $\phi(u)=[e,\phi(v)]$  for $u\in \g_{1/2},\,  u=[e,v],\,  v\in \g_{-1/2}$, 
we extend $\phi$ to  a conjugate  linear bijection $\g\to\g$. In particular, $\phi$ is unique. It remains to prove that it is a conjugate linear  automorphism. Note first that, by \eqref{ods},  equation 	\eqref{ugualee} holds for $a,b\in\g_0$.
% Next we prove that $\phi$ is an automorphism of $\g$. 
Consider elements
%$$g=\ggg\,e + u+ a + v+ \d\,f +\s\,x,\quad g'=\ggg'\,e + u'+ a '+ v'+ \d'\,f +\s'\,x$$
%$\ggg,\ggg',\d,\d',\s,\s'\in\C,\, u,u'\in\g_{1/2},\,v,v'\in \g_{-1/2},\, a,a'\in\g^\natural$. Then 
%\begin{align}\label{aaaaa}\phi([g,g'])&=\phi([e,v']+\ggg\d'x-\ggg\s'e+[u,u']+[u,a']+[u,v']+\d'[u,f]\\&-\tfrac{1}{2}\s' u+[a,u']+[a,a']+[a,v']+
%\ggg'[v,e]+[v,u']+[v,a']+[v,v']\notag\\&+\tfrac{1}{2}\s' v-\d\ggg' x+\d [f,u']  +\d \s'f+\s\ggg' e+\tfrac{1}{2}\s u'-\tfrac{1}{2}\s v'-\d'\s f)\notag
%\end{align}
%\begin{align}\label{bbbbb}&[\phi(g),\phi(g')]=\\&\bar\ggg [e,\phi(v')]+\bar\ggg\bar\d'x-\bar\ggg\bar\s'e+[\phi(u),\phi(u')]+[\phi(u),\phi(a')]+[\phi(u),\phi(v')]+\bar\d'[\phi(v),f]
%\notag\\&-\tfrac{1}{2}\bar\s'\phi(u)+[\phi(a),\phi(u')]+[\phi(a),\phi(a')]+[\phi(a),\phi(v')]+\notag
%\bar\ggg'[\phi(v),e]+[\phi(v),\phi(u')]\\&+[\phi(v),\phi(a')]+[\phi(v),\phi(v')]+\tfrac{1}{2}\bar\s' \phi(v)-\bar\d\bar\ggg' x+\d [f,\phi(u')]+\d \s'f+%\bar\s\bar\ggg e+\tfrac{1}{2}\bar\s \phi(u')\notag
%\&-\tfrac{1}{2}\bar\s \phi(v')-\bar\d'\bar\s f)\notag\end{align}
$$g=\a\,e + u+ a + v+ \beta\,f +\gamma\,x,\quad g'=\a'\,e + u'+ a '+ v'+ \beta'\,f +\gamma'\,x,$$ where
$\a,\a',\beta,\beta',\gamma,\gamma'\in\C,\, u,u'\in\g_{1/2},\,v,v'\in \g_{-1/2},\, a,a'\in\g^\natural$. Then 
\begin{align}\label{aaaaa}&\phi([g,g'])=\\&\notag\phi([e,v']+\a\beta'x-\a\gamma'e+[u,u']+[u,a']+[u,v']+\beta'[u,f]\\&-\tfrac{1}{2}\gamma' u+[a,u']+[a,a']+[a,v']+
\a'[v,e]+[v,u']+[v,a']+[v,v']\notag\\&+\tfrac{1}{2}\gamma' v-\beta\a' x+\beta [f,u']  +\beta \gamma'f+\gamma\a' e+\tfrac{1}{2}\gamma u'-\tfrac{1}{2}\gamma v'-\beta'\gamma f),\notag
\\\label{bbbbb}&[\phi(g),\phi(g')]=\\&\bar\a [e,\phi(v')]+\bar\a\bar\beta'x-\bar\a\bar\gamma'e+[\phi(u),\phi(u')]+[\phi(u),\phi(a')]+[\phi(u),\phi(v')]+\bar\beta'[\phi(v),f]
\notag\\&-\tfrac{1}{2}\bar\gamma'\phi(u)+[\phi(a),\phi(u')]+[\phi(a),\phi(a')]+[\phi(a),\phi(v')]+\notag
\bar\a'[\phi(v),e]+[\phi(v),\phi(u')]\\&+[\phi(v),\phi(a')]+[\phi(v),\phi(v')]+\tfrac{1}{2}\bar\gamma' \phi(v)-\bar\beta\bar\a' x+\beta [f,\phi(u')]+\beta \gamma'f+\bar\gamma\bar\a e\notag\\&+\tfrac{1}{2}\bar\gamma \phi(u')
-\tfrac{1}{2}\bar\gamma \phi(v')-\bar\beta'\bar\gamma f).\notag\end{align}
Hence \eqref{aaaaa} equals \eqref{bbbbb}, provided  the following equalities hold
\begin{align}
\label{w1}&\phi([u,u'])=[\phi(u),\phi(u')],\\
\label{w2}&\phi([u,a'])=[\phi(u),\phi(a')],\\
\label{w3}&\phi([u,v'])=[\phi(u),\phi(v')],\\
\label{w4}&\phi([v,v'])=[\phi(v),\phi(v')],\\
\label{w44}&\phi([v,a'])=[\phi(v),\phi(a')],\\
\label{w5}&\phi([u,f])=[\phi(u),f].
\end{align}

Relation  \eqref{primarel} implies at once \eqref{w44}. To prove \eqref{w4} note that $[v,v']=\langle v,v'\rangle f$, so it is enough to prove that 
$\langle \phi(v),\phi(v')\rangle=\overline{ \langle v,v'\rangle}$. By \eqref{GGsimplifiedfurther}, $${G^{\{\phi(v)\}}}_{3/2}G^{\{\phi(v')\}}=
4p(k)\langle\phi(v),\phi(v') \rangle\vac=\phi(4p(k)\langle v,v'\rangle \vac)=4p(k)\overline{\langle v,v'\rangle}\vac.$$ Since $p(k)\ne 0$ 
($k$ is not collapsing)  and $k$ is real, we have the claim. \par Now we prove \eqref{w5}. Here and in the following we write  $u=[e,v],\, v\in\g_{-1/2}$. Then 
$$\phi([u,f])=\phi([[e,v],f])=-\phi([x,v])=\tfrac{1}{2}\phi(v)=-[x,\phi(v)]=[[e,\phi(v)],f]=[\phi(u),f].$$\par
Next we prove \eqref{w3}.
%write as usual $u=[e,v],\,v\in\g_{-1/2}$, hence 
We have to prove that
$$\phi([[e,v],v'])=[[e,\phi(v)],\phi(v')].$$
 By \eqref{GGsimplifiedfurther2}
$${G^{\{\phi(v)\}}}_{1/2}{G^{\{\phi(v')\}}}=\sum_i\tfrac{p(k)}{k_i}J^{\{[[e,\phi(v)],\phi(v')]_i^\natural\}}.
$$
On the other hand 
$$\psi({G^{\{v\}}}_{1/2}G^{\{v'\}})=\sum_i\tfrac{p(k)}{k_i}J^{\{\phi([[e,v],v']_i^\natural)\}}.$$
Since $\phi$ is an automorphism of $\g^\natural$ there is a permutation $i\mapsto i'$ such that
$\phi(\g^\natural_i)=\g^\natural_{i'}$. It follows that
$\phi([[e,v],v']_i^\natural)=\phi([[e,v],v'])_{i'}^\natural$
hence $[[e,\phi(v)],\phi(v')]_{i'}^\natural=\phi([[e,v],v']_i^\natural)$ for all $i$,  and also
$$[[e,\phi(v)],\phi(v')]^\natural=\!\sum_{i'}[[e,\phi(v)],\phi(v')]_{i'}^\natural=\!\sum_{i}\phi([[e,v],v']_i^\natural)\!=
\sum_{i'}\phi([[e,v],v'])_{i'}^\natural=\phi([[e,v],v'])^\natural.
$$
To conclude we have to check the $x$--component:
$$(x|[[e,\phi(v)],\phi(v')])=([x,[e,\phi(v)]]|\phi(v'))=\tfrac{1}{2}([e,\phi(v)]|\phi(v'))=\tfrac{1}{2}\langle\phi(v),\phi(v')\rangle=
\tfrac{1}{2}\overline{\langle v , v'\rangle}.
$$
Since  \eqref{primas} holds on $\g_0$, we have
$$(x|\phi([[e,v],v'])=\overline{(\phi(x)|[[e,v],v'])}=\overline{(x|[[e,v],v'])}=\tfrac{1}{2}\overline{\langle v , v'\rangle}.
$$
Next, we  prove \eqref{w2}. We have 
$$\phi([[e,v],a'])=\phi([e,[v,a']])=[e,\phi([v,a'])]\!=\![e,[\phi(v),\phi(a')]]=[[e,\phi(v)],\phi(a')]=[\phi(u),\phi(a')].
$$
Next,  we prove \eqref{w1}. Consider  $u=[e,v],\, u'=[e,v'],\,v,v'\in\g_{-1/2}$.
$$\phi([u,u'])=\phi([[e,v],[e,v']])=\phi([e,[[e,v],v']])=[e,\phi([[e,v],v'])].
$$
By \eqref{w3}, we obtain 
$$\phi([u,u'])=[e,[\phi([e,v]),\phi(v')])]=[e,[\phi(u),\phi(v')]]=[\phi(u),\phi(u')].
$$

It remains to check that 
$$
(\phi(a)|\phi(b))=\overline{(a|b)}
$$
for $a,b\in\g$. We already observed that this relation holds for $a,b\in\g_0$ and it is obvious that $(\phi(e)|\phi(f))=\overline{(e|f)}$. We now compute for $u\in\g_{1/2}$, $v'\in\g_{-1/2}$,
$$
(\phi(u)|\phi(v'))=([e,\phi(v)]|\phi(v'))=\langle \phi(v),\phi(v')\rangle=\overline{\langle v, v'\rangle}=\overline{(u|v')}.
$$
\end{proof}

By Proposition \ref{constructphi}  there is a    conjugate linear involution $ \phi$ on $\g$ such that $\phi(x)=x,\,\phi(f)=f$ and $(\g^\natural)^{\phi}$ is a compact real form of $\g^\natural$, hence, by Proposition \ref{descend}, $\phi$ induces a  conjugate linear involution of the vertex algebra $W^k_{\min}(\g)$, and descends to  a  conjugate linear involution  of its unique simple quotient $W_k^{\min}(\g)$, which we again denote by $\phi$.

  By \cite[Proposition 7.4 (b)]{KMP}, 
$W^k_{\min}(\g)$ admits a unique $\phi$--invariant Hermitian form $H(\cdot,\cdot)$ such that $H(\vac,\vac)=1$. Recall that if
$k+h^\vee\ne 0$ then the kernel of $H(\cdot,\cdot)$  is the unique maximal ideal of  $W^k_{\min}(\g)$, hence $H(\cdot,\cdot)$  descends to a non-degenerate $\phi$--invariant Hermitian form on $W_k^{\min}(\g)$, which we again denote by $H(\cdot,\cdot)$. 
\vskip5pt
%\begin{definition} We say that $W_k^{\min}(\g)$ is {\sl $\psi$-unitary} if the $\psi$-invariant Hermitian form $H(\cdot,\cdot)$  is positive definite.  We say that $W_k^{\min}(\g)$ is {\it unitary} if it  is {\sl $\psi$-unitary} for some conjugate linear automorphism $\psi$. \end{definition}
%\begin{remark} We also say that 
%$W^k_{\min}(\g)$ is $\psi$-unitary if  $H(\cdot,\cdot)$ is positive semidefinite. It is clear, if $k\ne - h^\vee$, that  if $W^k_{\min}(\g)$ is $\psi$-unitary then $W_k^{\min}(\g)$ is $\psi$-unitary.
%\end{remark}
%In the following we investigate the unitarity of  $W_k^{\min}(\g)$ (cf. Definition \ref{unidef}). With a slight abuse of terminology, we also say that 
%$W^k_{\min}(\g)$ is  unitary if  $H(\cdot,\cdot)$ is positive semidefinite. It is clear that
%, for   $k\ne - h^\vee$, 
% $W^k_{\min}(\g)$ is  unitary if and only if $W_k^{\min}(\g)$ is unitary.

We need  to fix notation  for affine vertex algebras.
Let $\aa$ be a Lie superalgebra equipped with a nondegenerate  invariant supersymmetric bilinear form $B$. The universal affine vertex algebra $V^B(\aa)$ is  the universal enveloping vertex algebra of  the  Lie conformal superalgebra $R=(\C[T]\otimes\aa)\oplus\C$ with $\lambda$--bracket given by
$$
[a_\lambda b]=[a,b]+\lambda B(a,b),\ a,b\in\aa.
$$
In the following, we shall say that a vertex algebra $V$ is an affine vertex algebra if it is a quotient of some $V^B(\aa)$. If $\aa$ is simple Lie algebra, we denote by $(.|.)^\aa$ the normalized invariant bilinear form on $\aa$, defined by the condition $(\a|\a)^\aa=2$ for a long root $\a$. Then  $B=k(.|.)^\aa$, and  we simply write $V^k(\aa)$.  If $k\ne -h^\vee$, then  $V^k(\aa)$ has a unique simple quotient, which will be denoted by 
$V_k(\aa)$.

Let $\psi$ be a conjugate linear involution of $\aa$ such that $
(\psi(x) |\psi(y))=\overline{(x|y)}$. By \cite[\S5.3]{KMP} there exists a unique $\psi$--invariant Hermitian form $H_\aa$ on $V^k(\aa)$.  
%If $k\ne - h^\vee$, then 
The kernel of  $H_\aa$ is the maximal ideal of $V^k(\aa)$, hence $H_\aa$ descends to $V_k(\aa)$.\par

\subsection{Some numerical information}
Recall the decomposition \eqref{gnatural}  of the Lie algebra $\g^\natural$, and that  we assume that $\g^\natural$ is not abelian, i.e. $s\ge 1$ in \eqref{gnatural}. Let $\theta_i$ be the highest root of the simple component $\g^\natural_i$ for $i>0$. Set 
\begin{equation}\label{MIK}M_i(k)=
\frac{2}{u_i}\left(k+\frac{h^\vee-\bar h^\vee_i}{2}\right),\quad i\ge 0,\end{equation}
where
$$u_i=\begin{cases}2\quad&\text{if $i=0$,}\\
(\theta_i|\theta_i)\quad&\text{if $i>0$.}
\end{cases}
$$
Let  $(.|.)_i^\natural$ denote the  invariant bilinear form on $\g^\natural_i$, normalized by the condition $(\theta_i|\theta_i)_i^\natural=2$ for $i>0$,  and let 
$(.|.)_0^\natural=(.|.)_{|\g_0^\natural\times\g_0^\natural}$.  Note that, for $i>0$,  $(a|b)_i^\natural= \d_{i,j}\frac{(\theta_i|\theta_i)}{2}(a|b)$,  hence,
formula \eqref{eq:5.17} can be written as 
\begin{align}
\label{abbbb}
\beta_k(a, b)&=\d_{i,j} M_i(k)\frac{(\theta_i|\theta_i)}{2}(a|b)\\
\label{ab}&=\d_{i,j} M_i(k)(a|b)_i^\natural  \qquad\text{ for $a\in \g^\natural_i,\ b\in \g^\natural_j,\ i,j\ge 0$.}
\end{align}
In other words, the vertex subalgebra of $W^k_{\min}$ generated by $J^{\{a\}},\, a\in\g^\natural$, is $\bigotimes\limits_{i\ge 0} V^{M_i(k)}(\g_i^\natural)$.\par
Closely related to the vertex algebra $ W^k_{\min}(\g)$ is the universal affine vertex algebra $V^{\a_k}(\g_0)$ 
%with $\lambda$-bracket 
(see \cite[(5.16)]{KW1}), 
%$$[a_\lambda b]=[a,b]+\l\a_k(a,b),$$
where
\begin{equation}
\label{eq:5.16}
  \alpha_k (a, b) = ((k+h^\vee)(a|b)-\tfrac{1}{2}
\kappa_{\fg_0}(a,b))\,,
\end{equation}
and where $\kappa_{\fg_0}$ denotes the Killing form of $\fg_0$. Note that $$\a_k(a,b)= \d_{i,j}(k+h^\vee-\bar h_i^\vee)(a|b)\text{ if } a\in\g^\natural_i,\ b\in \g^\natural_j,\,i,j\ge 0.$$
\par
We have another formula for the cocycle $\a_k,$ closely related to \eqref{ab}:
\begin{equation}
\label{abb}
\a_k(a, b)=\d_{i,j} \frac{2}{(\theta_i|\theta_i)}\left(k+h^\vee-\bar h^\vee_i\right)(a|b)_i^\natural= \d_{i,j} (M_i(k)+\chi_i)(a|b)_i^\natural \text{ for $a\in \g^\natural_i, b\in \g^\natural_j,\ i,j\ge 0$,}
\end{equation}
where 
\begin{equation}\label{chii}\chi_i=\frac{h^\vee-\bar h^\vee_i}{u_i},\ i\geq 0.\end{equation}
The relevant data for computing the $M_i(k)$ and $\chi_i$ are collected in  Table 2, where their explicit values are   also displayed. Note that $M_0(k)=k+\tfrac{1}{2}h^\vee$.
\begin{table}{\small
\begin{tabular}{   c | c| c| c| c | c | c}
$\g$ & $\g^\natural$ & $u_i$ & $h^\vee$ & $\bar h_i^\vee$ & $M_i(k)$ &$\chi_i$\\\hline
$sl(2|m), m >2 $ & $ \C\oplus sl_m$ & $2,-2$ & $2-m$ & $0,-m$ & $k-\frac{m-2}{2}, -k-1$&$1-m/2,-1$\\\hline
$psl(2|2)$& $sl_2$  & $-2$ &  $0$ & $-2$ & $-k-1$&$-1$\\\hline
$osp(4|m), m>2$ & $sl_2\oplus sp_m$&$2, -4$ & $2-m$ & $ 2, -m-2$ & $k-\frac{m}{2}, -\frac{1}{2}k-1$ &$-m/2,-1$\\\hline
$spo(2|3)$& $sl_2$  &$-1/2$ & $1/2$ & $-1/2$ & $-4k-2$&$-2$\\\hline
$spo(2|m), m>4$ & $so_m$ &  $-1$& $2-m/2$ & $1-m/2$ &  $-2k-1$&$-1$\\\hline
%$D(2,1;a)$ &  $sl_2\oplus sl_2$ & $-2-2a, 2a$&$0$ &  $-2-2a,2a $ & $-\frac{1}{1+a}k-1, \frac{1}{a}k-1$ &$-1,-1$\\\hline
$D(2,1;a)$ &  $sl_2\oplus sl_2$ & $-\tfrac{2}{1+a}, -\tfrac{2a}{1+a}$&$0$ &  $-\tfrac{2}{1+a}, -\tfrac{2a}{1+a} $ & $-(1+a)k-1, -\frac{1+a}{a}k-1$ &$-1,-1$\\\hline
$F(4)$& $so_7$  & $-4/3$ &$-2$ & $-10/3$ & $-\frac{3}{2} k-1$ &$-1$\\\hline
$G(3)$& $G_2$ &$-2/3$  &$-3/2$  &  $-3$ &$ -\frac{4}{3} k-1$&$-1$
\end{tabular}
}
\vskip5pt

\caption{Numerical information}
\end{table}\par
As in the Introduction, denote by $\xi\in(\h^\natural)^*$ a highest weight of the $\g^\natural$--module $\g_{-1/2}$.
\begin{lemma} For $i\geq 1$ we have 
\begin{equation}\label{chiiii}\chi_i=-\xi(\theta_i^\vee),\end{equation} with the exception of $\chi_1$ for $\g=osp(4|m)$.
\end{lemma}
\begin{proof} The weights $\xi$ are restrictions to $\h^\natural$ of the maximal  odd roots of $\g$; they are listed in Table 3, together with the maximal roots $\theta_i$. 
\begin{table}{\small
\begin{tabular}{   c | c| c}
$\g$ & $\text{highest odd roots}$ & $\theta_i$\\\hline
$sl(2|m), m >2 $ & $ \e_1-\d_m, \d_1-\e_2$ & $\d_1-\d_m$\\\hline
$psl(2|2)$& $ \e_1-\d_2, \d_1-\e_2$ & $\d_1-\d_2$\\\hline
$osp(4|m), m>2$ & $\e_1+\d_1$&$\e_1-\e_2,2\d_1$ \\\hline
$spo(2|3)$& $\d_1+\e_1$  &$\e_1$\\\hline
$spo(2|m), m>4$ & $\d_1+\e_1$  &$\e_1+\e_2$\\\hline
$D(2,1;a)$ &$\e_1+\e_2+\e_3$ & $2 \e_2, 2 \e_3$\\\hline
$F(4)$& $\tfrac{1}{2}(\d_1+\e_1+\e_2+\e_3)$ &$\e_1+\e_2$\\\hline
$G(3)$& $\d_1+\e_1+\e_2$  &$\e_1+2\e_2$
\end{tabular}
}
\vskip5pt

\caption{Highest  odd roots and highest  roots of $\g^\natural$}
\end{table}
 Relation \eqref{chiiii} is then checked directly using the data in Tables 1, 2, 3.
\end{proof}

%Recall from \cite{AKMPP} that a level $k$ is {\it collapsing} for $W_k^{\min}(\g)$ if $W_k^{\min}(\g)$ is the unique simple 
%graded quotient of $V^{\beta_k}(\g^\natural)$.  
Recall from \cite{AKMPP} that a level $k$ is {\it collapsing} for $W_k^{\min}(\g)$ if $W_k^{\min}(\g)$ is a subalgebra of the 
simple affine vertex algebra $V_{\beta_k}(\g^\natural)$.

%The main result of \cite{AKMPP} states that $k$ is collapsing if and only if $p(k)=0$.

We summarize in the following result the content  of Theorem 3.3 and Proposition 3.4 of \cite{AKMPP} relevant to our setting.
We say that an ideal in $\g^\natural$ is a {\it component} of $\g^\natural$ if it is simple or $1$--dimensional.
\begin{theorem}\label{oldresults} Let $\g$ be a basic  Lie superalgebra from  Table 2. Assume $k\ne -h^\vee$.  Let 
$p(k)$ be the monic quadratic polynomial in $k$, proportional to 
\begin{equation}\label{pol}\begin{cases} M_1(k)M_2(k)\quad & \text{if $\g^\natural$ has two components,}\\
M_1(k)(k+\frac{\bar h_1^\vee}{2}+1)\quad & \text{otherwise.}
\end{cases}\end{equation}
Then 
\begin{enumerate}
\item $k$ is collapsing if and only if $p(k)=0$.
\item If $\g^\natural$ is simple then 
\begin{enumerate}\item 
$W_k^{\min}(\g)=\C$ if and only if $M_1(k)=0$;
\item if  $k=-\tfrac{\bar h_{1}^\vee}{2}-1,$ then $W^{\min}_k(\g)\cong V_{M_1(k)}(\g^\natural)$. 
\end{enumerate}
\item If $\g=D(2,1;a) $ and  $k$ is collapsing, then 
 $W_k^{\min}(\g)=V_{M_j(k)}(\g^\natural_j)$, with $j\ne i$ if $M_i(k)=0$.
 % In particular, if $\g=spo(2|4)\cong D(2,1;1)$ and $M_1(k)=M_2(k)=0$ then  $W_k^{\min}(\g)=\C$.
\end{enumerate} 
\end{theorem}
%\begin{remark}\label{63} We say that an ideal in $\g^\natural$ is a {\it component} of $\g^\natural$ if it is simple or $1$-dimensional.
%It follows from \cite[Lemma 3.1]{AKMPP},  \cite[Theorem 5.9]{KMP} that, up to a constant factor 
%$$p(k)=\begin{cases} M_1(k)M_2(k)\quad & \text{if $\g^\natural$ has two components,}\\
%M_1(k)(k+\frac{\bar h_1^\vee}{2}+1)\quad & \text{otherwise.}
%\end{cases}$$
%The roots of $p(k)$ are the collapsing levels defined in the Introduction.
%\end{remark}
\begin{remark}\label{7.4} If $M_i(k)\in \mathbb Z_+$ for all $i\geq 1,\,\g\ne osp(4|m)$ and $M_i(k)<-\chi_i$ for some 
$i\geq 1$, then $k$ is a collapsing level (or critical). This is clear by looking at Table 2.
\end{remark}

\section{Necessary conditions for unitarity of modules over $W^k_{\min}(\g)$}\label{Necessary}
We assume that $\g$ is from the list \eqref{ssuper}; in particular, $\g^\natural$ is a reductive Lie algebra.
We parametrize the highest weight modules for $W^k_{\min}(\g)$ following Section 7 of \cite{KW1}.
Let $\h^\natural$ be a Cartan subalgebra of $\g^\natural$, and choose a triangular decomposition
$\g^\natural=\n^\natural_-\oplus\h^\natural\oplus\n^\natural_+$. For  $\nu\in (\h^\natural)^*$ and $l_0\in\C$, let $L^W(\nu,\ell_0)$ (resp.
$M^W(\nu,\ell_0)$ ) denote the irreducible highest weight (resp. Verma) $W^k_{\min}(\g)$--module with highest weight $(\nu,\ell_0)$ and highest weight vector $v_{\nu,\ell_0}$. This means that one has
\begin{align*}
&J^{\{h\}}_0v_{\nu,\ell_0}=\nu(h)v_{\nu, \ell_0} \text{ for } h\in\h^\natural,\quad L_0v_{\nu,\ell_0}=l_0v_{\nu,\ell_0},\\
&J^{\{u\}}_nv_{\nu,\ell_0}=G^{\{v\}}_nv_{\nu,\ell_0}=L_nv_{\nu,\ell_0}=0\text{ for $n>0$, $u\in \g^\natural$},\ v\in\g_{-1/2}, \\
&J^{\{u\}}_0v_{\nu,\ell_0}=0\text{ for } u\in\n^\natural_{+}.
\end{align*}

Let $\phi$ is an almost compact conjugate linear involution of $\g$ (see Definition \ref{ac}); in particular,  the fixed points set $\g^\natural_\R$ of $\phi_{|\g^\natural}$  is a compact Lie algebra (the adjoint group is compact). 
Set $\h^\natural_\R=\g^\natural_\R\cap \h^\natural$. Recall that $\nu\in (\h^\natural_\R)^*$ is said to be purely imaginary if $\nu(\h^\natural_\R)\subset\sqrt{-1}\R$. It is well-known that if 
$\a$ is a root of $\g^\natural$ and $\nu$ is purely imaginary then $\nu(\a)\in \R$.
\begin{lemma}\label{ef} Assume that $l_0\in\R$ and that $\nu$ is purely imaginary. Then 
$L^W(\nu,\ell_0)$ admits a unique $\phi$--invariant nondegenerate Hermitian form $H(\,.\,,.\,)$ such that $H(v_{\nu,\ell_0},v_{\nu,\ell_0})=1$. 
\end{lemma}
\begin{proof}It is enough to show that the Verma module $M^W(\nu,\ell_0)$ admits a $\phi$--invariant Hermitian  form $H$ such that 
$H(v_{\nu,\ell_0},v_{\nu,\ell_0})=1$.
Fix a basis 
 $\{v_i\mid i\in I\}$  of $\g_{-1/2}$ and a basis $\{u_i\mid i\in J\}$  of $\mathfrak n^\natural_{-}$. Set $A^{\{i\}}=J^{\{u_i\}}$ if $i\in J$,  $A^{\{i\}}=G^{\{v_i\}}$ if $i\in I$, and $A^{\{0\}}=L$. Then
\begin{equation*}
  \label{eq:6.13}
\mathcal B=\left\{  \left( A^{\{ 1 \}}_{-m_1}\right)^{b_1} \cdots
   \left( A^{\{ s \}}_{-m_s}\right)^{b_s} v_{\nu,\ell_0}\right\}
\end{equation*}
 where  $b_i \in \ZZ_+ \, , \,
   b_i \leq 1$   if  $i\in I$,  $m_i >0$
  or   $m_i=0$ when  $i\in J$,
is a basis of $M^W(\nu,\ell_0)$. 

Define the conjugate-linear map $F:M\to \C$ by setting $F(v_{\nu,\ell_0})=1$ and $F(v)=0$ if $v\in\mathcal B$, $v\ne v_{\nu,\ell_0}$. 
%We will identify $\h^\natural$ with its dual using the bilinear form $(.|.)$.

If $v\in M^W(\nu,\ell_0)$, $m>0$, and $u\in\g^\natural$, then
$$
(J^{\{u\}}_mF)(v)=-F(J^{\{\phi(u)\}}_{-m}v)=0.
$$
Similarly we see that, if $u\in\g_{-1/2}$, then  
$$
(G^{\{u\}}_mF)(v)=(L_mF)(v)=0.
$$
On the other hand, if $u\in \mathfrak n_{0+}$, then, since $\phi(u)\in\mathfrak n_{0-}$, 
$$
(J^{\{u\}}_0F)(v)=-F(J^{\{\phi(u)\}}_{0}v)=0.
$$
%If $\a$ is a root of $\g^\natural$ then, since $(\nu,\a)\in\R$, 
If $h\in \h_\R^\natural$, then, since $\nu(h)$ is purely imaginary,
$$
(J^{\{h\}}_0F)(v_{\nu,\ell_0})=-F(J^{\{\phi(h)\}}_{0}v_{\nu,\ell_0})=-F(J^{\{h\}}_{0}v_{\nu,\ell_0})=\nu(h)F(v_{\nu,\ell_0}),
$$
and, if $v\in\mathcal B$, $v\ne v_{\nu,\ell_0}$, then 
$$
(J^{\{h\}}_0F)(v)=-F(J^{\{\phi(h)\}}_{0}v)=-F(J^{\{h\}}_{0}v)=0.
$$
It follows that $J^{\{h\}}_0F=\nu(h)F$ for all $h\in\h^\natural$.
Finally, since $l_0\in\R$,
$$
(L_0F)(v_{\nu,\ell_0})=F(L_{0}v_{\nu,\ell_0})=l_0F(v_{\nu,\ell_0}),
$$
and, if $v\in\mathcal B$, $v\ne v_{\nu,\ell_0}$, then 
$$
(L_0F)(v)=F(L_{0}v)=0.
$$
so $L_0F=l_0F$.
It follows that there is a $W^k_{\min}(\g)$--module map  $\beta:M^W(\nu,\ell_0)\to M^W(\nu,\ell_0)^\vee$ mapping $v_{\nu,\ell_0}$ to $F$.
Define a Hermitian form on $M^W(\nu,\ell_0)$ by setting
$$
H(m,m')=\beta(m')(m).
$$
Let us check that this form is $\phi$--invariant: write $Y^{\nu,\ell_0}$ for the field $Y^{M^W(\nu,\ell_0)}$ and $\check Y^{\nu,\ell_0}$ for the field $Y^{M^W(\nu,\ell_0)^\vee}$. Then
\begin{align*}
H(m,Y^{\nu,\ell_0}(u,z)m')&=\beta(Y^{\nu,\ell_0}(u,z)m')(m)=\check Y^{\nu,\ell_0}(u,z)\beta(m')(m)\\&=\beta(m')(Y^{\nu,\ell_0}(A(z)u,z^{-1})m),
\end{align*}
so
$$
H(m,Y^{\nu,\ell_0}(u,z)m')=H(Y^{\nu,\ell_0}(A(z)u,z^{-1})m,m').
$$
\end{proof}
%{\color{red} The sixth column of  Table 2  allows to express $k$ as a function of $M_i(k)$. }
%\begin{proposition}\label{ncw} Assume $k+h^\vee\ne0$. If  there exists a unitary irreducible highest weight module  over $W^k_{\min}(\g)$, then   $M_i(k)\in\ZZ_+$ for all $i$. In particular, if $W_k^{\min}(\g)$ is unitary, then   $M_i(k)\in\ZZ_+$ for all $i$.
%\end{proposition}
%\begin{proof}
 %Recall that, if $(.|.)_i^\natural$ is the normalized invariant form on $\g_i^\natural$ and $(\beta_k)_{|\g^\natural_i\times \g^\natural_i}=M_i(k)(.|.)_i^\natural$, then there is an embedding $V^{\beta_k}(\g^\natural)\to W^k_{\min}(\g)$. Notice that  the restriction of the $\phi$-invariant Hermitian  form $H$ on $L^W(\nu,\ell_0)$  to $V^{\beta(k)}(\g^\natural). v_{\nu,\ell_0}$ is invariant with respect to the action of $V^{\beta_k}(\g^\natural)$: indeed, the generators $J^{\{a\}}$ of $V^{\beta_k}$ are primary with respect to $L$.
 %In particular, if there is $(\nu,\ell_0)\in (\h^\natural)^*\times\R$ such that $L^W(\nu,\ell_0)$ is unitary, then  $V^{\beta(k)}(\g^\natural). v_{\nu,\ell_0}$ is unitary.

%Since
%$$ V^{\beta(k)}(\g^\natural)=\bigotimes_{i\ge 1} V^{M_i(k)}(\g_i^\natural),$$
%and $\phi$ restricted to $\g^\natural$ defines a compact real form, it follows from \cite{VB} and  \S 5.3 of \cite{KMP} that $M_i(k)\in\ZZ_+$ for all $i$.
%\end{proof}
%In Proposition \ref{necessary} we shall prove the converse of the above statement. 
\begin{definition} The $W_k^{\min}(\g)$-module $L^W(\nu,\ell_0)$  is called unitary if the Hermitian form $H(\cdot,\cdot)$ is positive definite. The vertex algebra $W_k^{\min}(\g)$  is called unitary if its adjoint module is unitary.
\end{definition}
As usual, we denote $\Vert u\Vert =H(u,u),\,u\in L^W(\nu,\ell_0)$. 
In order to obtain necessary  conditions for unitarity of $L^W(\nu,\ell_0)$ we compute 
$||G^{\{v\}}_{-1/2}v_{\nu,\ell_0}||$.
\begin{lemma}\label{GGhalf}Let, as before,  $\xi$ be a highest weight of the $\g^\natural$--module $\g_{-1/2}$, and fix a highest weight vector $v\in\g_{-1/2}$ .  Then
\begin{align}
\Vert G^{\{v\}}_{-1/2}v_{\nu,\ell_0}\Vert^2=&(-2(k+h^\vee)l_0+(\nu|\nu+2\rho^\natural)-2(k+1)(\xi|\nu)+2(\xi|\nu)^2) \langle\phi(v),v\rangle.\label{GGv}
\end{align}
\end{lemma}
\begin{proof}To prove \eqref{GGv} we observe that,
since $g(G^{\{v\}})=G^{\{\phi(v)\}}$ and $G^{\{v\}}$ is primary,
\begin{align*}
H(G^{\{v\}}_{-1/2}v_{\nu,\ell_0},G^{\{v\}}_{-1/2}v_{\nu,\ell_0})&=H(G^{\{\phi(v)\}}_{1/2}G^{\{v\}}_{-1/2}v_{\nu,\ell_0},v_{\nu,\ell_0})\\
&=H([G^{\{\phi(v)\}}_{1/2},G^{\{v\}}_{-1/2}]v_{\nu,\ell_0},v_{\nu,\ell_0}).
\end{align*}
Using Borcherds' commutator formula
$$
[G^{\{\phi(v)\}}_{1/2},G^{\{v\}}_{-1/2}]=\sum_j\binom{1}{j}({G^{\{\phi(v)\}}}_{(j)}G^{\{v\}})_{0},
$$
and formula \eqref{GGsimplifiedfurther} with $u=\phi(v)$ we obtain
\begin{align}\label{nuova}
&[G^{\{\phi(v)\}}_{1/2},G^{\{v\}}_{-1/2}]=-2(k+h^\vee)\langle \phi(v),v\rangle L_0+\langle \phi(v),v\rangle\sum_{\alpha=1}^{\dim \g^\natural} 
:J^{\{u^\alpha\}}J^{\{u_\alpha\}}:_0+\\\notag
&2\sum_{\a,\be}\langle[u_\alpha,\phi(v)],[v,u^\be]\rangle:J^{\{u^\a\}}J^{\{u_\be\}}:_0
 +2(k+1)J^{\{[[e_{\theta},\phi(v)],v]^{\natural}\}}_0 \\\notag
&+ 2  \sum_{\a,\be}\langle[u_\alpha,\phi(v)],[v,u^\be]\rangle
  J^{\{ [u^\a,u_\be]\}}_0.
\end{align}
By the $-1$--st product  identity,
$$
:J^{\{u^\a\}}J^{\{u_\be\}}:_0=\sum_{j\in\ZZ_+}(J^{\{u^\a\}}_{-j-1}J^{\{u_\be\}}_{j+1}+J^{\{u_\be\}}_{-j}J^{\{u^\alpha\}}_{j}),
$$
hence
\begin{equation}\label{nuova2}
:J^{\{u^\a\}}J^{\{u_\be\}}:_0v_{\nu,\ell_0}=J^{\{u_\be\}}_{0}J^{\{u^\alpha\}}_{0}v_{\nu,\ell_0}.
\end{equation}
We choose the basis $\{u_\a\}$ so that $\{u_\a\}=\{u_\gamma\mid u_\gamma\in\g^\natural_\gamma\}\cup\{u_i\mid 1\le i\le \rank\,\g^\natural\}$ with $\{u_i\}$ a basis of $\h^\natural$. Then $u^\gamma \in\g^\natural_{-\gamma}$. It follows that
\begin{equation}\label{nuova3}
H(J^{\{u_\gamma\}}_{0}J^{\{u^{\gamma'}\}}_{0}v_{\nu,\ell_0},v_{\nu,\ell_0})\ne0\Rightarrow \gamma=\gamma'.
\end{equation}
Since
$$
[[e_{\theta},\phi(v)],v]^{\natural}=\sum_{\gamma\in\D^\natural}([[e_{\theta},\phi(v)],v]|u_\gamma)u^\gamma+\sum_i([[e_{\theta},\phi(v)],v]|u_i)u^i,
$$
we see that
\begin{equation}\label{nuova4}
H(J^{\{[[e_{\theta},\phi(v)],v]^{\natural}\}}_{0}v_{\nu,\ell_0},v_{\nu,\ell_0})=\sum_i([[e_{\theta},\phi(v)],v]|u_i)\nu(u_i)=\sum_i(e_{\theta}|[\phi(v),[v,u_i]])\nu(u_i).
\end{equation}
We assume that $v\in \g_\xi$. Then \eqref{nuova4} yields
\begin{equation}\label{nuova5}
H(J^{\{[[e_{\theta},\phi(v)],v]^{\natural}\}}_{0}v_{\nu,\ell_0},v_{\nu,\ell_0})=-\sum_i(e_{\theta}|[\phi(v),v])\xi(u_i)\nu(u^i)=-\langle \phi(v),v\rangle(\xi|\nu).
\end{equation}
From \eqref{nuova3} we see that 
$(J^{\{ [u^{\gamma'},u_\gamma]\}}_0v_{\nu,\ell_0},v_{\nu,\ell_0})=0$ unless $\gamma'=\gamma$. Clearly $J^{\{ [u^i,u_j]\}}_0=0$ for all $i,j$.
Combining \eqref{nuova}, \eqref{nuova3}, \eqref{nuova5} we find
\begin{align}\label{nnuova}
&H([G^{\{\phi(v)\}}_{1/2},G^{\{v\}}_{-1/2}]v_{\nu,\ell_0},v_{\nu,\ell_0})\\\notag
&=-2(k+h^\vee)\langle \phi(v),v\rangle l_0+\langle \phi(v),v\rangle(\nu|\nu+2\rho^\natural)-2(k+1)\langle \phi(v),v\rangle(\xi|\nu)\\\notag
&+2\sum_{\a,\be}\langle[u_\alpha,\phi(v)],[v,u^\be]\rangle H(J^{\{u^\a\}}_0J^{\{u_\be\}}_0v_{\nu,\ell_0},v_{\nu,\ell_0}).
\end{align}

Recall that $\phi$ is a compact involution of $\g^\natural$, thus \begin{equation}\label{phiha}\phi(h_\a)=-h_\a\text{ for all $\a\in\D^\natural$.}\end{equation} (As usual $h_\a$ stands for the element of $\h^\natural$ corresponding to $\a$ in the identification of $\h^\natural$ with  $(\h^\natural)^*$ via $(.|.)$). It follows that $[h_\a,\phi(v)]=-\xi(h_\a)\phi(v)$, so the weight of $\phi(v)$ is $-\xi$.
In particular, since $v$ is a highest weight vector for the $\g^\natural$--module $\g_{-1/2}$, we have
\begin{align}\label{tecnico}
&\sum_{\a,\be}\langle[u_\alpha,\phi(v)],[v,u^\be]\rangle H(J^{\{u^\a\}}_0J^{\{u_\be\}}_0v_{\nu,\ell_0},v_{\nu,\ell_0})=\\\notag
&\sum_{i,j}\langle[u_i,\phi(v)],[v,u^j]\rangle H(J^{\{u^i\}}_0J^{\{u_j\}}_0v_{\nu,\ell_0},v_{\nu,\ell_0})
+\sum_{\gamma<0}\langle[u_\gamma,\phi(v)],[v,u^\gamma]\rangle H(J^{\{u^\gamma\}}_0J^{\{u_\gamma\}}_0v_{\nu,\ell_0},v_{\nu,\ell_0})\\\notag
&=\sum_{i,j}\xi(u_i)\nu(u^i)\xi(u^j)\nu(u_j)\langle \phi(v),v\rangle.
\end{align}
Substituting \eqref{tecnico} into \eqref{nnuova} we obtain
\begin{align*}
H(G^{\{v\}}_{-1/2}v_{\nu,\ell_0},G^{\{v\}}_{-1/2}v_{\nu,\ell_0})=&-2(k+h^\vee)\langle \phi(v),v\rangle l_0+\langle \phi(v),v\rangle(\nu|\nu+2\rho^\natural)\\
&-2(k+1)\langle \phi(v),v\rangle(\xi|\nu)+2(\xi|\nu)^2\langle \phi(v),v\rangle,
\end{align*}
as claimed.
\end{proof}
\begin{remark} Let $v\in \g_{-1/2}$ be as in Lemma \ref{GGhalf} and $u$ a root vector   for the root $\theta_i$. Then
 \begin{align*}
 \Vert J^{\{u\}}_{-1}G^{\{v\}}_{-1/2}v_{\nu,\ell_0}\Vert^2=&((\theta_i|\xi+\nu)(\phi(u)|u)-\beta_k(\phi(u),u))\Vert G^{\{v\}}_{-1/2}v_{\nu,\ell_0}\Vert^2.
 \end{align*}
%To prove \eqref{JGv}, note that in this case
Indeed,
\begin{align*}
&H(J^{\{u\}}_{-1}G^{\{v\}}_{-1/2}v_{\nu,\ell_0},J^{\{u\}}_{-1}G^{\{v\}}_{-1/2}v_{\nu,\ell_0})\\&=-H(G^{\{\phi(v)\}}_{1/2}J^{\{\phi(u)\}}_{1}J^{\{u\}}_{-1}G^{\{v\}}_{-1/2}v_{\nu,\ell_0},v_{\nu,\ell_0})\\
&=-H(G^{\{\phi(v)\}}_{1/2}[J^{\{\phi(u)\}}_{1},J^{\{u\}}_{-1}]G^{\{v\}}_{-1/2}v_{\nu,\ell_0},v_{\nu,\ell_0})\\
&=(\theta_i|\xi+\nu)(\phi(u)|u)H(G^{\{v\}}_{-1/2}v_{\nu,\ell_0},G^{\{v\}}_{-1/2}v_{\nu,\ell_0})-\beta_k(\phi(u),u)H(G^{\{v\}}_{-1/2}v_{\nu,\ell_0},G^{\{v\}}_{-1/2}v_{\nu,\ell_0}).\end{align*}
\end{remark}
%Let $\D^\natural$ be the set of $\h^\natural$--roots of $\g^\natural$, choose $\D^\natural_+=\D^\natural\cap \Dp$ as a set of positive roots,  and  let $\theta_i$ denote the  highest root of $\g^\natural_i$. 
Let $P^+\subset (\h^\natural)^*$ be the set of dominant integral weights  for $\g^\natural$ and let
\begin{equation}\label{p+k}P^{+}_k=\left\{\nu\in P^+\mid  \nu(\theta^\vee_i)\le M_i(k)\text{ for all $i\ge 1$}\right\}.\end{equation}
Recall that $\xi\in (\h^\natural)^*$ is  a  highest weight
of the $\g^\natural$--module $\g_{-1/2}$. Introduce the following number
\begin{equation}\label{Aknu}
 A(k,\nu)=\frac{(\nu|\nu+2\rho^\natural)}{2(k+h^\vee)}+\frac{(\xi|\nu)}{k+h^\vee}((\xi|\nu)-k-1).
 \end{equation}

\begin{prop}\label{l0nec} Assume that $k+h^\vee\ne0$. 
If the $W^k_{\min}(\g)$--module $L^W(\nu,\ell_0)$ is unitary, then $M_i(k)\in\ZZ_+$ for all $i\ge 1$, $\nu\in P^{+}_k$, and
\begin{equation}\label{89a}
\ell_0\ge A(k,\nu).
\end{equation}
\end{prop}
\begin{proof}
In order to prove that $M_i(k)\in\ZZ_+$ for all $i\ge 1$ and $\nu\in P^+_k$, it is enough to observe that, if $L^W(\nu,\ell_0)$ is a unitary  module  over  $W^k_{\min}(\g)$, then, in particular, $V^{\be_k}(\g^\natural)v_{\nu,\ell_0}$ is a unitary  module  over  $V^{\be_k}(\g^\natural)$, hence $\nu\in P^{+}_k$ \cite{VB}, which is non-empty if and only if $M_i(k)\in\ZZ_+$ for all $i\ge 1$. 

To prove the second claim recall that, by Proposition \ref{fh}, the Hermitian  form $\langle\phi(.),.\,\rangle$ is positive definite on $\g_{-1/2}$.  Since $k+h^\vee<0$, we obtain from \eqref{GGv} that
$$
\ell_0\ge\frac{(\nu|\nu+2\rho^\natural)}{2(k+h^\vee)}-\frac{(k+1)}{k+h^\vee}(\xi|\nu)+\frac{(\xi|\nu)^2}{k+h^\vee}=A(k,\nu),
$$
as claimed.
\end{proof}  
%Let $L^W(\nu,\ell_0)$ be  a unitary highest weight  module over $W^k_{\min}(\g)$. Then it is a unitary module when restricted to  $V^{\be_k}(\g^\natural)$. In particular, it is completely reducible and it is a module over the simple affine vertex algebra $V_{\be_k}(\g^\natural)$, hence  
%$$
%L^W(\nu,\ell_0)=\bigoplus_{\mu\inP^+_k}L^W(\nu,\ell_0)(\mu),
%$$
%where $L^W(\nu,\ell_0)(\mu)$ is the isotypic component  of the  $V^{\beta_k}(\g^\natural)$-module 
%$L^\natural(\mu)$.
Consider the short exact sequence
$$0\to I^k\to W^k_{\min}(\g)\to W_k^{\min}(\g)\to 0.$$
If a  $W^k_{\min}(\g)$--module $L^W(\nu,\ell_0)$ is   unitary, then, 
restricted to the subalgebra $V^{\beta_k}(\g^\natural)$ it is unitary, hence a direct sum of irreducible
integrable highest  weight  $\widehat \g^\natural$--modules of levels $M_i(k),\  i\ge 1$.
But it is well known that all these modules descend to  $V_{\beta_k}(\g^\natural)$. Also, all these modules are annihilated 
by the elements
\begin{equation}\label{ei}
(J^{\{e_{\theta_i}\}}_{(-1)})^{M_i(k)+1}\vac,\quad i \ge 1.\end{equation}
Let $\widetilde I^k\subset I^k$ be the ideal of  $W^k_{\min}(\g)$ generated by the elements \eqref{ei}, and let 
 $\widetilde W_k^{\min}=W^k_{\min}/\widetilde I^k$. We thus obtain  

\begin{proposition}\label{8.5} If the $W^k_{\min}(\g)$--module $L^W(\nu, \ell_0)$ is unitary, then it 
descends to $\widetilde W_k^{\min}(\g)$.
\end{proposition}
%\begin{proposition} The Zhu algebra $\widetilde Z^k=Zhu_{\mathbb Z}(\widetilde W_k^{\min}(\g))$ is isomorphic to $\C[L]\otimes \left(U(\g^\natural)/I_k\right)$, where
%$I_k$ is the two-sided ideal of $U(\g^\natural)$ generated by $e_{\theta_i}^{M_i(k)+1}, \,i\ge 1$.
%\end{proposition}
%\begin{proof} The algebra $Zhu_{L_0}(W^k_{\min}(\g))$ is described in \cite{KMP}. From this description it follows that for any irreducible module over
%$\C[L]\otimes U(\g^\natural)$ there exists the corresponding irreducible $W^k_{\min}(\g)$--module, hence  $Zhu_{\mathbb Z}(W^k_{\min}(\g))$ is a quotient 
%of $\C[L]\otimes U(\g^\natural)$ by some ideal $I$. By Corollary 6.1 from \cite{KW1}, $I=0$. The claim follows.
%\end{proof}
Note that a unitary $W^k_{\min}(\g)$--module descends to $W_k^{\min}(\g)$ if and only if 
\begin{equation}\label{ik}
\widetilde I^k=I^k.
\end{equation}
\vskip5pt
\noindent{\sl  {\bf Conjecture 4.}\footnote{The proof of this conjecture will appear in a forthcoming paper joint with D. Adamovi\'c.}  Equality \eqref{ik} holds for all unitary vertex algebras $W^k_{\min}(\g)$. Consequently, any unitary $W^k_{\min}(\g)$--module descends to 
$W_k^{\min}(\g)$.}
%\vskip5pt
%\begin{itemize}
%\item Introduce $V$
%\item $V$ is exact ?  Properties
%\begin{enumerate}
%\item If $M_0$ is a module for the Zhu algebra $\mathcal Z$, then $Zhu( V(M_0))=M_0$.
%\item If $M$ is a module for $\mathcal V$ generated by $M_{top}$, and $M_{top}=M_0$, then $M$ is a quotient of 
%$V(M_0)$.
%\end{enumerate}
%\item Consider $\nu\in \widehat P^k$, consider $L_{\g^\natural}(\nu)$ and extend to a module $L_\nu=\C[L]\otimes L_{\g^\natural}(\nu)$  for $\widetilde Z^k$ by letting $\C[L]$ act by left multiplication. 
%Consider $V(L_\nu)$.
%\item $V(L_\nu)$ is a free $\C[L]$--module ?
%\item  $V(L_\nu)$ is a right $\C[L]$--module
%\item Set $\overline M^k(\nu,h) = V(L_\nu)/(V(L_\nu) Ker\,ev_h)$
%\end{itemize}

\begin{definition}\label{extr} An element $\nu\in P^+_k$ is called an {\it extremal weight} if 
$\nu+\xi$ doesn't lie in $P^+_k$.
\begin{prop}\label{boundary} If $L^W(\nu,\ell_0)$ is unitary and $\nu$ is an extremal weight,  then 
$$\ell_0=A(k,\nu).
$$
\end{prop}
\begin{proof}Let $u$ be a root vector for $\xi$.
Then $G^{\{u\}}_{-1/2}v_{\nu,\ell_0}$ is a singular vector for $V^{\be_k}(\g^\natural)$. Since  $L^W(\nu,\ell_0)$ is unitary, all  vectors that are singular for 
$V^{\be_k}(\g^\natural)$ should have weight in $P^+_k$. By the assumption, we have 
$G^{\{u\}}_{-1/2}v_{\nu,\ell_0}=0$, hence the norm of this vector is $0$, and we can apply 	\eqref{GGv}.
\end{proof}
In the setting of the above proposition, note that $\nu$ is extremal iff $\nu(\theta_i^\vee)>M_i(k)+\chi_i$ for some $i$.
Moreover,  $k$ is collapsing iff $M_i(k)+\chi_i <0$ (cf. Remark \ref{7.4}).
\end{definition}
%\begin{remark} One can obtain necessary conditions for the unitarity of $L^W(\nu,\ell_0)$ by restricting the action to $\g^\natural$ and $L$, hence regarding it as a module for the vertex algebra $V^{\beta_k}(\g^\natural)\otimes Vir$, which is the universal affine vertex algebra of the  semidirect product of the affinization of $\g^\natural$ with cocycle $\beta_k$ and the Virasoro algebra. Unitary representations of these semidirect products, called {\it conformal current algebras}, were studied in 
%\cite{Kacconformal}: Proposition 1 of {\it loc. cit.} gives necessary and sufficient conditions for unitarity. It turns out that the conditions obtained in this way are already implied by Proposition \ref{l0nec}.
%and Proposition  \ref{boundary}.
%\end{remark} 
\begin{prop}\label{casi}\ 
\begin{itemize} 
\item[(a)] For
$k\ne -1$, $W_{\min}^k(sl(2|m)),\, m\ge 3, $ has no unitary highest weight modules.
In particular, $W^{\min}_k(sl(2|m)),\, m\ge 3, $ is unitary if and only if $k=-1$ and this $W$--algebra collapses  to the free boson. 
\item[(b)]
The $W$--algebra $W_{\min}^k(osp(4|m)),\, m>2,$ has  no unitary highest weight modules for all $k$.
\end{itemize}
\end{prop}
\begin{proof} (a) Let 
$\g=sl(2|m)$. Then $\g_0^\natural= \C \varpi$, where $\varpi=\left(
\begin{array}{cc|c}
m/2 & 0 &  0\\0 & m/2   & 0\\   \hline 0& 0 & I_m
\end{array}\right)$, and $(a|b)=str(ab)$. By Theorem \ref{oldresults}, the  collapsing levels are $k=-1$ and $k=m/2-1$.\par
If $k=-1$ then $M_0(-1)=-m/2,  M_1(-1)=0$ and  $W_k^{\min}(\g)$ is the Heisenberg vertex algebra $M(\C\varpi)=V^{-m/2}(\C\varpi)=V_{-m/2}(\C\varpi)$ and this vertex algebra is unitary. \par If $k=m/2-1$ then  $M_0(m/2-1)=0,\  M_1(m/2-1)=-m/2$ and $W_k^{\min}(sl(2|m))=V_{-m/2}(sl(m))$ which has no unitary highest weight modules.\par
Assume that $k$ is not collapsing.  Let  $\psi$ be a conjugate linear involution of $W^k_{\min}(sl(2|m))$ such that $L^W(\nu,\ell_0)$  has a  positive definite $\psi$--invariant Hermitian form 
$H$, normalized by the condition $H(v_{\nu,\ell_0},v_{\nu,\ell_0})=1$. By Proposition \ref{converse}, the involution $\psi$ is induced by an involution $\psi$ on $\g$ satisfying \eqref{primas}. This implies that $\psi(\varpi)=\zeta\varpi$ with $|\zeta|=1$.

 The vertex algebra 
$V^{k-m/2-1}(\C\varpi)\otimes V^{-k-1}(sl(m))$ embeds in  $W^k_{\min}(sl(2|m))$. In particular, $(V^{k-m/2-1}(\C\varpi)\otimes V^{-k-1}(sl(m))). v_{\nu,\ell_0}$ is a unitary module.
This implies that $\psi_{|sl(m)}$ corresponds to a compact real form of $sl(m)$ and $-k-1\in\Z_+$.
%Let $\langle b \rangle_{ev}$ denote the expectation value of $b$, i.e.
%the coefficient of the projection of $b$ on the vacuum vector $\vac$. 
Using the formulas given in
\cite[\S 5.3]{KMP} we have 
\begin{align*}0\le H(J^{\{\varpi\}}v_{\nu,\ell_0},J^{\{\varpi\}}v_{\nu,\ell_0})&=H(-J^{\{\psi(\varpi)\}}_1J^{\{\varpi\}}_{-1}v_{\nu,\ell_0},v_{\nu,\ell_0})\\&=-(k-m/2-1)\zeta^{-1} (\varpi|\varpi)=-\zeta^{-1}(k-m/2-1)(m^2/2-m).
\end{align*}
Therefore $\zeta=1$, so that 
\begin{equation}\label{=pi}\psi(\varpi)=\varpi.
\end{equation}
%Note that, from the relation $[J^{\{a\}},G^{\{u\}}]=J^{\{[a,u]\}}$, $a\in\g^\natural, u\in \g_{-1/2}$,  it follows that $\psi$ can be extended to $\g_{-1/2}$, in such a way that $\psi([a,u])=[\psi(a),\psi(u)]$.
%Consider  the following symmetric bilinear  form on $\g_{-1/2}$ 
%\begin{equation}\label{sesq}\langle u,v\rangle=(e|[u,v]).\end{equation}
Note that 
\begin{equation}\label{id} [\varpi, u]=\pm\tfrac{m}{2}u,\quad u\in\g_{-1/2}.
\end{equation}
Write $\g_{-1/2}=\g_{-1/2}^+\oplus\g_{-1/2}^-$ for the corresponding eigenspace decomposition. Since\break  $\psi(\varpi)=\varpi$, we have $\psi(\g^\pm_{-1/2})=\g^\pm_{-1/2}$. Since the form $\langle.,\,.\rangle$ is $\g^\natural$--invariant, we have
$$
\langle \g^+,\g^+\rangle=\langle \g^-,\g^-\rangle=0.
$$
It follows that, if $u\in\g_{-1/2}$,
\begin{equation}\label{uu=0}
\langle\psi(u),u\rangle=0.
\end{equation}
Observe now that by \cite{AKMPP}, since $k$ is not collapsing, the image of $G^{\{u\}}$ in $W_k^{\min}(\g)$ is non-zero if $u\ne 0$. 
We observe that,
since $g(G^{\{u\}})=G^{\{\psi(u)\}}$ and $G^{\{u\}}$ is primary, for $n\in \tfrac{1}{2}+\mathbb Z_+$
\begin{align}\label{HGG}
H(G^{\{u\}}_{-n}v,G^{\{u\}}_{-n}v)&=H(G^{\{\psi(u)\}}_{n}G^{\{u\}}_{-n},v)\\\notag
&=H([G^{\{\psi(u)\}}_{n},G^{\{u\}}_{-n}]v,v)
\end{align}
for any $v\in L^W(\nu,\ell_0)$. Using Borcherds' commutator formula
$$
[G^{\{\psi(u)\}}_{n},G^{\{u\}}_{-n}]=\sum_j\binom{n+\tfrac{1}{2}}{j}({G^{\{\psi(u)\}}}_{(j)}G^{\{u\}})_{0},
$$
and combining  formulas \eqref{GGsimplifiedfurther} and \eqref{uu=0}, we obtain
\begin{align}\label{nuova33}
&[G^{\{\psi(u)\}}_{n},G^{\{u\}}_{-n}]=-2(k+h^\vee)\langle \psi(u),u\rangle L_0+\langle \psi(u),u\rangle\sum_{\alpha=1}^{\dim \g^\natural} 
:J^{\{u^\alpha\}}J^{\{u_\alpha\}}:_0+\\\notag
&2\sum_{\a,\be}\langle[u_\alpha,\psi(u)],[u,u^\be]\rangle:J^{\{u^\a\}}J^{\{u_\be\}}:_0
 +4n(k+1)J^{\{[[e_{\theta},\psi(u)],u]^{\natural}\}}_0 \\\notag
&+ (2n+1)  \sum_{\a,\be}\langle[u_\alpha,\psi(u)],[u,u^\be]\rangle
  J^{\{ [u^\a,u_\be]\}}_0+(2n^2-\tfrac{1}{2})p(k)\langle \psi(u),u\rangle\\\notag
&=2\sum_{\a,\be}\langle[u_\alpha,\psi(u)],[u,u^\be]\rangle:J^{\{u^\a\}}J^{\{u_\be\}}:_0
 +4n(k+1)J^{\{[[e_{\theta},\psi(u)],u]^{\natural}\}}_0 \\\notag
&+ (2n+1)  \sum_{\a,\be}\langle[u_\alpha,\psi(u)],[u,u^\be]\rangle
  J^{\{ [u^\a,u_\be]\}}_0.
\end{align}
Now we compute \eqref{HGG} for $v\in L^W(\nu,\ell_0)_{\ell_0}$.  As in the proof of Lemma \ref{GGhalf}, using \eqref{nuova3}, \eqref{nuova5} with $\psi$ instead of $\phi$, we find that \eqref{nuova33} becomes, with the notation of the proof of Lemma \ref{GGhalf},
\begin{align}\notag
H([G^{\{\psi(u)\}}_{n},G^{\{u\}}_{-n}]v,v)&=(2n+1)\sum_{\a,\be}\langle[u_\alpha,\psi(u)],[u,u^\be]\rangle H(J^{\{u^\a\}}_0J^{\{u_\be\}}_0v,v)\\\label{a}
&=(2n+1)\sum_{i,j}\langle[u_i,\psi(u)],[u,u^j]\rangle H(J^{\{u^i\}}_0J^{\{u_j\}}_0v,v)\\\label{b}
&+(2n+1)\sum_{\gamma\in\D^\natural}\langle[u_\gamma,\psi(u)],[u,u^\gamma]\rangle H(J^{\{u^\gamma\}}_0J^{\{u_\gamma\}}_0v,v).
\end{align} 
Recall that $\psi$ ia a compact involution of $[\g^\natural,\g^\natural]$, hence, by \eqref{phiha}, $\psi(u_\gamma)\in\g^\natural_{-\gamma}$, so that, for some constant $b$ we have 
$$\langle[u_\gamma,\psi(u)],[u,u^\gamma]\rangle=b \langle \psi( [u,u^\gamma]),[u,u^\gamma]\rangle, $$
so,  by \eqref{uu=0}, the summand  \eqref{b} is zero. \par The summand \eqref{a}  vanishes  since 
$\langle[u_i,\psi(u)],[u,u^j]\rangle $ is a multiple of $\langle \psi(u),u\rangle=0$. This shows that 
$Y^{L^W(\nu,\ell_0)}(G^{\{u\}},z)v=0$. 
By relation \eqref{primarel}, $G^{\{u\}}_n A_m v=0$ with $A\in V^{\beta_k}(\g^\natural)$ for all $n, m$, hence, since $G^{\{u\}}$ is primary,  $$Y^{L^W(\nu,\ell_0)}(G^{\{u\}},z)L^W(\nu,\ell_0)=0.$$
Hence $G^{\{u\}}$ lies in a proper ideal of $W^{k}_{\min}(\g)$, contradicting the fact that, since the level is not collapsing,  $G^{\{u\}}$ is non zero in $W_{k}^{\min}(\g)$.

(b)
 For $\g=osp(4|m)$, the  conditions of Proposition \ref{l0nec} imply
$k-m/2\in\mathbb Z_+,\ -\frac{1}{2} k-1\in \mathbb Z_+.
$
These relations  are never satisfied at the same time.
\end{proof}
\begin{prop}\label{810} Non-trivial unitary irreducible highest weight $W^k_{\min}(\g)$--modules
with $k\ne -h^\vee$ may exist only in the following cases
\begin{enumerate}
\item $\g=sl(2|m),\ m\ge 3$, $k=-1$ (then $W^{\min}_k=W_{\min}^k$ is a free boson);
\item  $\g=psl(2|2)$, $-k\in\mathbb N+1$;
\item  $\g=spo(2|3)$, $-k\in\frac{1}{4}(\mathbb N+2)$;
\item $\g=spo(2|m),\ m>4$, $-k\in\frac{1}{2}(\mathbb N+1)$;
%\item   $\g=D(2,1;-\frac{m}{m+n})$, $-k\in\frac{mn}{m+n}\nat,\ m,n\in\mathbb N$ are coprime, $k\ne -\frac{1}{2}$;
\item   $\g=D(2,1;\frac{m}{n})$, $-k\in\frac{mn}{m+n}\nat,\ m,n\in\mathbb N$ are coprime, $k\ne -\frac{1}{2}$;
\item   $\g=F(4)$, $-k\in\frac{2}{3}(\mathbb N+1)$;
\item  $\g=G(3)$, $-k\in\frac{3}{4}(\mathbb N+1)$.
\end{enumerate}
\end{prop}
\begin{proof} By Proposition \ref{casi}, we may assume that $\g$ is not one of the Lie superalgebras $sl(2|m)$ woth $m\ge 3$ or 
$osp(4|m)$ with $m>2$. The remaining cases are treated, using only the easy necessary conditions $M_i=M_i(k)\in\mathbb Z_+$ for all $i$. In all cases, except for $\g=D(2,1;a)$, the condition $M_i\in \mathbb Z_+$ is obviously equivalent to the condition on $k$, given in the statement of the proposition. \par
Consider the remaining case $\g=D(2,1;a)$. By this we mean the contragredient Lie superalgebra with Cartan matrix $\begin{pmatrix}
0 & 1 & a\\ -1 & 2 & 0\\ -1 & 0 & 2 \end{pmatrix}$. 
%By Proposition \ref{casi} (a) we may exclude case (1). 
By Proposition \ref{l0nec},  we need to find the values of $a$ such that $M_i=M_i(k),\  i=1,2,$ from Table 2 are non-negative integers. These conditions imply that
\begin{equation}\label{NV}
k=-\tfrac{M_1+1}{a+1}\text{ and } k=-\tfrac{(M_2+1)a}{a+1}. \text{ where $M_1, M_2\in\mathbb Z_+$}.
\end{equation}
Equating these two expressions for $k$, we obtain $a=\tfrac{M_1+1}{M_2+1}$ is a positive rational number. Inserting this in either of the expressions 	\eqref{NV} for $k$, we obtain
$$k=-\tfrac{(M_1+1)(M_2+1)}{(M_1+1)+(M_2+1)},
$$
proving the claim. ($k=-1/2$ corresponds to the trivial $D(2,1;1)$--module.)
\end{proof}
\begin{definition}
Given $\g$ in the above list, we call the corresponding set of values of $k\neq -h^\vee$  the {\it unitarity range} of $W^k_{\min}(\g)$. 
\end{definition}
\begin{remark} For  $\g=D(2,1;a)$, there are actually three possible choices of the minimal root. We now describe how  the unitarity range depends on this choice.  We choose
$\{2\e_1,2\e_2,2\e_3\}$ as the set of positive roots in $\g_{\bar 0}$: hence, if $-\theta$ is a minimal root, then $\theta=2\e_i$ for some $1\leq i\leq 3$.
The bilinear form $(.|	.),$ displayed in   Table 1, corresponds to the choice $\theta=2\e_1$, so that $(2\e_1|2\e_1)=2$. If we choose $\theta=2\e_2$, then the bilinear form $(.|.)$ is  given by 
$$
(\e_1|\e_1)=-\tfrac{1+a}{2},\ (\e_2|\e_2)=\tfrac{1}{2}, \ (\e_3|\e_3)=\tfrac{a}{2}, \ (\e_1|\e_2)=(\e_1|\e_3)=(\e_2|\e_3)=0.$$
We have $M_1(k)=-\tfrac{1}{1+a}k-1, M_2(k)=\tfrac{1}{a}k-1$. Then $a=-\tfrac{m}{m+n}, \ m,n\in \mathbb N,$ $m$ and $n$ are coprime (i. e.  $a\in\mathbb Q$, $-1<a<0$) and in turn  $k\in - \tfrac{mn}{m+n}\mathbb N$.
If we choose $\theta=2\e_3$, then the bilinear form $(.|.)$ is  given by 
$$(\e_1|\e_1)=-\tfrac{a+1}{2a},\ (\e_2|\e_2)=\tfrac{1}{2a}, \ (\e_3|\e_3)=\tfrac{1}{2}, \ (\e_1|\e_2)=(\e_1|\e_3)=(\e_2|\e_3)=0.
$$
We have $M_1(k)=-\tfrac{a}{1+a}k-1,\ M_2(k)=ak-1$. Then $a=-\tfrac{m+n}{m}, \ m,n\in \mathbb N,$ $m$ and $n$ are coprime (i. e.  $a\in\mathbb Q$, $a<-1$) and in turn  $k\in - \tfrac{mn}{m+n}\mathbb N$.\par

Recall that one obtains isomorphic superalgebras of the family $D(2,1;a), a\ne 0,-1,$ under the action of the group $S_3$, generated by the transformations $a\mapsto 1/a, a\mapsto  -1-a$. These transformations permute transitively the domains $\mathbb Q_{>0}, 
\mathbb Q_{>-1}\cap \mathbb Q_{<0}$ and $\mathbb Q_{<-1}$, which correspond to the above three cases. 
\end{remark}
\begin{cor}\label{uur} If $k$ is from the unitarity range for $W^k_{\min}(\g)$, then $k+h^\vee$ is a negative rational number.
\end{cor}
 \section{Free field realization of minimal $W$--algebras}\label{8}

Let $\Psi: W^k_{\min}(\g)\to \mathcal V^k=V^{k+h^\vee}(\C x)\otimes V^{\alpha_k}(\g^\natural)\otimes F(\g_{1/2})$ be the free field realization introduced  in \cite[Theorem 5.2]{KW1}; it is explicitly given on the generators of $W^k_{\min}(\g)$ by 
\begin{align}
\label{FFR1}&   J^{\{ b \}} \mapsto b + \frac{1}{2}
      \sum_{\alpha \in S_{1/2}}: \Phi^{\alpha}
      \Phi_{[u_{\alpha},b]}: (b \in \fg^{\natural}), \\
\label{FFR2}&    G^{\{ v \}} \mapsto \sum_{\alpha \in S_{1/2}}
         : [v,u_{\alpha}]\Phi^{\alpha}:
         -(k+1)\sum_{\alpha \in S_{1/2}}
         (v|u_{\alpha})  T \Phi^{\alpha}\\
\notag&     +\frac{1}{3} \sum_{\alpha ,\beta \in S_{1/2}}
         : \Phi^{\alpha}
         \Phi^{\beta}\Phi_{[u_{\beta},[u_{\alpha},v]]}
         : (v \in \fg_{-1/2})\, , \\
\label{FFR4}&       L \mapsto \frac{1}{2 (k+h^\vee)}
         \sum_{\alpha \in S_0} :u_{\alpha}
         u^{\alpha}:+ \frac{k+1}{k+h^\vee}T x +\frac{1}{2}
         \sum_{\alpha \in S_{1/2}} : (T \Phi^{\alpha})
         \Phi_{\alpha}: .
\end{align}\par
Recall that $F(\g_{1/2})$ is the universal enveloping vertex algebra of the (non-linear) Lie conformal superalgebra $\C[T]\otimes \g_{1/2}$ with
$[a_\l b]=\langle a,b\rangle_{ne}\vac,\,a,b\in\g_{1/2}$, and 
%$\langle a,b\rangle_{ne}$ is defined in \eqref{sesq}, and 
$\{\Phi_\a\}_{\a\in S_{1/2}},$ $\{\Phi^\a\}_{\a\in S_{1/2}}$ are dual bases of $\g_{1/2}$ with respect to $\langle., .\rangle_{ne}$.\par
We now apply the results of Section \ref{freeb} to $V^{k+h^\vee}(\C x)$.  By Corollary \ref{uur}  unitarity of $W^k_{\min}(\g)$ implies  $k+ h^\vee<0$. Hence, using the normalization 
\begin{equation}\label{norm}a=\sqrt{-1}\frac{\sqrt{2}}{\sqrt{|k+h^\vee|}}x,
\end{equation}
we have $V^{k+h^\vee}(\C x)=V^{1}(\C a)$, since, by \eqref{eq:5.16}, $\a_k(x,x)=\tfrac{1}{2}(k+h^\vee)$, hence $\a_k(a,a)=1$

%Recall that for our goals we assume $k+h^\vee<0$.\par
Recall that  in Proposition \ref{fh} we proved that one can choose an almost compact involution $\phi$ of $\g$ that  fixes pointwise the $sl_2$--triple $\{e,x,f\}$ in such a way that  the Hermitian form 
$\langle \phi(u),v\rangle_{ne}$ on $\g_{1/2}$ is negative definite.
This conjugate linear involution induces a conjugate linear  involution of    $W^k_{\min}(\g)$ and  of $V^{\alpha_k}(\g_0)\otimes F(\g_{1/2})$ as well, both denoted again by $\phi$. It is readily checked, using \eqref{FFR1}, \eqref{FFR2}, and \eqref{FFR4}, that
\begin{equation}\label{Psiphi}
\Psi(\phi(v))=\phi(\Psi(v))\quad\text{for all $v\in W^k_{\min}(\g)$.}
\end{equation}
%Also recall that 
%$$(\cdot,\cdot)_2=(\cdot,\cdot)_{\C x}\otimes(\cdot,\cdot)_{\g^\natural}\otimes(\cdot,\cdot)_F.
%$$
Since $\phi(x)=x$, we see that $\phi(a)=-a$.
%Let $(\cdot,\cdot)_W$ be the invariant form on $W^k_{\min}(\g)$.
The conformal  vector of the vertex algebra $\mathcal V^k$ is
\begin{align}\label{Lfreefield}
L_{free}&=\frac{1}{2}: a a :+L_{\g^\natural}+L_F, 
\end{align}
where
$$ L_{\g^\natural}=\tfrac{1}{2(k+h^\vee)}\sum_{\a\in S^\natural} :u_\a u^\a,\quad L_F=\tfrac{1}{2}\sum_{\a\in S_{1/2}}:(T\Phi^\a)\Phi_\a:.$$
Here $\{u_\a\}_{\a\in S^\natural}$  and $\{u^\a\}_{\a\in S^\natural}$ are dual bases of $\g^\natural$ with respect to the bilinear form $(.|.)$ restricted to $\g^\natural$. 
Recall that $L_{\g^\natural}$ is the conformal vector of $V^{\alpha_k}(\g^\natural)$ and $L_F$ is the conformal vector of $F(\g_{1/2})$. Let 
\begin{equation}\label{sfix}s_k=\sqrt{-1}\frac{(k+1)}{\sqrt{2|k+h^\vee|}}.\end{equation}
It follows  from \eqref{FFR4} and \eqref{Lfreefield} that 
%$$\tfrac{1}{k+h^\vee}:xx:=\tfrac{1}{2}:aa:,$$
%hence 
\begin{equation}\label{Psi(L)}
\Psi(L) =L(s_k)+\widehat L=L_{free}+s_kT (a),
\end{equation}
where $\widehat L=L_{\g^\natural}+L_F$, and $L(s)=\tfrac{1}{2}:aa:+sTa,\  \widehat L(s)=L(s)+\widehat L$, cf. \eqref{L(t)} and \eqref{nuovovir}, respectively.
\par
Note that  $\mathcal V^k=V^1(\C a)\otimes V,$  where $V=V^{\alpha_k}(\g^\natural)\otimes F(\g_{1/2})$, and  $\Psi(L)=\widehat L(s)$ (cf. \eqref{628}, \eqref{nuovovir}).

Given $\mu\in\C$, let  $M(\mu)$ be the irreducible $V^1(\C a)$--module with highest weight $\mu$, and consider the $\mathcal V^k$--module $$N(\mu)=M(\mu)\otimes V.$$
Recall that $V$ carries a $\phi$--invariant Hermitian form $H_{\g^\natural}\otimes H_F, $ which is positive definite. Recall also that, by Proposition \ref{Eform}, the $V^1(\C a)$--module $M(\mu)$ carries a unique $L(t)$--invariant Hermitian form, provided that $t=\sqrt{-1}\Im(\mu)$, which is positive definite. This Hermitian form, normalized by the condition that the norm of the highest weight vector equals $1$, was denoted by $H_\mu$. Hence we have a $\phi$--invariant positive definite Hermitian form 
$H_\mu(\,.\,,.\,)\otimes H_{\g^\natural}(.\,,.)\otimes H_F(.\,,.)$ on $N(\mu)$, which we denote by $(\cdot,\cdot)_\mu$.\par
It follows from Proposition \ref{FCwithV} that, restricting the fields $Y^{\mu,t}(-,z)$ from $\mathcal V^k$ to $\Psi\!(W^k_{\min}(\g)\!)$, one equips $N(\mu)$ with a structure of  a $W^k_{\min}(\g)$--module.
We now  explicitly  describe this action of the generators of $W^k_{\min}(\g)$ on $N(\mu)$.
\begin{proposition}\label{FairlieW} For $b\in W^k_{\min}(\g)$, write
$$
Y^{\mu,t}(\Psi(b),z)=\sum_{n\in-\D_b+\ZZ}b^{\mu,t}_nz^{-n-\D_b},
$$
and let $\mu\in\R$. Then
\begin{align}
L^{\mu,t}_n&=\Psi( L)^\mu_n+2ta^\mu_n+2(t^2-st)\vac^\mu_n,\\
(J^{\{u\}})^{\mu,t}_n&=\Psi(J^{\{u\}})^{\mu}_n,\ u\in\g^\natural,\\
 (G^{\{ v \}})^{\mu,t}_n
        &=\Psi(G^{\{ v \}})^\mu_n+2t\sqrt{-1}\sqrt{2|k+h^\vee|} (\Phi_{[e,v]})^\mu_n,\ v\in \g_{-1/2}.
\end{align}
Furthermore, if $m,m'\in N(\mu)$, then
\begin{align}
(m,L^{\mu,t}_nm')_\mu&=(L^{\mu,s-t}_{-n}m,m')_\mu,\label{INV1}\\
(m,(J^{\{u\}})^{\mu,t}_nm')_\mu&=-((J^{\{\phi(u)\}})^{\mu,s-t}_{-n}m,m')_\mu,\label{INV2}\\
(m, (G^{\{ v \}})^{\mu,t}_nm')_\mu
        &=((G^{\{ \phi(v) \}})^{\mu,s-t}_{-n}m,m')_\mu.\label{INV3}
\end{align}
\end{proposition}
\begin{proof}
We already noted that $\Psi(L)=\widehat L(s)$. By \eqref{L(s)real},
\begin{align*}
\Psi(L)^{\mu,t}_n&=(L(s)+\widehat L)^{\mu,t}_n=\half :aa:^\mu_n+sTa^\mu_n+2ta^\mu_n+2(t^2-st)\vac_n+\widehat L^\mu_n\\
&=\widehat L(s)^\mu_n+2ta^\mu_n+2(t^2-st)\vac^\mu_n.
\end{align*}
If $u\in\g^\natural$, then $\Psi(J^{\{u\}})\in V^{\a_k}(\g^\natural)\otimes F(\g_{1/2})$, hence, by Lemma \ref{atmuntV}, $(J^{\{u\}})^{\mu,t}_n=\Psi(J^{\{u\}})^{\mu}_n$. 
Finally, if $v \in \fg_{-1/2}$,
$$[v,u_\a]=2([v,u_\a]|x)x+[v,u_\a]^\natural=\sqrt{-1}\frac{\sqrt{|k+h^\vee|}}{\sqrt{2}}(v|u_\a)a+[v,u_\a]^\natural,
$$
where $u^\natural $ is the orthogonal projection of $u$ onto $\g^\natural$ with respect to $(\,.\,|\,.\,)$.
Since
\begin{align*}
[e,v]=&\sum_\a \langle [e,v],u_\a\rangle_{ne} u^\a=\sum_\a ( f|[[e,v],u_\a]) u^\a=\sum_\a ( f|[e,[v,u_\a]]) u^\a\\
=&\sum_\a ( [[f,e],v]|u_\a) u^\a=-\sum_\a ( [x,v]|u_\a) u^\a=\tfrac{1}{2}\sum_\a ( v|u_\a) u^\a,
\end{align*}
we can write
\begin{align*}   
   \Psi(G^{\{ v \}})=&\sqrt{-1}\sqrt{2|k+h^\vee|}:a \Phi_{[e,v]}:+\sum_{\alpha \in S_{1/2}}
         : [v,u_{\alpha}]^\natural\Phi^{\alpha}:
         -2(k+1)  T \Phi_{[e,v]}\\
\notag&     +\frac{1}{3} \sum_{\alpha ,\beta \in S_{1/2}}
         : \Phi^{\alpha}
         \Phi^{\beta}\Phi_{[u_{\beta},[u_{\alpha},v]]}
         : .
         \end{align*}
         Set 
$$G^{\{v\}}=\sum_{\alpha \in S_{1/2}}
         : [v,u_{\alpha}]^\natural\Phi^{\alpha}:+\frac{1}{3} \sum_{\alpha ,\beta \in S_{1/2}}
         : \Phi^{\alpha}
         \Phi^{\beta}\Phi_{[u_{\beta},[u_{\alpha},v]]}:.
         $$
so that 
\begin{align*}   
   (G^{\{ v \}})^{\mu,t}_n=&\sqrt{-1}\sqrt{2|k+h^\vee|}:a \Phi_{[e,v]}:^\mu_n+2t\sqrt{-1}\sqrt{2|k+h^\vee|} (\Phi_{[e,v]})^\mu_n\\
   &-2(k+1) ( T \Phi_{[e,v]})^\mu_n    +{G^{\{v\}}}^\mu_n.
\end{align*}
Thus, 
\begin{align}   \label{GFairlie}
   (G^{\{ v \}})^{\mu,t}_n
        =\Psi(G^{\{ v \}})^\mu_n+2t\sqrt{-1}\sqrt{2|k+h^\vee|} (\Phi_{[e,v]})^\mu_n.
\end{align}

For proving \eqref{INV1}, \eqref{INV2}, and \eqref{INV3}, it is enough to observe that $L$, $G^{\{ v \}}$, and $J^{\{u\}}$ are quasiprimary for $\widehat L(s)$ and apply \eqref{mutprimary}. We use the fact that $\Psi(g(b))=g(\Psi(b))$ for all $b\in W^k_{\min}(\g)$ , where $g$ is defined by \eqref{113}. This follows from \eqref{Psiphi} and the fact that $\Psi$ preserves both parity and conformal weight.
\end{proof}
As an application of Proposition \ref{FairlieW}, we obtain a generalization of  the Fairlie construction to minimal $W$--algebras.
\begin{proposition}\label{FCW} Set $s=s_k$ (cf. \eqref{sfix}) and 
$$
L^{\mu,s/2}_n=\Psi(L)^\mu_n+sa^\mu_n+\tfrac{|s|^2}{2}\vac^\mu_n=\Psi(L)^\mu_n+\frac{k+1}{k+h^\vee}x^\mu_n-\tfrac{(k+1)^2}{4(k+h^\vee)}\vac^\mu_n,
$$
$$
(G^{\{ v \}})^{\mu,s/2}_n
        =\Psi(G^{\{ v \}})^\mu_n-(k+1)(\Phi_{[e,v]})^\mu_n,
$$
$$
(J^{\{ u \}})^{\mu,s/2}_n
        =\Psi(J^{\{ u \}})^\mu_n.
$$
The fields 
\begin{align*}Y^{\mu,s}(L,z)&=\sum_{n\in \Z} L^{\mu,s}_nz^{-n-2},\\ Y^{\mu,s}(G^{\{v\}},z)&=\sum_{n\in 1/2+\Z} (G^{\{v\}})^{\mu,s}_nz^{-n-3/2}, \\
Y^{\mu,s}(J^{\{u\}},z)&=\sum_{n\in \Z} (J^{\{u\}})^{\mu,s}_nz^{-n-1}
\end{align*} endow  $N(\mu)$ with a $W^k_{\min}(\g)$--module structure. Moreover,  the Hermitian form $(\,.\,,.\,)_\mu$ on $N(\mu)$ is invariant. 
\end{proposition}
\begin{proof} Plug $t=s/2$ in Proposition \ref{FairlieW}.
By \eqref{INV1}, \eqref{INV2}, and \eqref{INV3}, we have
\begin{align*}
(m,L^{\mu,s/2}_nm')_\mu&=(L^{\mu,s/2}_{-n}m,m')_\mu,\\
(m,(J^{\{u\}})^{\mu,s/2}_nm')_\mu&=-((J^{\{\phi(u)\}})^{\mu,s/2}_{-n}m,m')_\mu,\\
(m, (G^{\{ v \}})^{\mu,s/2}_nm')_\mu
        &=((G^{\{ \phi(v) \}})^{\mu,s/2}_{-n}m,m')_\mu.
\end{align*}
thus the representations $N(\mu)$ acquire a $W^k_{\min}(\g)$--module structure and the Hermitian form $(\cdot, \cdot)_\mu$ is $\phi$--invariant. 
\end{proof}

\section{Sufficient conditions for  unitarity 
of modules over $W^k_{\min}(\g)$}\label{ur}
Due to the Proposition \ref{casi} (a), we may assume in this section that $\g\ne sl(2|m)$ and $osp(4|m),\, m>2$. Then, in particular,   $\g^\natural=\oplus_{i\ge 1} \g^\natural_i$ is  the decomposition of $\g^\natural$ into simple ideals, and the $\chi_i$ are given by \eqref{chiiii}.
\begin{proposition}\label{necessary} Assume that $k+h^\vee\ne0$. Then there exists a unitary  module  $L^W(\nu,\ell_0)$ over $W^k_{\min}(\g)$ if and only if  $M_i(k)\in\ZZ_+$ for all $i$ and $\nu \in P^+_k$.  
%Moreover these modules are unitary for $l_0\gg 0$ if $k$ is not a collapsing level.
\end{proposition}
\begin{proof}
One implication has been already proven in Proposition \ref{l0nec}.
To show that the converse implication also holds, assume $M_i(k)\in \ZZ_+$ for all $i$. 
Recall (see \eqref{abb}) that the cocycle $\a_k$ is  given by
$${\a_k}_{|\g^\natural_i\times \g^\natural_i}=(M_i(k)+\chi_i)(.|.)^\natural_i.
$$
Assume first that $M_i(k)+\chi_i\in\ZZ_+$ for all $i$. Then the simple quotient $V_{\a_k}(\g^\natural)$ of $V^{\a_k}(\g^\natural)$ is unitary, since it is an integrable $\widehat \g^\natural$--module \cite{Kacsuper}. Next, the vertex algebra $F(\g_{1/2})$ is unitary due to Proposition \ref{fh} and \cite[\S 5.1]{KMP}. Finally, the $V^1(\C a)$--module $M(s)$ , where $s$ is given by \eqref{sfix}, is unitary by the observation following Lemma \ref{64}. 
\par
Consider the unitary $W^k_{\min}(\g)$--module
$ M(s)\otimes V_{\a_k}(\g^\natural)\otimes F(\g_{1/2}),
$ and its submodule
$$U=\Psi(W^k_{\min}(\g)). (v_{s}\otimes \vac\otimes\vac).$$
Since the Hermitian form $H_s(\,.\,,\,.)$ is $\widehat L(s)$--invariant and $\Psi(L)=\widehat L(s)$, we see that $U$ admits a $\phi$--invariant Hermitian positive definite form, thus 
$U$ is a unitary highest weight module for $W^k_{\min}(\g)$.

Now we look at the missing cases, where there is $i$ such that $0\le M_i(k)< -\chi_i$, described in Remark \ref{7.4}. Assume first that $\g^\natural$ is simple. If $\chi_1=-1$  then the only possible value is $M_1(k)=0$,
 so,  $W_k^{\min}(\g)=\C$, by Theorem \ref{oldresults} (1) (a).  In the case of $\g=spo(2|3)$ one should consider the cases $M_1(k)=1$ and $M_1(k)=0$: in the former case 
 $k=-\frac{h_1^\vee}{2}-1$, hence Theorem \ref{oldresults} (1) (b) applies and $W_k^{\min}(spo(2|3))=V_1(sl(2))$, whereas in the latter case $k+h^\vee=0$. If $\g^\natural$ is semisimple but not simple, then $\g=D(2,1;a)$.
In this case we have to consider only the case in which either $M_1(k)$ or $M_2(k)$ is zero. If $M_1(k)=0$ (resp. $M_2(k)=0$) then, by Theorem \ref{oldresults} (2), $W_k^{\min}(D(2,1;a))=V_{M_2(k)}(sl(2))$ (resp. $=V_{M_1(k)}(sl(2))$).
\end{proof}

We now generalize the construction given in the proof of Proposition \ref{necessary} to provide families of unitary representations. For
$\nu\in P^+_k$ introduce the following number
\begin{equation}\label{bknu}
 B(k,\nu)=\frac{(\nu|\nu+2\rho^\natural)}{2(k+h^\vee)}-\frac{(k+1)^2}{4(k+h^\vee)}.
 \end{equation}
 
\begin{proposition}\label{sufficient} Assume that $k+h^\vee\ne 0$ and $M_i(k)+\chi_i\in\ZZ_+$ for all $i>0$. 
%Let $\rho^\natural$ be the $\rho$--vector for $\g^\natural$.
If $\nu\in  P^+$ is such that $\nu(\theta_i^\vee)\le M_i(k)+\chi_i $ for all $i>0$ (then $\nu \in P^+_k$) and
% then
%$N(\mu+s,\nu)$ is unitary. 
\begin{equation}\label{m2}
\ell_0\ge B(k,\nu),
\end{equation}
then $L^W(\nu,\ell_0)$ is a unitary $W^k_{\min}(\g)$--module.
\end{proposition}
\begin{proof}Let $L^\natural(\nu)$ be the irreducible highest weight $V^{\a_k}(\g^\natural)$--module of highest weight $\nu$ and let $v_\nu$ be a highest weight vector.
Fix $\mu\in\R$ and set
$$
N(\mu,\nu)=\Psi(W^k_{\min}(\g)). (v_{\mu+s}\otimes v_\nu\otimes\vac)\subset M(\mu+s)\otimes L^\natural(\nu)\otimes F(\g_{1/2}),
$$
where $s=s_k$ is given by formula \eqref{sfix}. Note that the Hermitian form $(\cdot,\cdot)_{\mu+s}$ is $\widehat L(s)$--invariant.   Since $M_i(k)+\chi_i\in\ZZ_+$ and $\nu(\theta_i^\vee)\le M_i(k)+\chi_i $ for all $i$, then $L^\natural(\nu)$ is integrable for $V^{\a_k}(\g^\natural)$, hence  unitary \cite{VB}. Thus $N(\mu,\nu)$ is a unitary representation of $W^k_{\min}(\g)$.

We now compute the highest weight of $N(\mu,\nu)$.
Recall that
$$
\Psi(J^{\{h\}} )=h + \frac{1}{2}
      \sum_{\alpha \in S_{1/2}}: \Phi^{\alpha}
      \Phi_{[u_{\alpha},h]}:.
$$
By the $-1$--st product  identity,
$$
: \Phi^{\alpha}
      \Phi_{[u_{\alpha},h]}:_0=\sum_{j\in \tfrac{1}{2}+\ZZ_+}\left(\Phi^{\a}_{-j}(\Phi_{[u_{\alpha},h]})_j-(\Phi_{[u_{\alpha},h]})_{-j}\Phi^{\a}_{j}\right)
$$
so 
$$\Psi(J^{\{h\}})_0. (v_{\mu+s}\otimes v_\nu\otimes\vac)=\nu(h)(v_{\mu+s}\otimes v_\nu\otimes\vac).
$$
It follows that $N(\mu,\nu)=L^W(\nu,\ell_0)$ for some $\ell_0$. We now compute $\ell_0$:
$$
L_0(v_{\mu+s}\otimes v_\nu\otimes\vac)=\left(\frac{\mu^2-s^2}{2}+\frac{(\nu|\nu+2\rho^\natural)}{2(k+h^\vee)}\right)(v_{\mu+s}\otimes v_\nu\otimes\vac)
$$
so that, using \eqref{sfix},
$$
\ell_0=\frac{\mu^2-s^2}{2}+\frac{(\nu|\nu+2\rho^\natural)}{2(k+h^\vee)}=\frac{\mu^2}{2}-\frac{(k+1)^2}{4(k+h^\vee)}+\frac{(\nu|\nu+2\rho^\natural)}{2(k+h^\vee)}.
$$
Hence
$
\ell_0\ge B(k,\nu)$. 
Letting 
$\mu=2\sqrt{\ell_0-B(k,\nu)},$
we see that the module $L^W(\nu,\ell_0)=N(\mu,\nu)$ is unitary.
\end{proof}

\section{Unitarity of minimal $W$--algebras and modules over them}
The main result of this paper is the following.
\begin{theorem}\label{u??} Let $k\ne -h^\vee$, and recall the number $A(k,\nu)$ given by \eqref{Aknu}.
If k lies in the unitary range (hence $M_i(k)\in\mathbb Z_+$ for $i\ge 1$), then
the $W^k_{\min}(\g)$--module $L^W(\nu,\ell_0)$ is unitary for all non extremal $\nu\in \widehat P_k^+$ and $\ell_0\geq A(k,\nu)$. 
%Consequently, $W_k^{\min}(\g)$ is a unitary vertex algebra if and only if $k$ lies in the unitary range.
\end{theorem}
\begin{cor}\label{u?} 
If k lies in the unitary range, then
the $W^k_{\min}(\g)$--module $L^W(0,\ell_0)$ is unitary for all $\ell_0\geq 0$. 
Consequently, $W_k^{\min}(\g)$ is a unitary vertex algebra if and only if $k$ lies in 
the unitary range.
\end{cor}
In the rest of this section we give a proof of these results. 
First, by Proposition \ref{casi} (a), we may exclude $\g=sl(2|m),\ m>2$, from consideration, so that $\g^\natural$ is semisimple and by Proposition \ref{l0nec}, conditions $M_i(k)\in\mathbb Z_+$  are necessary for unitarity, hence we shall assume that these conditions hold.
Let $\ga=	(\C[t,t^{-1}]\otimes \g)\oplus \C K \oplus \C d$ be the affinization of $\g$ (with bracket 
$[t^m\otimes a,t^n\otimes b]=t^{n+m}\otimes[a,b]+\d_{m,-n}m K(a|b),\,a,b\in \g$). Let $\ha=\h\oplus \C K \oplus \C d$ be its Cartan subalgebra. Define  $\L_0$ and $\d\in \ha^*$ setting $\L_0(\h)=\L_0(d)=\d(\h)=\d(K)=0$ and $\L_0(K)=\d(d)=1$. Let $\Da\subset\ha^*$ be the set of roots of $\ga$. As a subset  of simple roots for $\ga$ we choose $\Pia=\{\a_0=\d-\theta\}\cup\Pi$, where $\Pi$ is the set of simple roots for $\g$ given in Table 1. We denote by $\Dap$ the corresponding set of positive roots and by  $\widehat \rho\in \ha^*$ the corresponding $\rho$--vector.\par

For
$\nu\in P^+_k$ and $h\in\C$, set 
\begin{equation}\label{nuh}
\widehat\nu_h=k\L_0+\nu+h\theta\in \ha^*.
\end{equation}
Let $\widehat{\mathfrak p}$ be the parabolic subalgebra of $\ga$ with Levi factor $\ha+\g^\natural$ and the nilradical $\widehat{\mathfrak u}_+=\sum_{\a\in\Dap\setminus\D^\natural}\ga_\a$.  Set $\widehat{\mathfrak u}_-=\sum_{\a\in\Dap\setminus\D^\natural}\ga_{-\a}$. Let $V^\natural(\nu)$ denote the irreducible $\g^\natural$--module with highest weight $\nu$ and extend the $\g^\natural$ action to 
 $\widehat\p$
by letting 
$\widehat{\mathfrak u}_+$ act trivially; $x$, $K$, and $d$ act by $h$, $k$, and $0$ respectively. Let $M^\natural(\widehat\nu_h)$ be the corresponding generalized Verma module for $\ga$, i.e.
$$M^\natural(\widehat\nu_h)=U(\ga)\otimes_{U(\widehat\p)} V^\natural(\nu).$$ 

We denote by $v_{\widehat \nu_h}$ a highest weight vector for $M^\natural(\widehat \nu_h)$ .
If  $\widehat\mu\in  \ha^*$ and $M$ is a $\ga$--module, we denote by 
 $M_{\widehat\mu}$  the corresponding weight space.
Let $\eta_i=\d-\theta_i, 1\leq i\leq s$ (recall that $s=1$ or $2$).
%Assume that $\lambda\in P^+_\natural$ and  set $\nu=\l_{|\h^\natural}$. 
%Fix an ordered basis of root vectors for $\widehat{\mathfrak u}_-$ and let  $\{b_j\}$ be the corresponding PBW-basis of $U(\widehat{\mathfrak u}_-)$. Fix also  a basis $\{z_j\}$ of $V^\natural(\nu)$  consisting of weight vectors. For $\widehat\nu\in  \ha^*$, set $\mathcal B_{\widehat\nu}=\{b_j\otimes z_r\}\cap M^\natural(\widehat \nu_h)_{\widehat\mu}$. Then  $\mathcal B_{\widehat\mu}$ is a basis of $M^\natural(\widehat \nu_h)_{\widehat\mu}$.

If $\a\in\Da$ is a non-isotropic root, denote by $s_\a\in End(\ha^*)$ the corresponding reflection and the group generated by them by $\Wa$. 
If $\be\in\Da\setminus \mathbb Z\d$ is an odd isotropic root, we let $r_{\be}$ denote the corresponding odd reflection. We denote by $x_\a$ a root vector attached to $\a\in \Da$. Denote by $w.$ the shifted action of $\Wa$: $w.\l= w(\l+\widehat \rho)-\widehat \rho$.
 
  \begin{lemma}\label{oddref} Let $\widehat \Pi'$ be a set of simple roots for $\Da$. Let $M$ be a $\ga$--module and assume that $m\in M$ is a singular  vector with respect to $\widehat \Pi'$. If $\a_j\in\widehat \Pi'$ is an isotropic root and  $x_{-\a_j}m\ne0$, then $x_{-\a_j}m$ is a singular vector with respect to $r_{\a_j}(\widehat \Pi')$.
 \end{lemma}
 \begin{proof}
 Since $\a_j$ is odd isotropic, it follows that $x_{-\a_j}^2m=0$. If $r\ne j$ and $(\a_r|\a_j)=0$ then $x_{\a_r}x_{-\a_j}m=x_{-\a_j}x_{\a_r}m=0$. If $r\ne j$ and $(\a_r|\a_j)\ne0$ then $x_{\a_r+\a_j}x_{-\a_j}m=x_{-\a_j}x_{\a_r+\a_j}m+x_{\a_r}m=0$.
 \end{proof}

For  $\nu\in P^+_k$ set
\begin{equation}\label{Ni}
N_i(k,\nu)=(\widehat \nu_h+\widehat\rho|\eta_i^\vee).
\end{equation}
Note that $N_i(k,\nu)$ does not depend on $h$. We will simply write $N_i$ when the dependence on $k$ and  $\nu$ is clear from the context.

\begin{lemma}\label{Submodule}For $\nu\in P^+_k$ not extremal,  we have 
\begin{equation}\label{10}
N_i(k,\nu)=M_i(k)+\chi_i+1-(\nu|\theta_i^\vee)\in\nat.
\end{equation}
Moreover, for 
$$v_i(h):=x_{-\eta_i}^{N_i}x_{-\a_0-\a_1}x_{-\a_1}v_{\widehat\nu_h},
$$
 the subspace $\sum_iU(\ga)v_{i}(h)$ is a proper submodule of the $\ga$-module $M^\natural(\widehat\nu_h)$.
\end{lemma}
\begin{proof}

Note that
\begin{align*}
(\widehat \nu_h+\widehat\rho|\eta_i^\vee)&=\frac{2}{(\theta_i|\theta_i)}(k+h^\vee)-(\nu+\rho|\theta_i^\vee)=\frac{2}{(\theta_i|\theta_i)}(k+\frac{h^\vee-\bar h_i^\vee}{2}+\frac{h^\vee+\bar h_i^\vee}{2}-(\nu+\rho^\natural|\theta_i))\notag\\
&=M_i(k)+\frac{2}{(\theta_i|\theta_i)}(\frac{h^\vee-\bar h_i^\vee}{2}+\frac{(\theta_i|\theta_i)}{2})-(\nu|\theta_i^\vee)\\
&=M_i(k)+\chi_i+1-(\nu|\theta_i^\vee).\notag
\end{align*}
Since  $\nu$ is not extremal, $(\widehat \nu_h+\widehat\rho|\eta_i^\vee)\in\mathbb N$.

Recall from Table 1 the set $\Pi$ of simple roots for $\g$. Let $\a_1$ be an odd root in $\Pi$. 
A direct (easy) verification shows that
$\a_0+\a_1$ is an odd root and that
the set of simple roots $r_{\a_0+\a_1}(r_{\a_1}(\widehat \Pi))$ contains both $\a_0$ and $\{\eta_i\mid 1\le i\le s\}$. 
Clearly $x_{-\a_0-\a_1}x_{-\a_1}v_{\widehat \nu_h}\ne0$ in $M^\natural(\nu_h)$ so, by Lemma \ref{oddref}, 
$x_{-\a_0-\a_1}x_{-\a_1}v_{\widehat \nu_h}$ is a singular vector for the set of simple roots
$r_{\a_0+\a_1}(r_{\a_1}(\widehat \Pi))$. The weight of this singular vector is, clearly, $\widehat \nu_h'=\widehat \nu_h-\a_0-2\a_1$.
%By \eqref{typical}, 
%$$(\widehat \nu_h|\a_1)=(\widehat \nu_h+\widehat\rho|\a_1)\ne0$$ and
%$$(\widehat \nu_h-\a_1|\a_0+\a_1)=(\widehat \nu_h-\a_1+\widehat\rho+\a_1|\a_0+\a_1)=(\widehat \nu_h+\widehat\rho|\a_0+\a_1)\ne0,
%$$
Since the $\rho$--vector $\widehat \rho'$ of $r_{\a_0+\a_1}(r_{\a_1}(\widehat \Pi))$ is $\widehat \rho+\a_0+2\a_1$, we see that $(\widehat \nu_h'+\widehat \rho'|\eta_i^\vee)=(\widehat \nu_h+\widehat \rho|\eta_i^\vee)=N_i$.
Since $\eta_i$ is a simple root in $r_{\a_0+\a_1}(r_{\a_1}(\widehat \Pi))$,  we obtain that $x_{-\eta_i}^{N_i}x_{-\a_0-\a_1}x_{-\a_1}v_{\widehat \nu_h}$ is a singular vector for  the set of simple roots $r_{\a_0+\a_1}(r_{\a_1}(\widehat \Pi))$. It follows that $\sum_iU(\ga)v_{i}(h)$ is a proper submodule of $U(\ga)x_{-\a_0-\a_1}x_{-\a_1}v_{\widehat \nu_h}\subset M^\natural(\widehat\nu_h)$.
\end{proof}

Set
\begin{equation}\label{modulo}
\overline{M}(\widehat \nu_h)=M^\natural(\widehat \nu_h)/(\sum_iU(\widehat \g)v_i(h)).
\end{equation}
Recall (cf. \cite{KK} in the non-super case) that for $\widehat \mu, \widehat\l\in \ha^*$, $\widehat \mu$ is said to be  linked to $\widehat \l$  if there exists a sequence of roots   $\{\gamma_1,\ldots,\gamma_t\}\subset\Dap$  and weights $\widehat\l=\mu_0,\mu_1\ldots, \mu_t=\widehat \mu$ such that, for $1\leq r\leq t$ one has

\begin{itemize}
\item $ (\mu_{r-1}+\widehat\rho|\gamma_r)=\frac{m_r}{2}(\gamma_r|\gamma_r),\ m_r\in\nat$, where $m_r=1$ if $\gamma_r$ is an odd isotropic root and 
$m_r$ is odd  if $\gamma_r$ is an odd non-isotropic root,
\item $\mu_r=\mu_{r-1}-m_r\gamma_r$.
\end{itemize}
The proof of the following proposition is inspired by \cite[Section 11]{GK2}. It also provides a simple proof of Lemma 2 from \cite{J}.
\begin{proposition}\label{irr} 
Assume that  $\nu\in P^+_k$ is  not extremal and that 
\begin{equation}\label{gi}
(\widehat \nu_h+\widehat \rho|\a)\ne \tfrac{n}{2}(\a|\a)\text{ for all $n\in \mathbb N$ and
$\a\in \Dap\setminus \Dap(\g^\natural)$.}
\end{equation}
 Then 
\begin{enumerate}
\item[(i)] the module $\overline M(\widehat \nu_h)$ is irreducible;
\item[(ii)]  its character is 
\begin{equation}\label{chf}ch \overline M(\widehat \nu_h)=\sum_{w\in \Wa^\natural} \det(w) ch M(w.\widehat \nu_h).\end{equation}
\end{enumerate}\end{proposition}
\begin{proof} We have 
\begin{enumerate}
\item $(\widehat \nu_h+\widehat \rho|\a)\neq 0$ for all odd isotropic roots;
\item $(\widehat \nu_h+\widehat \rho|\a^\vee)\in \nat$ for all $\a\in\Dap(\g^\natural)$;
\item $(\widehat \nu_h+\widehat \rho|\a)\ne \tfrac{n}{2}(\a|\a)$ for all $n\in \mathbb N$ and for all
 positive roots $\a$ of the affinization of  $sl_2=\langle e,f,x\rangle$ and for all non-isotropic odd positive roots.
\end{enumerate}
Indeed, (1), (3) follow from \eqref{gi}. To prove (2), first remark that if $\a$ is a simple root for $\D^\natural_+$, then 
$\a\in\Pia$. It follows that $(\widehat\rho|\a^\vee)=(\rho^\natural|\a^\vee)=1$. This implies that 
 $(\widehat \nu_h+\widehat \rho|\a^\vee)=(\nu+\rho^\natural|\a^\vee)\in\nat$ for $\a\in \D_+^\natural$.
Since $\nu$ is not extremal, \eqref{10}  gives  $(\widehat \nu_h+\widehat \rho|\eta_i^\vee)\in\nat$.\par
We have 
\begin{equation}\label{iop} ch\,\overline M(\widehat \nu_h)=\sum_{w\in \Wa^\natural} c(w) ch M(w.\widehat \nu_h),\text{\it where $c(w)\in \mathbb Z$}.\end{equation}
Indeed, if   $ch\,M(\widehat\mu)$ appears in $ch \,\overline M(\widehat \nu_h)$ 
then, using the determinant formula proved in \cite{GK}, and the 
corresponding Jantzen filtration \cite{J}, one shows, as in \cite{KK}, that there  is a sequence of roots   $\{\gamma_1,\ldots,\gamma_t\}\subset\Dap$
 linking $\widehat \mu$ to  $\widehat \nu_h$. Properties  (1), (3) imply that $\gamma_i\in\Dap(\g^\natural)$  and this yields \eqref{iop}.
%If $\a$  is a simple root in $\D_+^\natural$, then   $\a\in\Pia$. 
It is clear that $\g^\natural$ acts locally finitely on $M^\natural(\widehat\nu_h)$, hence also on $\overline M(\widehat\nu_h)$. By (1), $x_{-\a_0-\a_1}x_{-\a_1}v_{\widehat\nu_h}$ generates $\overline M(\widehat\nu_h)$. Since $x_{-\eta_i}^{N_i}(x_{-\a_0-\a_1}x_{-\a_1}v_{\widehat\nu_h})=v_i(h)=0$ in $\overline M(\widehat\nu_h)$,
 $\overline M(\widehat \nu_h)$ is integrable for 
$\ga^\natural$, in particular
 $ch \,\overline M(\widehat \nu_h)$ is $\Wa^\natural$--invariant.  Hence, we obtain $c(w)=\det(w)$; therefore (ii) holds.
Since the proof of (ii) didn't use irreducibility,
the irreducible quotient of $\overline M(\widehat \nu_h)$ has the same character,
proving (i).
\end{proof}
 The following functions $h_{n, \epsilon m},\, h_{m,\gamma }$ relate singular weights of Verma modules over $\ga$ to those over  $W^k_{\min}(\g)$ \cite[Remark 7.2]{KW1}:
\begin{align}
  h_{n, \epsilon m} (k,\nu) &= \frac{1}{4(k+h^\vee)} (( \epsilon  m
  (k+h^\vee)-n)^2 - (k+1)^2 + 2(\nu |\nu
  +2\rho^{\natural})),\, \label{1244}\\
    h_{m,\gamma }(k,\nu) &= \frac{1}{4(k+h^\vee)}
  ((2(\nu + \rho^{\natural}|\gamma)+2m(k+h^\vee))^2
  -(k+1)^2 + 2(\nu |\nu +2 \rho^{\natural})  ) \, .  \label{123}
\end{align}
Here $\gamma\in\D'$, the set of $\g^\natural$--weights in $\g_{-1/2}$, $\epsilon=2$ (resp. $1$) if $0\in \D'$ (resp. $0\notin \D'$), $m,n\in \epsilon^{-1}\mathbb N$ and $m-n\in \mathbb Z$ in \eqref{1244}  and $m\in \tfrac{1}{2}+\mathbb Z_+$ in \eqref{123}. 
\begin{lemma}\label{sign2} Let $k$ be in the unitarity range and let $A(k,\nu)$ be as in \eqref{Aknu}. Assume that $\nu$ is not extremal. Then
\begin{align}
\label{1262}h_{n,\epsilon m} (k,\nu)&\le A(k,\nu),\\
\label{1261}h_{m,\gamma }(k,\nu) &\le A(k,\nu).
\end{align}
\end{lemma}
\begin{proof} First we prove \eqref{1262}. Plugging \eqref{1244} into \eqref{1262} we get
$$\frac{( \epsilon  m
  (k+h^\vee)-n)^2 - (k+1)^2 + 2(\nu |\nu
  +2\rho^{\natural})}{4(k+h^\vee)} \leq \frac{(\nu|\nu+2\rho^\natural)}{2(k+h^\vee)}+\frac{(\xi|\nu)}{k+h^\vee}((\xi|\nu)-k-1),$$
  which is equivalent to 
\begin{equation}\label{prova}n-\epsilon  m(k+h^\vee)\ge |(k+1) -2(\xi|\nu)|.\end{equation}
Since $k+h^\vee<0$, it is enough to check \eqref{prova} with $\e m=1,\ n=1/\e$. In the case $(k+1) \le2(\xi|\nu),$ \eqref{prova} reads
\begin{equation}\label{parz}1/\e-h^\vee\ge 2(\xi|\nu)-1.\end{equation}
Looking at the values of $h^\vee$ in Table 2, we see that the L.H.S. of \eqref{parz} is non-negative. Now we prove that $(\xi|\nu)\leq 0$. Indeed, from Table 1 we deduce that  the restriction of $(. |.)$ to the real span of $\D^\natural$ is negative definite.
 From Tables 1 and 3 one checks that $\xi$ is a   linear combination with non-negative coefficients of simple roots of $\g^\natural$; since $\nu$ is dominant, if $\a\in\D^\natural$ is a simple root then $\nu(\a^\vee)\ge 0$, hence $(\nu|\a)\le0$ since $(\a|\a)<0$.
In the case  $(k+1) \ge 2(\xi|\nu)$ we have to prove that 
%\begin{equation}\label{parz2}1/\e-k-h^\vee\ge k -2(\xi|\nu)+1.\end{equation}
\begin{equation}\label{parziale} k+\tfrac{h^\vee}{2} \le (\xi|\nu) +\tfrac{1-\e}{2\e}.\end{equation}
The non-extremality condition means that $(\nu+\xi)(\theta_i^\vee)\le \frac{2}{(\theta_i|\theta_i)}\left(k+\frac{h^\vee-\bar h^\vee_i}{2}\right)$ or
\begin{equation}\label{nonextr}k+\tfrac{h^\vee}{2}\le  (\nu+\xi|\theta_i)+\tfrac{\bar h^\vee_i}{2},\end{equation}
hence it is enough to prove that 
\begin{equation}	\label{daprovare}(\nu+\xi|\theta_i)+\tfrac{\bar h^\vee_i}{2}\leq  (\xi|\nu) +\tfrac{1-\e}{2\e}.\end{equation}
Note that  $\theta_i=\xi+\beta_i$, where, as above, $\beta_i$  is a linear combination with non-negative coefficients of simple roots of $\g^\natural$. Therefore \eqref{daprovare} can be written as
$$(\nu|\beta_i )+(\xi|\xi+\beta_i)\leq   \tfrac{1-\e}{2\e}-\tfrac{\bar h^\vee_i}{2},$$
%which is clearly verified if $\e=1$, since the left hand side is negative and  $\bar h^\vee_i\leq 0$. The two remaining cases, in which $\e=2$, are dealt with a direct analysis.
which is clearly verified, since the left hand side is negative and the right hand side is positive (use the data in Table 2). 

Now we prove \eqref{1261}.  Substituting  \eqref{1244} in it  we obtain 
$$
  (2(\nu + \rho^{\natural}|\gamma)+2m(k+h^\vee))^2 - \left((k+1) -2(\xi|\nu)\right)^2\ge 0,$$
  which is equivalent to 
\begin{equation}\label{qwee}
|(2(\nu + \rho^{\natural}|\gamma)+2m(k+h^\vee)|\geq |(k+1) -2(\xi|\nu)|.\end{equation}
%It suffices to prove \eqref{qwee} choosing the maximal weight $\xi$ so 
Recall that, even though $\g_{-1/2}$ can be reducible as a $\g^\natural$--module, all irreducible components have the same highest weight $\xi$. It follows that 
\begin{equation}\label{ximax}-(\xi|\nu)=\max_{\gamma\in\D'} (\gamma|\nu).\end{equation} 
A direct check on Table 4 shows that 
\begin{equation}\label{e11}2\max_{\gamma\in\D'}(\rho^{\natural}|\gamma)+h^\vee=1.\end{equation}
\begin{table}
{\scriptsize
\begin{tabular}{c |  c| c |c | c }
$\g$& $\e$& $\rho^\natural$&  $\max(\rho^\natural |\gamma)$& $h^\vee$\\
\hline
$psl(2|2)$& $1$&
$\tfrac{1}{2}(\d_1-\d_2)$ & $1/2$& $0$
\\\hline
$spo(2|2m), m\ge 3$& $1$&
$(m-1)\e_1+(m-2)\e_2+\ldots+\e_{m-1}$& $(m-1)/2$ & $2-m$
\\\hline
$spo(2|2m+1), m\ge 1 $& $2$&
$\tfrac{2m-1}{2}\e_1+\tfrac{2m-3}{2}\e_2+\ldots+\tfrac{1}{2}\e_m$& $(2m-1)/4$  & $3/2-m$
\\\hline
$D(2,1;a)$& $1$&  $\e_2+\e_3$ & $\tfrac{1}{2}$ & $0$\\\hline
%$\frac{\e_1-\e_2}{2}+\d$ 
 $F(4)$& $1$& $\tfrac{5}{2}\e_1+\tfrac{3}{2}\e_2+\tfrac{1}{2}\e_3$& $3/2$ & $-2$
\\\hline
$G(3)$& $2$& $2\e_1+3\e_2$& $5/4$& $-3/2$\\
\end{tabular}
}
\vskip5pt
\caption{Data employed in the proof of Lemma \ref{sign2}}

\end{table}
Note that, by \eqref{ximax} and \eqref{e11}
\begin{equation}\label{valesempre} (k+h^\vee)+2(\nu + \rho^{\natural}|\gamma)\leq (k+1)- 2(\xi|\nu).\end{equation}
Therefore, if   $(k+1) \leq 2(\xi|\nu)$ then \begin{align*}
&2(\nu + \rho^{\natural}|\gamma)+2m(k+h^\vee)=\\
&2(\nu + \rho^{\natural}|\gamma)+(k+h^\vee)+(2m-1)(k+h^\vee)\leq 2(\nu + \rho^{\natural}|\gamma)+(k+h^\vee)\leq (k+1) -2(\xi|\nu)\leq 0,\end{align*}
and \eqref{qwee} reads
$$2(\nu + \rho^{\natural}|\gamma)+2m(k+h^\vee)\leq (k+1) -2(\xi|\nu), $$
which is clearly true.

Now consider the case 
\begin{equation}\label{geq}(k+1) \geq 2(\xi|\nu), \quad -2(\nu + \rho^{\natural}|\gamma)-2m(k+h^\vee)\geq 0.\end{equation} The inequality  \eqref{qwee} becomes
\begin{equation*}\label{voo}
-2(\nu + \rho^{\natural}|\gamma)-2m(k+h^\vee)\geq (k+1) -2(\xi|\nu).
\end{equation*}
which is implied by 
\begin{equation}\label{vo}
-2(\nu + \rho^{\natural}|\gamma)-(k+h^\vee)\geq (k+1) -2(\xi|\nu).
\end{equation}
If  $\gamma=-\xi$, then the left hand side of \eqref{vo} is
$$-2(\nu + \rho^{\natural}|-\xi)-(k+h^\vee)=2(\nu|\xi)+h^\vee-1-(k+h^\vee)=2(\nu|\xi)-k-1,$$
hence \eqref{geq} implies that both members of \eqref{vo} are zero.
\par
If  $\gamma\ne-\xi$, then  \eqref{vo}  is equivalent to 
\begin{equation}
k+\tfrac{h^\vee}{2}\leq -\tfrac{1}{2} +(\xi|\nu)-(\nu + \rho^{\natural}|\gamma),\end{equation}
hence, by \eqref{nonextr}, we are done if we prove that 
\begin{equation}\label{dap3}
 (\nu+\xi|\theta_i)+\tfrac{\bar h^\vee_i}{2}\leq  -\tfrac{1}{2} +(\xi|\nu)-(\nu + \rho^{\natural}|\gamma).
\end{equation}
 Remark that,  since $\gamma\ne -\xi$, then $\xi-\gamma=\a\in \D_+^\natural\cup\{0\}$. If $\g^\natural$ is simple, then $(\nu|\theta_i)\leq (\nu|\xi-\gamma)$, hence
 $$ (\nu+\xi|\theta_i)+\tfrac{\bar h^\vee_i}{2}\leq (\nu|\xi-\gamma)+(\xi|\theta_i)+\tfrac{\bar h^\vee_i}{2},$$
and therefore \eqref{vo} is implied by 
$$(\nu|\xi-\gamma)+(\xi|\theta_i)+\tfrac{\bar h^\vee_i}{2}\leq -\tfrac{1}{2} +(\xi|\nu)-(\nu + \rho^{\natural}|\gamma),$$
 or
 \begin{equation}\label{dap4}
(\xi|\theta_i)+\tfrac{\bar h^\vee_i}{2}\leq -\tfrac{1}{2} -(\rho^{\natural}|\gamma).
 \end{equation}
The minimum of the left hand side of \eqref{dap4} is obtained when $(\rho^{\natural}|\gamma)$ is maximum, hence, by \eqref{e11}, we are left with proving that 
  \begin{equation}\label{dap44}
(\xi|\theta_i)+\tfrac{\bar h^\vee_i}{2}\leq \tfrac{h^\vee}{2}-1.
 \end{equation}
 This relation is checked using the data in Tables 1,2,3.  When $\g^\natural$ is not simple, i.e. $\g=D(2,1;a)$, relation \eqref{vo} is proven directly.
 We have $\nu=r\e_2+s\e_3,\,r,s\in\mathbb Z_+\,\ \gamma=\pm\e_2\pm\e_3,\ \xi=\e_2+\e_3$; if we exclude $\gamma=-\xi$,  \eqref{vo}  translates into 
\begin{equation}\label{dap5}
k\leq 0,\quad k\leq -\tfrac{r+1}{1+a},\quad k\leq -\tfrac{(s+1)a}{1+a},
\end{equation}
according to whether  $\gamma=\e_2+\e_3,\,\e_2,\,\e_3$. The non extremality conditions are 
\begin{equation}\label{dap6}
k\leq -\tfrac{r+2}{1+a},\quad k\leq -\tfrac{(s+2)a}{1+a},
\end{equation}
so that \eqref{dap6} implies 	\eqref{dap5}.

\par
We are left with  proving  \eqref{qwee} when both arguments in the absolute values are non-negative, i.e.
\begin{equation}\label{qweebis}
2(\nu + \rho^{\natural}|\gamma)+2m(k+h^\vee)\geq (k+1) -2(\xi|\nu).\end{equation}
We claim that the conditions $2(\nu + \rho^{\natural}|\gamma)+2m(k+h^\vee)\geq 0$ combined with \eqref{nonextr} force $m=1/2$
and $\gamma=-\xi$.
 Taking this fact for granted,  \eqref{qweebis} reads
  $$
 -2(\nu + \rho^{\natural}|\xi)+h^\vee-1\geq -2(\xi|\nu),
$$
 which holds by \eqref{ximax} and \eqref{e11}.\par
To prove our claim, assume that there is $m>1/2$ such that  
%\eqref{qweebis} as 
$$k\geq \frac{1 -2mh^\vee-2(\xi|\nu)-2(\nu + \rho^{\natural}|\gamma)}{2m-1},$$
or
$$k+\tfrac{h^\vee}{2}\geq \frac{2 -(2m+1)h^\vee-4(\xi|\nu)-4(\nu + \rho^{\natural}|\gamma)}{2(2m-1)}.$$
Taking \eqref{nonextr} into account, we are done if we prove that  
\begin{equation}\label{finalissima}
\frac{2 -(2m+1)h^\vee-4(\xi|\nu)-4(\nu + \rho^{\natural}|\gamma)}{2(2m-1)}> 
(\nu+\xi|\theta_i)+\tfrac{\bar h^\vee_i}{2}.\end{equation}
We have 
\begin{align}\notag\text{L. H. S. of \eqref{finalissima}}&\geq 
\frac{2 -(2m+1)h^\vee-4(\xi|\nu)-4(\nu |\gamma)+2(h^\vee-1)}{2(2m-1)}\\&=-\tfrac{h^\vee}{2}+\frac{-4(\xi|\nu)-4(\nu |\gamma)}{2(2m-1)}\geq 
-\tfrac{h^\vee}{2}\geq \tfrac{\bar h^\vee_i}{2} \geq (\nu+\xi|\theta_i)+\tfrac{\bar h^\vee_i}{2}.\label{??}
\end{align}
The next to last inequality in \eqref{??} follows from Table 2; more precisely, the strict inequality holds in all cases except for  $spo(2|3)$. The  last  inequality  in \eqref{??}
uses that $(\nu+\xi|\theta_i)\leq 0$. For $\g=spo(2|3)$ the last inequality in \eqref{??} is strict, hence \eqref{finalissima} is proven in all cases.

Hence we have necessarily $m=1/2$ in \eqref{qweebis}. 
%We have to prove that $\gamma=-\xi$. 
We now  prove that if 
\begin{align}\label{c1}&2(\nu + \rho^{\natural}|\gamma)+(k+h^\vee)\geq 0,\\&(k+1) -2(\xi|\nu)\geq 0,\label{c2}\end{align}
hold, then  \eqref{nonextr} implies $\gamma=-\xi$.
%\begin{equation}\label{F}
%2(\nu |\xi+\gamma)\geq 1-h^\vee-2 (\rho^{\natural}|\gamma).
%\end{equation}
We proceed case by case.

$\bullet$ $\g=psl(2|2)$ or  $spo(2|3).$ Since $\gamma\in\{0,\pm\xi\}$, relation  \eqref{c1} forces $\gamma=-\xi$.
%), and \eqref{F} becomes $0\geq 1-h^\vee+2(\rho^\natural|\xi)=0$ by \eqref{ximax}.

$\bullet$  $\g=spo(2|m),\,m>4.$ In this case $\nu=\sum_i n_i\e_i, n_1\ge n_2\ge\ldots\ge 0,\ \gamma=\pm \e_i,\ \rho^\natural=\sum_i (\tfrac{m}{2}-i)\e_i,$ $ \xi=\e_1.$
Then \eqref{c1} reads 
$$2(\sum_i (n_i+\tfrac{m}{2}-i)|\pm \e_j)+k+2-\tfrac{m}{2}\geq 0,
$$
or
$$\mp(n_j+\tfrac{m}{2}-j)+k+2-\tfrac{m}{2}\geq 0.
$$
Since $k+2-\tfrac{m}{2}\le 0$, we have 
$$n_j-j+k+2\geq 0.
$$
By \eqref{nonextr} 
$$
k\le  -\tfrac{1}{2}n_1-\tfrac{1}{2}n_2-1,
$$
therefore
$$
0\leq n_j-j+k+2\leq -\tfrac{1}{2}n_1-\tfrac{1}{2}n_2+n_j-j+1.
$$
This relation can be written as 
$$
0\leq\tfrac{n_j-n_1}{2}+\tfrac{n_j-n_2}{2}-j+1,
$$
which holds only if $j=1$, since the $n_j$ are non-increasing half integers. If $j=1$ then $\gamma=-\epsilon_1=-\xi$.

$\bullet$  $\g=D(2,1;a).$ In this case  $\nu=r\e_2+s\e_3,\,r,s\in\mathbb Z_+,\, \gamma=\pm\e_2\pm\e_3,\ 
 \rho^\natural=\e_2+\e_3,\ \xi=\e_2+\e_3,$
and in this case \eqref{c1} becomes 
$$
2((r+1)\e_2+(s+1)\e_3|\pm\e_2\pm\e_3)+k\geq 0,$$
which gives
\begin{equation}\label{asd}\mp(r+1)\mp(s+1)a+(1+a)k\geq 0.\end{equation}
Condition \eqref{nonextr} is 
$$
k\le((r+1)\e_2+(s+1)\e_3|2\e_2)-\tfrac{1}{1+a},\quad k\le  ((r+1)\e_2+(s+1)\e_3|2\e_3)-\tfrac{a}{1+a},$$
%$$
%k\le -(r+2)\tfrac{1}{1+a},\quad k\le  -(s+2)\tfrac{a}{1+a}$$
or
\begin{equation}\label{dsa}
(1+a)k\le -(r+2),\quad (1+a)k\le  -(s+2)a.
\end{equation}
The only possibility to fulfill \eqref{asd} and \eqref{dsa} at the same time is to take $\gamma=-\e_2-\e_3=-\xi$. 

$\bullet$  $\g=F(4).$ In this case $\nu=n_1\e_1+n_2\e_2+n_3\e_3,\,n_1\ge n_2\ge n_3\ge 0,\ \rho^\natural=\tfrac{5}{2}\e_1+\tfrac{3}{2}\e_2+\tfrac{1}{2}\e_3,$ $ \gamma=\tfrac{1}{2}(\pm\e_1\pm\e_2\pm\e_3), \  \xi=\tfrac{1}{2}(\e_1+\e_2+\e_3)$.
Then \eqref{c1} reads 
\begin{equation}\label{bnm}
-\tfrac{2}{3}(\pm(n_1+\tfrac{5}{2})\pm(n_2+\tfrac{3}{2})\pm(n_3+\tfrac{1}{2}))+k-2\geq 0.
\end{equation}
By \eqref{nonextr} we have
\begin{equation}\label{mnb}k\le -\tfrac{2}{3}(n_1+n_2)-\tfrac{4}{3}.\end{equation}
Write now \eqref{bnm} using \eqref{mnb}
\begin{align*}0&\leq -\tfrac{2}{3}(\pm(n_1+\tfrac{5}{2})\pm(n_2+\tfrac{3}{2})\pm(n_3+\tfrac{1}{2}))+k-2\\&\leq-\tfrac{2}{3}(\pm(n_1+\tfrac{5}{2})\pm(n_2+\tfrac{3}{2})\pm(n_3+\tfrac{1}{2})) -\tfrac{2}{3}(n_1+n_2)-\tfrac{10}{3}\\
&\leq-\tfrac{2}{3}(\pm n_1-\tfrac{5}{2}\pm n_2-\tfrac{3}{2} \pm n_3-\tfrac{1}{2})-\tfrac{2}{3}(n_1+n_2)-\tfrac{10}{3}\\
%&\leq-\tfrac{2}{3}(\pm n_1\pm n_2 \pm n_3)-\tfrac{2}{3}(n_1+n_2)-\tfrac{1}{3}\\
&=-\tfrac{2}{3}(\pm n_1\pm n_2)-\tfrac{2}{3}(n_1+n_2\pm n_3)-\tfrac{1}{3}.
\end{align*}
This inequality holds if and only if the minus sign is taken in all occurrences of $\pm$, i.e. $\gamma=-\xi$. 

$\bullet$  $\g=G(3).$ In this case $\nu=m(\e_1+\e_2)+n(\e_1+2\e_2),\,m,n\in\mathbb Z_+,\gamma\in\{0, \pm\e_1,\pm\e_2,\pm(\e_1+\e_2)\},$ $\rho^\natural=2\e_1+3\e_2,\ \xi=\e_1+\e_2.$
Then \eqref{c1} reads 
\begin{align}\label{fgh}&2((m+n+2)\e_1+(m+2n+3)\e_2|\gamma)+k-\tfrac{3}{2}\geq 0.\end{align}
and we can  confine ourselves to consider $\gamma\in\{-\e_1,-\e_2,-\e_1-\e_2\}$. The  inequalities corresponding to $\gamma=-\e_1, \gamma =-\e_2$ are 
\begin{align}
\label{11111}&k+\tfrac{m}{2}-1\geq 0,\\
\label{2222}&k+\tfrac{m+3n+1}{2}\ge 0,\end{align}
respectively.
Relation \eqref{hgf} gives 
$$k\le  ((m+n+1)\e_1+(m+2n+1)\e_2|\e_1+2\e_2)-\tfrac{3}{4}.$$
or
\begin{equation}\label{hgf}k\le -\tfrac{3}{4}(m + 2 n)-\tfrac{3}{2}\end{equation}
Substituting 	\eqref{11111}, \eqref{2222}, into \eqref{hgf}  we obtain
\begin{align}
&0\leq k+\tfrac{m}{2}-1\leq -\tfrac{1}{4}m-\tfrac{3}{2}n-\tfrac{5}{2},\label{ee}\\
&0\leq k+\tfrac{m+3n+1}{2}\leq -\tfrac{m}{4}-1,\label{-}\end{align}
respectively.  Inequalities \eqref{ee}, \eqref{-} are never verified. Once again we conclude that $\gamma=-\xi$.
\end{proof}
Let $H_0$ denote the quantum Hamiltonian reduction functor, from the category $\mathcal O$ of $\ga$-modules of level $k$ to the category of 
$W^k_{\min}(\g)$-modules. Recall that, for a $\ga$--module $M$, $H_0(M)$ is the zeroth homology  of the complex $(M\otimes F(\g,x,f),d_0)$ defined in \cite{KRW}. Recall that the functor $H_0$ maps Verma modules to Verma modules \cite[Theorem 6.3]{KW1} and it is exact  
\cite[Corollary 6.7.3]{Araduke}.
By \cite[Lemma 7.3 (b)]{KW1}, if $M$ is a highest weight module over $\ga$ of highest weight $\L\in\ha^*$, $H_0(M)$ is  either zero or a  highest weight module over $W^k_{\min}(\g)$ of highest weight $(\nu,\ell)$ with
\begin{equation}\label{pgen}\nu=\L_{|\h^\natural},\quad \ell=\frac{(\L|\L+2\widehat \rho)}{2(k+h^\vee)}-\L(x+d).\end{equation}
\begin{remark}\label{nonzero} Let $L(\L)$ denote the irreducible $\ga$-module of highest weight $\L\in\ha^*$. By Arakawa's theorem \cite[Main Theorem]{Araduke} $H_0(L(\L))$ is either irreducible or zero, and it is zero if and only if  $(\L|\a_0)=\tfrac{n}{2}(\a_0|\a_0),\,n\in \mathbb Z_+$. 
 In particular, if \eqref{gi} holds, then  $H_0(\overline M(\widehat \nu_h))$ is a non-zero highest weight module of highest weight 
$(\nu, \ell(h))$, where
\begin{equation}\label{p}
\ell(h)=\frac{(\widehat \nu_h|\widehat \nu_h+2\widehat \rho)}{2(k+h^\vee)}-h.
\end{equation}
\end{remark}

%\begin{remark}\label{abuseofnotation}
For $\L\in\ha^*$, by a slight abuse of notation, we set $M^W(\L)=H_0(M(\L))$, where $M(\L)$ is the Verma module over $\ga$ of highest weight $\L$. Note that  
$M^W(\L)=M^W(\nu,\ell)$, where $\nu,\,\ell$ are given by \eqref{pgen}. 
%\end{remark}

From now on we assume
\begin{itemize}
\item $k$ is in the unitarity range;
\item $\nu\in P^+_k$;
\item $\ell(h)\in\R$.
\end{itemize}

\begin{lemma}\label{dmtheta} Let $h,h'$ be the solutions of the equation $\ell(h)=\ell_0$. If  $(\widehat\nu_h+\widehat\rho|\d-\theta)=n$, $n\in \nat$, then
$(\widehat\nu_{h'}+\widehat\rho|\d-\theta)\notin \nat$.
\end{lemma}
\begin{proof} Recalling that 
$$
\ell(h)=\frac{(\widehat \nu_h|\widehat \nu_h+2\widehat \rho)}{2(k+h^\vee)}-h=\frac{(\nu |\nu+2\rho^\natural)}{2(k+h^\vee)}+\frac{h(h-k-1)}{k+h^\vee}
$$
we see that $h'=k+1-h$. If
$
(\widehat\nu_h+\widehat\rho|\d-\theta)=n\in\nat
$, then 
$$
((k+h^\vee)\L_0+h\theta+\nu+\rho|\d-\theta)=k+1-2h=n,
$$
hence
$h=(k+1-n)/2
$ and  $h'=(k+n+1)/2$ so that 
$$
(\widehat\nu_{h'}+\widehat\rho|\d-\theta)=k+1-2h'=-n.
$$
\end{proof}

\begin{theorem}\label{chW} If $\ell(h)>A(k,\nu)$, then $H_0(\overline M(\widehat \nu_h))$ is an irreducible $W^k_{\min}(\g)$-module  and its character is 
\begin{equation}\label{chfW}
ch\,H_0(\overline M(\widehat\nu_h))=\sum_{w\in \widehat W^\natural}det(w)ch\,M^W(w.\widehat\nu_h).\end{equation}
\end{theorem}
\begin{proof} If $\ell(h)>A(k,\nu)$, then,  by Lemma \ref{sign2}
\begin{equation}\label{factorform1}
\ell(h)\ne h_{n,\epsilon m}(k,\nu)\quad\text{and}\quad \ell(h)\ne h_{m,\gamma}(k,\nu).
\end{equation}
 By \cite[Lemma 7.3 (c)]{KW1},  \eqref{factorform1} implies that
 $(\widehat\nu_h+\widehat\rho|\a)\ne \tfrac{n}{2}(\a|\a)$ for all $\a\in\Dap\setminus(\Dap(\g^\natural)\cup\{\d-\theta\})$. By exchanging $h$ and $h'$ if $h\in \nat$ and applying Lemma \ref{dmtheta}, we find that one can choose $h$ so that \eqref{gi} is satisfied.
Hence,   by  Propositions \ref{boundary} and  \ref{irr}, $\overline M(\widehat\nu_h)$ is irreducible. By Remark \ref{nonzero}, $H_0(\overline M(\widehat \nu_h))$ is irreducible
and non-zero. On the other hand, by Theorem 6.2 of \cite{KW1}, we find that $H^j((\overline M(\widehat\nu_h)\otimes F(\g,x,f)))=0$ if $j\ne0$. Thus, using Euler-Poincar\'e character, the fact that $H_0$ maps Verma modules over $\ga$ to Verma modules over $W^k_{\min}(\g)$,  and (ii) in Proposition \ref{irr}, we find that \eqref{chfW} holds.
 \end{proof}
 Recall from Section \ref{freeb} the Heisenberg algebra $\mathcal H$.   Let $y$ be an indeterminate. Define an action of $\mathcal H_0=\C a+\C K$ on $\C[y]$ by letting $K$ act as the identity and $a$ act by multiplication by $y$. Let $M(y)$ be the corresponding Verma module. This module can be regarded as a $V^1(\C a)$--module by means of the field $ Y(a,z)$ defined by setting, for $m\in M(y)$,
 $$
 Y(a,z)m=\sum_{j\in\ZZ}(\tau^{j}\otimes a)\cdot m\,z^{-j-1}.
 $$
Note also that  $M(y)$ is free over $\C[y]$ with basis 
\begin{equation}\label{basisy}
\{(\tau^{-j_1}\otimes a)^{i_1} \cdots (\tau^{-j_r}\otimes a)^{i_r}(1\otimes 1)\mid j_1>\cdots >j_r>0\}.
\end{equation}
  
Recall from Section \ref{8} the free field realization $\Psi:W^k_{\min}(\g)\to \mathcal V^k=V^1(\C a)\otimes V^{\a_k}(\g^\natural)\otimes F(\g_{1/2})$. If $\nu\in P^+_k$ is not extremal, recall that we denoted by  $L^\natural(\nu)$ the integrable $V^{\a_k}(\g^\natural)$--module of highest weight $\nu$. We also let $v_\nu$ be a highest weight vector of $L^\natural(\nu)$. Then
$$
M(y)\otimes L^\natural(\nu)\otimes F(\g_{1/2})
$$
is a $\mathcal V^k$--module, hence, by means of $\Psi$, a $W^k_{\min}(\g)$--module. Set
$$
N(y,\nu)=\Psi(W^k_{\min}(\g))\cdot (1\otimes \C[y]\otimes v_\nu\otimes\vac)\subset M(y)\otimes L^\natural(\nu)\otimes F(\g_{1/2}).
$$
Since $M(y)\otimes L^\natural(\nu)\otimes F(\g_{1/2})$ is free as a $\C[y]$--module, $N(y,\nu)$ is also free. 
If $\mu\in\C$, set also
$$
N(\mu,\nu)=\left(\C[y]/(y-\mu)\right)\otimes_{\C[y]} N(y,\nu).
$$

By construction $N(\mu,\nu)$ is clearly a highest weight module for $W^k_{\min}(\g)$. As shown in Section \ref{ur}, its highest weight  is $(\nu,\ell_0)$ with
$$
\ell_0=\tfrac{1}{2}\mu^2-s_k\mu+\frac{(\nu|\nu+2\rho^\natural)}{2(k+h^\vee)}.
$$
Since we are looking for unitary representations, we will always assume that $\ell_0\in \R$.
\begin{lemma}\label{basisfree}\ 
 If $\ell_0>A(k,\nu)$ then $N(\mu,\nu)$ is an irreducible $W^k_{\min}(\g)$--module.
\end{lemma}
\begin{proof} Choose  $h\in\C$  such that  $\ell_0=\ell(h)$. By Theorem \ref{chW},  $H_0(\overline M(\widehat \nu_h))$ is  an irreducible $W^k_{\min}(\g)$-module, hence 
 there is an onto map $N(\mu,\nu)\to H_0(\overline M(\widehat \nu_h))=L^W(\ell_0,\nu)$. If $\ell_0\gg 0$, by the proof of Proposition \ref{sufficient}, $N(\mu,\nu)=L^W(\ell_0,\nu)$. Observe that, since $\ell_0>A(k,\nu)$, by Lemma \ref{sign2}, relations \eqref{gi} hold for our chosen $h$. It follows from \eqref{chfW} that
 \begin{equation}\label{charNmu}
ch\,N(\mu,\nu)=\sum_{w\in \widehat W^\natural}det(w)ch\,M^W(w.\widehat\nu_h)\text{ for $\ell_0\gg 0$}.
\end{equation}
By \eqref{pgen},   the highest weight of $M^W(w.\widehat\nu_h)$ is
$(\nu(w,h),\ell_0(w,h))$ where
$$
\nu(w,h)=(w.\widehat\nu_h)_{|\h^\natural},\ \ell_0(w,h)=\frac{\Vert w(\widehat\nu_h+\widehat\rho)\Vert^2-\Vert\widehat\rho\Vert^2}{2(k+h^\vee)}-(w.\widehat\nu_h)(x+d).
$$
Since $w\in\widehat W^\natural$, $(w.\widehat\nu_h)(x)=h$ and $(w.\widehat\nu_h)(d)$ as well as $\nu(w,h)$ do not depend on $h$. 
We can therefore write
\begin{align*}
\ell_0(w,h)&=\frac{\Vert w(\widehat\nu_h+\widehat\rho)\Vert^2-\Vert\widehat\rho\Vert^2}{2(k+h^\vee)}
-(w.\widehat\nu_0)(d)-h\\&=\frac{\Vert (\widehat\nu_0+\widehat\rho)\Vert^2-\Vert\widehat\rho\Vert^2}{2(k+h^\vee)}
-(w.\widehat\nu_0)(d+x)+\frac{\Vert (\widehat\nu_h+\widehat\rho)\Vert^2-\Vert\widehat\nu_0+\widehat\rho\Vert^2}{2(k+h^\vee)}-h
\\&=\frac{\Vert w (\widehat\nu_0+\widehat\rho)\Vert^2-\Vert\widehat\rho\Vert^2}{2(k+h^\vee)}
-(w.\widehat\nu_0)(d+x)+\frac{2h^2+(h^\vee-1)h}{2(k+h^\vee)}-h\\
&=\ell_0(w,0)+\frac{2h^2+(h^\vee-1)h}{2(k+h^\vee)}-h.
\end{align*}
It follows that
$$
ch\,M^W(w.\widehat\nu_h)=ch\,M^W(w.\widehat\nu_0)e^{(0,\tfrac{2h^2+(h^\vee-1)h}{2(k+h^\vee)}-h)}
$$
and
\begin{equation}\label{noth}
\sum_{w\in \widehat W^\natural}det(w)ch\,M^W(w.\widehat\nu_h)=\left(\sum_{w\in \widehat W^\natural}det(w)ch\,M^W(w.\widehat\nu_0)\right)e^{(0,\tfrac{2h^2+(h^\vee-1)h}{2(k+h^\vee)}-h)}.
\end{equation}
In particular, if  $\ell_0\gg 0$, then
$$
ch\,N(\mu,\nu)=\left(\sum_{w\in \widehat W^\natural}det(w)ch\,M^W(w.\widehat\nu_0)\right)e^{(0,\tfrac{2h^2+(h^\vee-1)h}{2(k+h^\vee)}-h)}.
$$
Since $N(y,\nu)$ is a free $\C[y]$--module, the dimensions  of the weight spaces of $N(\mu,\nu)$ do not depend on $\mu$. By \eqref{noth}, the coefficents of both sides of \eqref{charNmu} do not depend on $\mu$. It follows that \eqref{charNmu} holds for all $\mu$. In particular, if $\ell_0>A(k,\nu)$, by Theorem \ref{chW},  
$$
ch\, N(\mu,\nu)=ch\,H_0(\overline M(\widehat\nu_h)),
$$
hence $N(\mu,\nu)\simeq H_0(\overline M(\widehat\nu_h))$ is irreducible.
\end{proof}

The lowest energy space of $N(\mu,\nu)$ is $1\otimes1\otimes V^\natural(\nu)\otimes\vac$ with $L_0$ acting by multiplication by $\ell_0$.  This space admits a $\omega$--invariant Hermitian form hence there exists a $\phi$--invariant Hermitian form $H(\cdot,\cdot)$ on  $N(\mu,\nu)$. 
 
 If $\widehat\zeta(y)\in Hom_{\C[y]}(\C[y]\otimes\ha,\C[y])$ is a weight of $N(y,\nu)$, fix a basis  $\mathcal B_{\widehat\zeta(y)}$ of $N(y,\nu)_{\widehat\zeta(y)}$. Set $\widehat\zeta=\widehat\zeta(\mu)$. Then $1\otimes \mathcal B_{\widehat\zeta(y)}$ gives a basis $\mathcal B_{\widehat\zeta}$ of $N(\mu,\nu)_{\widehat\zeta}=\left(\C[y]/(y-\mu)\right)\otimes_{\C[y]}N(y,\nu)_{\widehat\zeta(y)}$. Let 
$det_{\widehat\zeta}(\ell_0)$ be the determinant of the matrix in this basis of the Hermitian form $H(\cdot,\cdot)$ restricted to $N(\mu,\nu)_{\widehat\zeta}$. Note that $det_{\widehat\zeta}(\ell_0)$ is a polynomial in  $\ell_0$.

%{lemma}\label{124} If     $\ell_0=\mu^2-s_k\mu+\frac{(\nu|\nu+2\rho^\natural)}{2(k+h^\vee)}>A(k,\nu)$,
%then $\det_{\widehat\zeta}(\ell_0)\ne 0$ for any weight $\widehat\zeta$ of the $W^k_{\min}(\g)$--module $N(\mu,\nu)$. 
%\end{lemma}
%\begin{proof} By Lemma \ref{basisfree}, $N(\mu,\nu)$ is irreducible.
%It follows that $det_{\widehat\zeta}(\ell_0)\ne0$.
%\end{proof}

%\begin{lemma}\label{sign} Let $k$ be in the unitarity range. Then $h_{m,\gamma }(k,0) \leq 0$ and $h_{n,\epsilon m} (k,0)\leq 0$.
%\end{lemma}
%%\begin{proof} First consider $h_{n,\epsilon m} (k,\nu)$. Since $k+h^\vee<0$, it suffices to prove that $$ (\epsilon m(k+h^\vee)-n)^2 - (k+1)^2\geq 0, $$
%or, since $k+1<0$,  $\epsilon m(k+h^\vee)-n \le k+1$, which is  clearly verified if $h^\vee\leq 0$. In the  remaining case $\g=spo(2|3)$ the inequality is 
%$2m(k+\tfrac{1}{2})< k+1=k+\tfrac{1}{2}+\tfrac{1}{2}$, which holds since $k+\tfrac{1}{2}<0$.

%Consider now $h_{m,\gamma }(k,0)$, and as above we have to show that 
%\begin{equation}\label{qwe}(2(\rho^{\natural}|\gamma)+2m(k+h^\vee))^2
%  -(k+1)^2  \ge 0.
%\end{equation}
%A direct check on Table 4 shows that 
%\begin{equation}\label{e1}2\max_{\gamma\in\D'}(\rho^{\natural}|\gamma)+h^\vee=1.\end{equation}
%Relation  \eqref{e1} implies  
%$$2(\rho^{\natural}|\gamma)+2m(k+h^\vee)\leq 2\max_{\gamma\in\D'}(\rho^{\natural}|\gamma)+h^\vee +k
%\le k+1,$$
 %hence inequality \eqref{qwe} holds.

%\end{proof}

\begin{proof}[End of proof of Theorem \ref{u??} and Corollary \ref{u?} ]
%Recall  the definition \eqref{Aknu}  of $A(k,\nu)$.
%. and  \eqref{bknu} of $B(k,\nu)$. 
%A direct calculation shows that 
%\begin{equation}\label{bma} B(k,\nu)-A(k,\nu)=-\tfrac{(1+k-2(\nu|\xi))^2}{4(k+h^\vee)}\ge 0.\end{equation}
 We may assume that the level is not collapsing, so that 
$M_i(k)+\chi_i\in\mathbb Z_+$ by Remark \ref{7.4}. Then, by Proposition \ref{sufficient}, the Hermitian form on $L^W(\nu,\ell_0)$ is positive definite for $\ell_0\gg 0$. 
By Lemma \ref{basisfree}, $N(\mu,\nu)=L^W(\nu,\ell_0)$ if $\ell_0=\tfrac{1}{2}\mu^2-s_k\mu+\frac{(\nu|\nu+2\rho^\natural)}{2(k+h^\vee)}>A(k,\nu)$, hence  $det_{\widehat\zeta}(\ell_0)\ne0$ for all  weights  $\widehat\zeta$  of $N(\mu,\nu)$. It follows that  the Hermitian form is positive definite for $\ell_0>A(k,\nu)$, hence positive semidefinite for $\ell_0=A(k,\nu)$.

Corollary \ref{u?} follows from  Proposition \ref{810} and Theorem \ref{u??} in the case $\nu=0$,  since $A(k,0)=0$, and Remark 
\ref{7.4}.
\end{proof}
\section{Explicit  necessary conditions and  sufficient conditions of 
unitarity}\label{summary}
Looking for the pairs $(\nu,\ell_0),\ \nu \in \widehat{P}^+_k,\ \ell_0\in \mathbb R,$ such that $L^W(\nu,\ell_0)$ is a unitary
$W^k_{\min}(\g)$--module for $k$ in the unitarity range, 
we rewrite for each case (excluding the trivial case (1)) the conditions  in terms of the parameters $M_i=M_i(k)$ from Table 2.  
Namely, we provide the necessary and sufficient conditions of unitarity of $L^W(\nu,\ell_0)$ for a non-extremal weight $\nu$, given by Theorem
\ref{u??},  and the necessary condition of unitarity for an extremal weight $\nu$, given by Proposition \ref{boundary}.  We also provide  explicit expressions for the cocycle $\a_k$ and the central charge $c$ of $L$. Recall 
the invariant bilinear form $(.| .)^\natural_i$  on $\g^\natural_i$, introduced in Section \ref{7}.

\subsection{$psl(2|2)$}\label{1111} In this case $\g^\natural=sl(2)$, $M_1\in\nat$ and $\a_k=(M_1-1)(\,.\,|\,.\,)^\natural_1$. If   $\nu=r\theta_1/2$, with $r\in \mathbb Z_{\ge 0}$ (i.e. 
$\nu$ is dominant integral), and $r\le M_1-1$, 
then the necessary and sufficient condition for unitarity is
$$
\ell_0\ge \frac{r}{2}.
$$
If $r=M_1$, then then necessary condition is $\ell_0=M_1/2$.\par\noindent
The central charge is $c=-6(k+1)=6M_1$.

\subsection{$spo(2|3)$}\label{spo23} In this case $\g^\natural=sl(2)$, $M_1\in\nat$ and $\a_k=(M_1-2)(\,.\,|\,.\,)^\natural_1$. If   $\nu=r\theta_1/2=r\a/2$, with $r\in \mathbb Z_{\ge 0}, r\le M_1-2$, 
then the necessary and  sufficient condition for unitarity is
$$
\ell_0\ge \frac{r}{4}.
$$
If $M_1-1\leq r \leq M_1$, then then necessary condition is $\ell_0=r/4$.
\par\noindent
The central charge is $c=-6 k-\frac{7}{2}=\tfrac{3}{2}M_1-\tfrac{1}{2}$.

\subsection{$spo(2|m)$, $m>4$}In this case $\g^\natural=so(m)$, $M_1\in\nat$ and $\a_k=(M_1-1)(\,.\,|\,.\,)^\natural_1$. If $\nu$ is dominant integral, $\nu(\theta_1^\vee)\le M_1-1$, 
then the necessary and  sufficient condition for unitarity is
\begin{equation}\label{1a}
\ell_0\ge\frac{(\nu|\nu+2\rho^\natural)^\natural}{2(M_1+m-3)}+\frac{r (M_1-r-1)}{2 (m+M_1-3)}=-\frac{(\nu|\nu+2\rho^\natural)^\natural-r (2k+r+2)}{2 (2k-m+4)},
\end{equation}
where $r=(\omega_1|\nu)^\natural$, and  $\omega_1$ is the highest weight of the standard representation of 
$so(m)$.\par\noindent
If $\nu(\theta_1^\vee)=M_1$, the necessary condition is that equality must hold in \eqref{1a}.
\par\noindent
The central charge is $c=\frac{M_1 \left(m^2+6 M_1-10\right)}{2 (m+M_1-3)}=-\frac{(2 k+1) \left(12 k-m^2+16\right)}{4 k-2 m+8}$.

\subsection{$D(2,1;\frac{m}{n})$, $m,n\in\mathbb N$,  $m,n$ {\rm coprime}}

In this case $\g^\natural=\g^\natural_1\oplus \g^\natural_2$ with $\g^\natural_i\simeq sl(2)$, 
%$(M_1,M_2)\in (m,n) \nat$ 
and
$$
\a_k(b,c)=(M_i(k)-1)(b|c)^\natural_i\ \text{ if $b,c\in \g^\natural_i$.}$$
 If  $\nu=\tfrac{r_1}{2}\theta_1+\tfrac{r_2}{2}
 \theta_2$ is dominant integral with $r_i\le M_i(k)-1$,
then the necessary and  sufficient condition for unitarity is
\begin{equation}\label{1b}
\ell_0\ge\frac{2 (M_1+1) r_2+2 (M_2+1) r_1+(r_1-r_2)^2}{4
   (M_1+M_2+2)}=\frac{2 (a+1) k (a r_2+r_1)-a
   (r_1-r_2)^2}{4 (a+1)^2 k}
\end{equation}
If $r_i=M_i$ for some $i$, then the necessary condition is that equality must hold in \eqref{1b}.
%In the special case $r_1=r_2=r$ the necessary condition becomes
%$$
%\ell_0\ge r/2.
%$$
The central charge is $c=6\frac{(M_1+1)(M_2+1)}{M_1+M_2+2}-3=-3(1+2k)$.

\subsection{$F(4)$}

In this case $\g^\natural=so(7)$, $M_1\in\nat$ and 
$\a_k=(M_1-1)(\,.\,|\,.\,)^\natural.$
 If   $\nu(\theta_1^\vee)\le M_1-1$, 
then the necessary and  sufficient condition for unitarity is 
\begin{align}\label{1c}
\ell_0&\ge 
   \frac{r_1 (M_1+7)+r_2 (M_1+4)+r_3
   (M_1+1)+r_1^2+r_2^2+r_3^2-r_1r_2-r_1r_3-r_2r_3}{3 (M_1+4)}\\
  &= \frac{r_1 (6-\tfrac{3}{2}k)+r_2 (3-\tfrac{3}{2}k)+r_3
   (-\tfrac{3}{2}k)+r_1^2+r_2^2+r_3^2-r_1r_2-r_1r_3-r_2r_3}{3 (3-\tfrac{3}{2}k)},\notag
\end{align}
where we write $\nu=r_1\epsilon_1+r_2\epsilon_2+r_3\epsilon_3$ with $\epsilon_i$ as in Table 1.
If   $\nu(\theta_1^\vee)=M_1$, then  the necessary condition is that equality must hold in \eqref{1c}.\par\noindent
The central charge is $c=\frac{2 M_1 (2 M_1+11)}{M_1+4}=-\frac{2 (k-3) (3 k+2)}{k-2}$.

\subsection{$G(3)$}\label{fff}

In this case $\g^\natural=G_2$, $M_1\in\nat$ and 
$\a_k=(M_1-1)(\,.\,|\,.\,)^\natural.$
 If  $\nu(\theta_1^\vee)\le M_1-1$, 
then the necessary and  sufficient condition for unitarity is
\begin{equation}\label{1d}
\ell_0\ge\frac{r_1 (3 M_1+1)+r_2 (3 M_1+7)+3(r_1-r_2)^2}{12
   (M_1+3)}=\frac{r_1 (-2-4k)+r_2 (4-4k)+3(r_1-r_2)^2}{8
   (3-2k)},
\end{equation}
where we write $\nu=r_1\epsilon_1+r_2\epsilon_2$ with $\epsilon_i$ as in Table 1.
If   $\nu(\theta_1^\vee)=M_1$, then  the necessary condition is that equality must hold in \eqref{1d}.\par\noindent
The central charge is $c=\frac{M_1 (9 M_1+31)}{2 (M_1+3)}=\frac{-24 k^2+26 k+33}{4 k-6}$.

\section{Unitarity for extremal modules over the $N=3,\, N=4$ and $\text{ big }N=4$ superconformal algebras}
A module $L^W(\nu,\ell_0)$ for $W^k_{\min}(\g)$ is called {\it extremal} if the weight $\nu$ is extremal (see Definition \ref{extr}). In this section we give a partial solution of Conjecture 2 for some $\g$. Namely,   $\g$ will be either  $spo(2|3)$, or $psl(2|2)$, or $D(2,1;a)$, so that $W^k_{\min}(\g)$ is  related to the   $N=3,\,N=4$ and  big $N=4$ superconformal algebra, respectively. Recall from \cite[Section 8]{KW1} that in these cases, up to adding a suitable 
number of bosons and fermions, it is always possible to make the $\l$--brackets  between the generating fields linear, hence the span of their
Fourier coefficients gets endowed with a Lie superalgebra structure, called the $N=3,\, N=4$ and big $N=4$ superconformal algebra respectively. \par
Recall that, by Proposition \ref{boundary}, for each extremal weight $\nu$ there is at most one $\ell_0$ for which the extremal module $L^W(\nu,\ell_0)$ is unitary, hence for each extremal $\nu$ it suffices to construct one such unitary module.\par
\subsection{$\g=spo(2|3)$} Consider $W^k_{\min}(spo(2|3))$ and the  Lie  conformal superalgebra  $R=(\C[\partial]\otimes \aa)\oplus \C K$, where $\aa$ is an 8--dimensional superspace with basis $\tilde L, \tilde G^\pm,\tilde G^0, J^\pm, J^0,\Phi, $ where $\tilde L, J^\pm, J^0$ are even and $\tilde G^\pm,\tilde G^0,\Phi$ are odd, and the following
  $\l$--brackets 
\begin{align*}
  & [{J^0}{}_{\lambda}\tilde{G}^0] = -2\lambda
       \Phi \, , \, [{J^+}{}_{\lambda}\tilde{G}^-]
       =-2\tilde{G}^0 +2\lambda \Phi \, , \,
 [{J^-}{}_{\lambda}\tilde{G}^+] =
        \tilde{G}^0 +\lambda \Phi \, , \,
        [{\tilde{G}^{\pm}}{}_{\lambda}\tilde{G}^{\pm}]=0\, , \\
 &[{\tilde{G}^+}{}_{\lambda} \tilde{G}^-] = \tilde{L} +\tfrac{1}{4}
        (\partial +2\lambda)J^0 -
       \lambda^2K \, ,
        [{\tilde{G}^+}{}_{\lambda}\tilde{G}^0] =
        \tfrac{1}{4}(\partial +2\lambda)J^+ \, ,
         [{\tilde{G}^0}{}_{\lambda}\tilde{G}^0]=\tilde{L}-\lambda^2 K \, , \\
&[{\tilde{G}^-}{}_{\lambda} \tilde{G}^0] =-\tfrac{1}{2}
        (\partial +2\lambda) J^- \, , \,
        [{\tilde{G}^+}{}_{\lambda} \Phi ] = \tfrac{1}{4}J^+ \, , \,
        [{\tilde{G}^-}{}_{\lambda}\Phi]  =\tfrac{1}{2} J^- \, , \,
        [{\tilde{G}^0}{}_{\lambda}\Phi] =-\tfrac{1}{4} J^0 \, \\    
        &[\Phi_{\lambda}\Phi]=-K, [J^+_\l J^-]=J^0-4\l K, , [J^0_\l J^\pm]=\pm2J^\pm, [J^0_\l J^0]=-8\l K,\\
        &[ \tilde{L}_\l \tilde{L}]=\partial  \tilde{L}+2\l \tilde{L}-\tfrac{\l^3}{2}K.
\end{align*}
Furthermore $\tilde G^\pm,\tilde G^0, J^\pm, J^0,\Phi$ are primary for $\tilde L$ of  conformal weight  $\tfrac{3}{2},\tfrac{3}{2},1,1,\tfrac{1}{2}$, respectively.

The $N=3$ superconformal algebra $\mathcal W_{N=3}^k$ is $V(R)/(K-(k+\tfrac{1}{2})\vac)$,
where $V(R)$ is the universal enveloping vertex algebra of $R$. Let $F_\Phi$ be the fermionic vertex algebra generated by an odd element $\Phi$, with $\l$--braket $[{\Phi}_{\l}\Phi]=-(k+\tfrac{1}{2})\vac$.
Then there is a conformal vertex algebra embedding
$$\mathcal W_{N=3}^k\hookrightarrow W^k_{\min}(spo(2|3))\otimes F_\Phi $$
given by (cf \cite[\S 8.5]{KW1})
\begin{eqnarray*}
  &\tilde{L} \mapsto L-\tfrac{1}{2k+1}:\partial \Phi \Phi :\, , \,
  \tilde{G}^+ \mapsto\tfrac{\sqrt{-1}}{\sqrt{k+1/2}}G^+ -\tfrac{1}{4k+2}:J^+\Phi : \, , \,\\
 & \tilde{G}^- \mapsto \tfrac{-\sqrt{-1}}{\sqrt{k+1/2}}G^- -\tfrac{1}{2k+1}:J^- \Phi : \, , \,
  \tilde{G}^0 \mapsto\tfrac{-\sqrt{-1}}{\sqrt{k+1/2}}G^0 +\tfrac{1}{4k+2}:J^0 \Phi : \, .\\
  &\Phi\mapsto \Phi,\, J^\pm\mapsto J^\pm,\,J^0\mapsto J^0.
\end{eqnarray*}
%
%Then $\tilde{L}$ is a Virasoro field with central charge
%
%\begin{displaymath}
%  \tilde{c}  =-6k-3 \, ,
%\end{displaymath}
 Extend the conjugate linear involution $\phi$ to $W^k_{\min}(spo(2|3))\otimes F_\Phi$ setting $\phi(\Phi)=-\Phi$. Recall from \cite{KMP} that   the unique $\phi$--invariant Hermitian form on $F_\Phi$  is positive definite.
Also recall that the tensor product of  invariant Hermitian forms is still invariant; in particular if we prove that $L^W(\nu,\ell_0)\otimes F_\Phi$ is unitary for $\mathcal W_{N=3}^k$, then
$L^W(\nu,\ell_0)$  is a unitary  $W^k_{\min}(spo(2|3))$--module. 
Recall that, for  $a,b\in V(R)$, the modes of $a,b$ have a Lie superalgebra structure given by 
$$[a_r,b_s]=\sum_{j\in \mathbb Z_+}\binom{\D_a+r-1}{j} (a_{(j)}b)_{r+s}.$$
 Observe that the span $\mathcal L$ of $\tilde L_n, \tilde G^\pm_m,\tilde G^0_m, J_n^\pm, J^0_n,\Phi_m, K,\,n\in\mathbb Z, m\in \tfrac{1}{2}+\mathbb Z,$ is a Lie superalgebra. If $M$  (resp. $M'$)  are modules 
 for $\mathcal W_{N=3}^k$ (resp. $\mathcal W_{N=3}^{k'}$), then $M\otimes M'$ inherits an action of $\mathcal L$ which makes $M\otimes M'$ a $\mathcal W_{N=3}^{k+k'+\tfrac{1}{2}}$--module. Clearly, if both $M,M'$ are unitary, then 
 $M\otimes M'$ is unitary. The argument used in the next proposition generalizes the one used for the oscillator representation of the  Virasoro algebra in 	\cite[\S 3.4]{KR}.
 \begin{prop} Let $M_1=-4k-2\in \mathbb N$.  Then the extremal $W^k_{\min}(spo(2|3))$--modules $L^W(\tfrac{M_1-1}{2}\a,\tfrac{M_1-1}{4})$, $L^W(\tfrac{M_1}{2}\a,\tfrac{M_1}{4})$ are both unitary, where $\a$ is the simple root of $\g^\natural=sl_2$.
 \end{prop}
 \begin{proof} To make the argument  more transparent we make explicit the dependence on $k$, so we write $L(k,\nu,\ell_0)$ for the $W^k_{\min}(spo(2|3))$--module $L^W(\nu,\ell_0)$. Recall that $\nu=r\a/2$.
\par
 We proceed by induction on $M_1$. The base case $M_1=1$ corresponds to the collapsing level $k=-3/4$, when $W_{-3/4}^{\min}(spo(2|3))=V_1(sl_2)$. Recall that  $V_1(sl_2)$ has only two irreducible modules $N_1$ and $N_2$, which are both unitary and have highest weights $\nu=0$ and $\nu=\a/2$ respectively.
Recall from \S \ref{spo23} that if $M_1-1\leq r \leq M_1$, then the necessary condition for unitarity is $\ell_0=M_1/4$. Hence $N_1$ and $N_2$ are  $L(-3/4,0,0)$ and $L(-3/4,\a/2,1/4)$. Set $k_1=-\tfrac{M_1+1}{4}$. Assume by induction that 
 $L(k_1,\tfrac{M_1-2}{2}\a,\tfrac{M_1-2}{4})$ and  $L(k_1,\tfrac{M_1-1}{2}\a,\tfrac{M_1-1}{4})$ are unitary. Then  $M=L(k_1,\tfrac{M_1-2}{2}\a,\tfrac{M_1-2}{4})\otimes F_\Phi$ is unitary for 
 $\mathcal W_{N=3}^{k_1}$ and $M'=L(-3/4,\a/2,1/4)\otimes F_\Phi$ is unitary for $\mathcal W_{N=3}^{-3/4}$. Therefore  $M\otimes M'$ is unitary for 
 $\mathcal W_{N=3}^{k_2},\,k_2=k_1-\tfrac{3}{4}+\tfrac{1}{2}=-\tfrac{M_1+1}{4}-\tfrac{3}{4}+\tfrac{1}{2}=-\tfrac{M_1}{4}-\tfrac{1}{2}=k$. In particular, the
 $W^{k}_{\min}(spo(2|3))$--module generated by $v_{\tfrac{M_1-2}{2}\a,\tfrac{M_1-2}{4}}\otimes \vac\otimes v_{\tfrac{\a}{2},\tfrac{1}{4}}\otimes \vac$ is unitary, and its weight is 
 $(\tfrac{M_1-1}{2}\a,\tfrac{M_1-1}{4})$, as required. 
 %Similarly, tensoring $L(k_1,\tfrac{M_1-1}{2}\a,\tfrac{M_1-1}{4})\otimes F_\Phi$ with $L(-3/4,0,0)\otimes F_\Phi$ and proceding as above, we prove the unitarity of 
 %$L(k,\tfrac{M_1}{2}\a,\tfrac{M_1}{4})$.
 Repeating this argument with  $L(k_1,\tfrac{M_1-1}{2}\a,\tfrac{M_1-1}{4})\otimes F_\Phi$ proves the unitarity of 
 $L(k,\tfrac{M_1}{2}\a,\tfrac{M_1}{4})$.
 \end{proof}
 \subsection{$\g=psl(2|2)$}  We choose  strong generators $J^0,J^\pm,G^\pm,\bar G^\pm,L$ for $W^{k}_{\min}(psl(2|2))$ as in \cite[\S 8.4]{KW1}. We can choose the generators so that, if $\phi$ is the almost compact involution corresponding to the real form described in Section \ref{4}, then
\begin{equation}\label{N4phi}
\phi(L)=L,\ \phi(J^+)=-J^-,\   \phi(J^0)=-J^0,\ \phi(G^+)=\bar G^-,\ \phi(G^-)=\bar G^+.
\end{equation}
 The   $\l$--brackets among these generators are linear, hence their Fourier coefficients span the $N=4$ superconformal algebra. 
 It is therefore enough to prove unitarity of the extremal module  $L^W(\theta_1/2,1/2)$ at level $k=-2$, since all the other extremal modules 
 at level $k<-2$ are obtained by iterated tensor product of $L^W(\theta_1/2,1/2)$. \par  
 The unitarity of $L^W(\theta_1/2,1/2)$ is proved by constructing this module as a submodule of a manifestly unitary module.
 This is achieved by using the free field realization  of   $W^{-2}_{\min}(psl(2|2))$ given in  \cite{ET1},  in terms of  four bosonic fields and four fermionic fields, which we now describe. Let  $\mathcal F$ be the vertex algebra  generated by four even fields $a^i, 1\leq i \leq 4$ and four odd fields $b^i, 1\leq i \leq 4$ with $\l$--bracket
$$
[{a_i}_\l a_j]=\d_{ij}\l,\ [{b_i}_\l b_j]=\d_{ij},\ [{a_i}_\l b_j]=0.
$$
There is an homomorphism    $FFR:W^{-2}_{\min}(psl(2|2))\to\mathcal F$ given by 
%with $\l$-brackets $[{a^i}_\l a^j]=\d_{ij}\l, [{b^i}_\l b^j]=\d_{ij}$
\begin{align*}L&\mapsto\tfrac{1}{2}\sum_{i=1}^4 (:a^ia^i:+:Tb^i b^i:)\\
J^+&\mapsto-\tfrac{1}{2}:b^1b^3:-\tfrac{1}{2}\sqrt{-1} :b^1b^4:-\tfrac{1}{2}\sqrt{-1}:b^2b^3:+\tfrac{1}{2}:b^2b^4:\\
J^-&\mapsto\tfrac{1}{2}:b^1b^3:-\tfrac{1}{2}\sqrt{-1} :b^1b^4:-\tfrac{1}{2}\sqrt{-1}:b^2b^3:-\tfrac{1}{2}:b^2b^4:\\
J^0&\mapsto-\sqrt{-1}:b^1b^2:-\sqrt{-1} :b^3b^4:\\
G^+&\mapsto\tfrac{1}{2}:(a^1+\sqrt{-1}a^2)(b^3+\sqrt{-1}b^4):-\tfrac{1}{2} :(a^3+\sqrt{-1}a^4)(b^1+\sqrt{-1}b^2):\\
G^-&\mapsto\tfrac{1}{2}:(a^1+\sqrt{-1}a^2)(b^1-\sqrt{-1}b^2):+\tfrac{1}{2} :(a^3+\sqrt{-1}a^4)(b^3-\sqrt{-1}b^4):\\
\bar G^+&\mapsto\tfrac{1}{2}:(a^1-\sqrt{-1}a^2)(b^1+\sqrt{-1}b^2):+\tfrac{1}{2} :(a^3-\sqrt{-1}a^4)(b^3+\sqrt{-1}b^4):\\
\bar G^-&\mapsto\tfrac{1}{2}:(a^1-\sqrt{-1}a^2)(b^3-\sqrt{-1}b^4):-\tfrac{1}{2} :(a^3-\sqrt{-1}a^4)(b^1-\sqrt{-1}b^2):
\end{align*}
We define a conjugate linear  involution $\psi$  on $\mathcal F$ by
$$a_i\mapsto  -a_i,\ b_i\mapsto  -b_i
$$
so that, according to \cite[\S 5.1,5.2]{KMP}, there is a $\psi$--invariant positive definite Hermitian form $H_{\mathcal F}$ on $\mathcal F$. 
It is  clear from \eqref{N4phi} that $\psi\circ FFR=FFR\circ \phi$. 
Using $FFR$ we can define an action of $W^{-2}_{\min}(psl(2|2))$ on $\mathcal F$. Since $H_{\mathcal F}$ is invariant with respect to the conformal vector $FFR(L)$, it follows that $\mathcal F$ is a unitary $W^{-2}_{\min}(psl(2|2))$--module.
An easy calculation shows that $v=b^1+\sqrt{-1}b^2$ is a singular vector for $W^{-2}_{\min}(psl(2|2))$, thus $v$ generates a unitary highest weight representation $L^W(\nu,\ell_0)$ of $W^{-2}_{\min}(psl(2|2))$. Clearly $FFR(L)_0v=\half v$, while $J^0v=v$, hence $\nu=\half\theta_1$ and $\ell_0=\half$. This proves that the highest weight module corresponding the extremal weight $\nu=\half\theta_1$ is indeed unitary.
 \subsection{$\g=D(2,1; \tfrac{m}{n})$} 
 %The big $N=4$ superconformal algebra is obtained from $W^k_{\min}(\g)$ by tensoring the latter vertex algebra by the vertex algebra strongly generated by four (odd) free fermions 
% $\s^{--},\s^{-+},\s^{+-},\s^{++}$ with non-zero $\l$-brackets $[{\s^{--}}_\l\s^{++}]=[{\s^{-+}}_\l\s^{+-}]=k$; and one (even) free boson $\xi$  with $\l$-bracket $[\xi_\l\xi]=k$. 
In this case we are able to prove unitarity only
 in the very special  case when  either $m=1$ of $n=1$.  
 
If $n=1$,  then the unitarity range is $\{-\tfrac{m}{m+1}N\mid N\in \nat\}$. Take $N=1$ and observe that $W^{-\tfrac{m}{m+1}}_{\min}(D(2,1; m))$ collapses to $V_{m-1}(sl(2))$. In this case there is only one extremal weight $\nu=\tfrac{m-1}{2}\a_2$, which gives rise to a unitary representation since it is  integrable. The case $m=1$ is dealt with in a similar way, switching the roles of $\a_2,\a_3$.

\section{Characters of the irreducible unitary $W^k_{\min}(\g)$--modules}\label{14}

Recall  that, for $\L\in\ha^*$, we denoted by  $M^W(\L)$ the Verma module $M^W(\nu,\ell)$, where $(\nu,\ell)$ is given by \eqref{pgen}. It follows from  \cite[(6.11)]{KW1}, that 
\begin{align}
  \label{eq:6.111}
 &ch\, M^W(\L) =e^\nu q^{\ell}F^{NS}(q),
 \end{align}
 where  $q=e^{(0,1)}$ and
 \begin{align}
  \label{dden}
      F^{NS}(q)= \prod^{\infty}_{n=1}\frac{\prod_{\alpha \in \Delta_{1/2}}\!
  (1+q^{n-\frac{1}{2}}e^{-\alpha })}{ (1-q^n)^{rank\g^\natural+1}\prod_{\alpha \in \Delta^\natural_{+}}\,( 
        (1- q^{n-1} e^{-\alpha })
        (1-q^n e^{\alpha})))}.
\end{align}
In particular,
\begin{align}
  \label{eq:6.11}
 &ch\, M^W(\widehat\nu_h) =e^\nu q^{\ell(h)}F^{NS}(q),
 \end{align}
 where $\ell(h)$ is given by \eqref{p}.
 
The characters of unitary $W^k_{\min}(\g)$-modules $L^W(\nu,\ell_0)$ are computed by applying the quantum Hamiltonian reduction to the irreducible highest weight $\ga$-modules 
$L(\widehat \nu_h)$, where $\nu\in P^+_k$ and $\ell_0=\ell(h)$,  and using the argument in the proof of Theorem \ref{chW}, which is based on Remark \ref{nonzero}. There are two cases to consider in computation of their characters. First, if the weight $\widehat \nu_h$ is typical, i.e. 
conditions \eqref{gi} hold, then $ch\,L(\widehat \nu_h)$ is given by the R.H.S. of  \eqref{chf}, by Proposition \ref{irr}.\par
The second case occurs when the weight $\widehat \nu_h$ satisfies the condition 
$$(\widehat \nu_h+\widehat \rho|\a)=0\text{ for all $\a\in\Pi_{\bar 1}$,}$$
where $\Pi_{\bar 1}$ denotes the set of simple isotropic roots of $\g$. Then the weight $\widehat\nu_h$ is maximally atypical, and $L(\widehat \nu_h)$ is integrable, hence the following formula is a special case of \cite[Formula (14)]{GK2} if $\g\ne D(2,1;\frac{m}{n})$
and of \cite[Section 6.1]{GK2} if $\g= D(2,1;\frac{m}{n})$ and $\nu=0$:
\begin{equation}\label{KWC}
e^{\widehat \rho}\, R\, ch\, L(\widehat \nu_h)=\sum_{w\in \Wa^\natural} det(w) w \frac{ e^{\widehat \nu_h+\widehat \rho}}{\prod_{\beta\in \Pi_{\bar 1}}(1+e^{-\beta})},
\end{equation}
where $R$ equals the character of the Verma module $M(0)$ over $\widehat\g$ with highest weight $0$.
%In the proof of the next theorem we verify formula \eqref{KWC}  when $\nu=0$.
\begin{theorem}\label{characters}
Let $k$ be in the unitary range 
and let $\nu\in P^+_k$. 
%be a non-extremal weight. 
Let $L^W(\nu,\ell_0)$ be a unitary irreducible $W^k_{\min}(\g)$--module. Choose $h$ so that $\ell(h)=\ell_0$ and let, as before,
$$
\widehat\nu_h=k\L_0+\nu+h\theta.
$$
(i) If $\ell_0>A(k,\nu)$, then 
\begin{equation}\label{ff1}ch L^W(\nu,\ell_0)=\sum_{w\in\widehat W^\natural}det(w)ch M^W(w.\widehat\nu_h).\end{equation}
(ii) If  $\ell_0=A(k,\nu)$, and $\nu=0$ if $\g=D(2,1;\frac{m}{n})$, then
\begin{equation}\label{ff2}
ch L^W(\nu,\ell_0)=\sum\limits_{w\in\widehat W^\natural}\sum_{\gamma\in \ZZ_+\Pi_{\bar 1}}(-1)^{\gamma}det(w)ch M^W(w.(\widehat\nu_h-\gamma)),
\end{equation}
 where $\Pi_{\bar 1}=\{\gamma_1,\gamma_2,\ldots\}$ is the set of isotropic simple roots for $\g$, and for $\gamma=n_1\gamma_1+\cdots$,  we write $(-1)^\gamma=(-1)^{n_1+\cdots}$.
\end{theorem}
%\begin{remark} \label{nu=0}Note that the condition $(\xi|\nu)=0$ is equivalent to $\nu=0$. The non-trivial implication follows from a direct computation using the fact that $\nu\in P^+_k$;  the $\h^\natural$-weights $\xi$ are  the restrictions to $\h^\natural$ of the maximal odd roots listed in Table 3.
%$\end{remark}
\begin{proof}
%[Proof of Theorem \ref{characters}]
Formula \eqref{ff1} follows from \eqref{chfW}.
%Observe that, using \eqref{MIK} and \eqref{chii},
%$$(\widehat\nu_h|\eta_i^\vee)=\tfrac{2}{u_i}k-\nu(\theta_i^\vee)=M_i(k)-\chi_i-\nu(\theta_i^\vee)\in \mathbb N,
%$$
%since $\nu\in P^+_k$ and $\chi_i<0$.
Formula  \eqref{ff2} follows from \eqref{KWC} by applying quantum Hamiltonian reduction to the $\ga$--module $L(\widehat\nu_h)$. 
%Observe that this module is integrable with respect to $\widehat\g^\natural$ and maximally atypical with the maximal subset of pairwise orthogonal isotropic roots consists of the set of isotropic simple roots. Hence we can apply Kac-Wakimoto character formula proved, in most of the cases, in \cite{GK2}.
In order to use \eqref{KWC}, write explicitly the relation $\ell_0=\ell(h)=A(k,\nu)$. We have 
$$\frac{(k\L_0+h\theta+\nu|k\L_0+h\theta+\nu+2h^\vee\L_0+ 2\rho)}{2(k+h^\vee)}-h= \frac{(\nu|\nu+2\rho^\natural)}{2(k+h^\vee)}+\frac{(\xi|\nu)}{k+h^\vee}((\xi|\nu)-k-1),$$
or
%$$\frac{2h^2+2(h^\vee-1)h+(\nu|\nu+2\rho^\natural)}{2(k+h^\vee)}-h= \frac{(\nu|\nu+2\rho^\natural)}{2(k+h^\vee)}+\frac{(\xi|\nu)}{k+h^\vee}((\xi|\nu)-k-1)$$
%or
%$$h^2+(h^\vee-1)h-(k+h^\vee)h= (\xi|\nu)((\xi|\nu)-k-1)$$
%or
$$h(h-1-k)= (\xi|\nu)((\xi|\nu)-k-1).$$
Hence either $h=(\xi|\nu)$ or $h=1+k-(\xi|\nu)$. We observe that if $\alpha\in \Pi_{\bar{1}}$, then, restricted to $\h^\natural$, 
it coincides with $-\xi$, hence $(\xi|\nu)=-(\alpha|\nu)$,
and also $(\theta|\alpha)=1$. Therefore, for $h=(\xi|\nu)$ we have
$$(\widehat\nu_h+\widehat\rho|\a)=((k+h^\vee)\L_0+(\xi|\nu)\theta+\nu+\rho|\a)=(\xi|\nu)+(\a|\nu)=0.$$
Hence we may apply \eqref{KWC}. Note that $H_0(L(\widehat\nu_h))\ne 0$ since $(\widehat\nu_h|\a_0)<0$, so that we can apply Remark \ref{nonzero}.\end{proof}

\begin{remark}\label{massless} It is still an open problem whether in the case 
$\g=D(2,1;\tfrac{m}{n})$  formula \eqref{KWC} holds for an arbitrary $\nu \in P^+_k$.
\end{remark}

\begin{remark} For the $N=4$ superconformal algebra, formula \eqref{ff1} appears,  in a different form,  in \cite[formula (14)]{ET2}, where it has been derived in a non-rigorous way.  To establish a dictionary to match the two formulas first observe that a parameter $y$ occurs in the formulas of \cite{ET2} corresponding to an extra $U(1)$--symmetry that we do not consider, hence, to compare the formulas, we set $y=1$. Next recall that in this case  $\Wa^\natural$ is of type $A_1^{(1)}$, hence its elements are of the form 
$u_i=\underbrace{s_0s_1\cdots}_{i \text{ factors}}$ or $u'_i=\underbrace{s_1s_0\cdots}_{i \text{ factors}}$ 
(set $u_0=u'_0=Id$). In the notation of \cite{ET2}, the pairs $(a_n,b_n)$ corresponding to the $\a$-series (resp. $\beta$-series)  in formula (12)  of \cite{ET2}  match  exactly the pairs $(\nu,\ell)$ given in \eqref{pgen} for the weight $\L=u_i.\widehat \nu_h$ (resp.   $\L=u'_i.\widehat \nu_h$).
The factor $F^{NS}(\theta,1)$ translates precisely to \eqref{eq:6.11} according to the dictionary
$$e^{\d_1-\d_2}\leftrightarrow e^{\sqrt{-1}\theta}.
$$
The character formula  \eqref{ff2} corresponds to the formula (26) in \cite{ET2} for the character of ``massless'' representations. To show this, we first remark that, if $\gamma\in\ZZ_+\Pi_{\bar 1}$, then
$$M^W(w.(\widehat\nu_h-\gamma))=M^W(\nu,\ell),$$
where $(\nu,\ell)$ is given by \eqref{pgen}. In particular
\begin{align*}\ell&=\frac{(w.(\widehat\nu_h-\gamma)|w.(\widehat\nu_h-\gamma)+2\widehat \rho)}{2(k+h^\vee)}-(w.(\widehat\nu_h-\gamma))(x+d)\\&=\frac{||\widehat\nu_h-\gamma+\widehat \rho||^2-||\widehat \rho||^2}{2(k+h^\vee)}-(w.(\widehat\nu_h-\gamma))(x+d)\\
&=\frac{||\widehat\nu_h+\widehat \rho||^2-||\widehat \rho||^2}{2(k+h^\vee)}-w.(\widehat\nu_h)(x+d)+w(\gamma)(x+d)\\
&=\ell(h)+(\widehat\nu_h+\widehat\rho)(x+d)-w(\widehat\nu_h+\widehat\rho)(x+d)+w(\gamma)(x+d),
\end{align*}
hence, using formula \eqref{eq:6.111},
$$
ch\,M^W(w.(\widehat\nu_h-\gamma))=q^{\ell(h)} F^{NS}(q)e^{(w.\widehat\nu_h)_{\vert\h^\natural}}q^{\left(\widehat \nu_h+\widehat\rho-w(\widehat \nu_h+\widehat\rho)\right)(x+d)}e^{-(w\gamma)_{\vert\h^\natural}}q^{w(\gamma)(x+d)},
$$
and we obtain that
\begin{equation}\label{chMW1}
\sum_{\gamma\in\ZZ_+\Pi_{\bar 1}}(-1)^\gamma ch\,M^W(w.(\widehat\nu_h-\gamma))=
 q^{\ell(h)}F^{NS}(q)\,
\frac{e^{(w.{(\widehat\nu_h)})_{|\h^\natural}}q^{\left(\widehat \nu_h+\widehat\rho-w(\widehat \nu_h+\widehat\rho)\right)(x+d)}}{\prod_{\a\in\Pi_{\bar 1}}(1+e^{-w(\a)_{|\h^\natural}}q^{w(\a)(x+d)})}.
\end{equation}
Since $\theta$ is orthogonal to $(\widehat \h^\natural)^*$ (where $\widehat \h^\natural=\C K+\C d+\h^\natural$), we can apply the formulas of  \cite[Chapter 6]{VB} to
$ \widehat\g^\natural$ and its Weyl group. Since, in our case, $\widehat\nu_h+\widehat\rho=k\L_0+(h-\half)\theta+(r+\half)\eta_1$, $r\in\half\ZZ_+$, we have, for $m\in\ZZ$,
$$
(s_0s_1)^m(\widehat \nu_h+\widehat\rho)=k\L_0+(h-\half)\theta+(r-km+\half)\eta_1-(m (-km+2 r+1))\d.
$$
and, if $\a\in\Pi_{\bar1}$,
$$
(s_0s_1)^m(\a)=\a+m\delta.
$$
Since $s_1=s_{\eta_1}$, it follows that
$$
((s_0s_1)^m.\widehat \nu_h)_{|\h^\natural}=(r-km)\eta_1,\ ((s_1(s_0s_1)^m).\widehat \nu_h)_{|\h^\natural}=-(r-km+1)\eta_1,
$$
\begin{align*}
\widehat \nu_h+\widehat\rho-(s_0s_1)^m(\widehat \nu_h+\widehat\rho)(x+d)&=\widehat \nu_h+\widehat\rho-s_1(s_0s_1)^m(\widehat \nu_h+\widehat\rho)(x+d)\\
&=(m (-km+2 r+1))=-km^2 +(2r+1)m,
\end{align*}
and
\begin{align*}
&(s_0s_1)^m(\a)_{|\h^\natural}=-\half\eta_1,\ s_1(s_0s_1)^m(\a)_{|\h^\natural}=\half\eta_1,\\ 
&(s_0s_1)^m(\a)(x+d)=s_1(s_0s_1)^m(\a)(x+d)=(m+\half).
\end{align*}
Substituting \eqref{chMW1} into \eqref{ff2}, recalling that $M_1(k)=-k-1$, we obtain
\begin{align*}
&ch\, L^W(r\eta_1,\ell_0)\\
&=q^{\ell_0}F^{NS}(q)\sum_{m\in\ZZ}\left(\frac{e^{(r+m(M_1(k)+1))\eta_1}}{(1+e^{\frac{1}{2}\eta_1}q^{m+\frac{1}{2}})^2}-\frac{e^{-(r+m(M_1(k)+1)+1)\eta_1}}{(1+e^{-\frac{1}{2}\eta_1}q^{m+\frac{1}{2}})^2}\right)q^{m^2( M_1(k)+1)+(2r+1)m}
\end{align*}
which, under our dictionary, corresponds to  formula (26) of  \cite{ET2}  in the NS sector.

 For $W^k_{\min}(spo(2|3))$, formula \eqref{ff1} appears (with a non-rigorous proof)  in \cite[formula (4.3)]{M}. Again, in this case  $\Wa^\natural$ is of type $A_1^{(1)}$and  its elements are of the form 
$u_i$ or $u'_i$ 
(notation as above). The pairs $(l_n,h_n)$ displayed in \cite[(4.2.a),(4.2.b)]{M},  corresponding to the $A$-series (resp. $B$-series),   match  exactly the pairs $(\nu,\ell)$ given in \eqref{pgen} for the weight $\L=u_i.\widehat \nu_h$ (resp.   $\L=u'_i.\widehat \nu_h$).
The denominator $F^{NS}(q,z)$ in \cite[(3.15.i)]{M} translates precisely to \eqref{eq:6.11} according to the dictionary
$$e^{\e_1}\leftrightarrow z.
$$
In the massless case, the character formula \eqref{ff2}  corresponds to formula (4.6.1) in \cite{M}, hence Theorem \ref{characters} provides a proof of it,
 since formula \eqref{KWC} holds in this case,
due to \cite[Subsection 12.3]{GK2}.
.\end{remark}

\noindent {\bf Acknowledgements}.  V.K. is partially supported by the Stephen Berenson mathematical exploration  fund
 and 
the Simons collaboration grant. The authors thank Drazen Adamovi\'c,  Thomas Creutzig and Anne Taormina for correspondence. The  authors are  grateful to Maria Gorelik for many very useful and enlightening discussions.

%\section*{Declarations}
%\noindent{\bf Competing Interests.} The authors have no competing interests to declare that are relevant to the content of this article.

%\noindent {\bf Data Availability Statement.} Data sharing not applicable to this article as no datasets were generated or analysed during the current study.

\vskip5pt
    \footnotesize{
\noindent{\bf V.K.}: Department of Mathematics, MIT, 77
Mass. Ave, Cambridge, MA 02139;\newline
{\tt kac@math.mit.edu}
\vskip5pt
\noindent{\bf P.M-F.}: Politecnico di Milano, Polo regionale di Como,
Via Anzani 42, 22100, Como, Italy;\newline {\tt pierluigi.moseneder@polimi.it}
\vskip5pt
\noindent{\bf P.P.}: Dipartimento di Matematica, Sapienza Universit\`a di Roma, P.le A. Moro 2,
00185, Roma, Italy;\newline {\tt papi@mat.uniroma1.it}, Corresponding author
}

   \end{document}